\documentclass[reqno,11pt]{amsart}
\usepackage{amsmath}
\usepackage{amsxtra}
\usepackage{amscd}
\usepackage{amsthm}
\usepackage{amsfonts}
\usepackage{amssymb}
\usepackage{eucal}
\usepackage{scalerel}
\usepackage{verbatim}
\usepackage[matrix,arrow,curve]{xy}

\usepackage{a4wide}

\usepackage{xr-hyper}
\usepackage[
pdftex,
bookmarks=false,
colorlinks=true,
debug=true,
pdfnewwindow=true]{hyperref}

\theoremstyle{plain}
\newtheorem{Thm}[subsection]{Theorem}
\newtheorem{Cor}[subsection]{Corollary}
\newtheorem{Lem}[subsection]{Lemma}
\newtheorem{Prop}[subsection]{Proposition}
\newtheorem{Conj}[subsection]{Conjecture}

\theoremstyle{definition}
\newtheorem{Def}[subsection]{Definition}

\theoremstyle{remark}

\newtheorem{Rem}[subsection]{Remark}

\newtheorem{conj}{Conjecture}

\numberwithin{equation}{section}
\renewcommand{\rm}{\normalshape}

\newif\ifShowLabels
\ShowLabelstrue
\newdimen\theight
\def\TeXref#1{%
    \leavevmode\vadjust{\setbox0=\hbox{{\tt
        \quad\quad  {\small \rm #1}}}%
    \theight=\ht0
    \advance\theight by \lineskip
    \kern -\theight \vbox to
    \theight{\rightline{\rlap{\box0}}%
    \vss}%
    }}%

\ShowLabelsfalse

\newenvironment{thm}[1]%
    { \begin{Thm} \label{T:#1}  \ifShowLabels \TeXref{T:#1} \fi }%
    { \end{Thm} }

\renewcommand{\th}[1]{\begin{thm}{#1} \sl }
\renewcommand{\eth}{\end{thm} }

\newenvironment{lemma}[1]%
    { \begin{Lem} \label{L:#1}  \ifShowLabels \TeXref{L:#1} \fi }%
    { \end{Lem} }
\newcommand{\lem}[1]{\begin{lemma}{#1} \sl}
\newcommand{\elem}{\end{lemma}}

\newenvironment{propos}[1]%
    { \begin{Prop} \label{P:#1}  \ifShowLabels \TeXref{P:#1} \fi }%
    { \end{Prop} }
\newcommand{\prop}[1]{\begin{propos}{#1}\sl }
\newcommand{\eprop}{\end{propos}}

\newenvironment{corol}[1]%
    { \begin{Cor} \label{C:#1}  \ifShowLabels \TeXref{C:#1} \fi }%
    { \end{Cor} }
\newcommand{\cor}[1]{\begin{corol}{#1} \sl }
\newcommand{\ecor}{\end{corol}}

\newenvironment{defeni}[1]%
    { \begin{Def} \label{D:#1}  \ifShowLabels \TeXref{D:#1} \fi }%
    { \end{Def} }
\newcommand{\defe}[1]{\begin{defeni}{#1} \sl }
\newcommand{\edefe}{\end{defeni}}

\newenvironment{remark}[1]%
    { \begin{Rem} \label{R:#1}  \ifShowLabels \TeXref{R:#1} \fi }%
    { \end{Rem} }
\newcommand{\rem}[1]{\begin{remark}{#1}}
\newcommand{\erem}{\end{remark}}

\newenvironment{conjec}[1]%
    { \begin{Conj} \label{Co:#1}  \ifShowLabels \TeXref{Co:#1} \fi }%
    { \end{Conj} }
\renewcommand{\conj}[1]{\begin{conjec}{#1} \sl }
\newcommand{\econj}{\end{conjec}}

\newcommand{\eq}[1]%
    { \ifShowLabels \TeXref{E:#1} \fi
       \begin{equation} \label{E:#1} }
\newcommand{\eeq}{ \end{equation} }

\newcommand{\prf}{ \begin{proof} }
\newcommand{\epr}{ \end{proof} }

\usepackage[dvipsnames]{xcolor}

\newcommand\nc{\newcommand}
\nc{\unl}{\underline}
\nc{\ol}{\overline}
\nc{\on}{\operatorname}

\nc{\BA}{{\mathbb{A}}}
\nc{\BC}{{\mathbb{C}}}
\nc{\BD}{{\mathbb{D}}}
\nc{\BF}{{\mathbb{F}}}
\nc{\BG}{{\mathbb{G}}}
\nc{\BM}{{\mathbb{M}}}
\nc{\BN}{{\mathbb{N}}}
\nc{\BO}{{\mathbb{O}}}
\nc{\BQ}{{\mathbb{Q}}}
\nc{\BP}{{\mathbb{P}}}
\nc{\BR}{{\mathbb{R}}}
\nc{\BZ}{{\mathbb{Z}}}
\nc{\BS}{{\mathbb{S}}}

\nc{\CA}{{\mathcal{A}}} \nc{\CB}{{\mathcal{B}}} \nc{\CalC}{{\mathcal
C}} \nc{\CalD}{{\mathcal D}} \nc{\CE}{{\mathcal{E}}}
\nc{\CF}{{\mathcal{F}}} \nc{\CG}{{\mathcal{G}}}
\nc{\CH}{{\mathcal{H}}} \nc{\CI}{{\mathcal{I}}}
\nc{\CK}{{\mathcal{K}}} \nc{\CL}{{\mathcal{L}}}
\nc{\CM}{{\mathcal{M}}} \nc{\CN}{{\mathcal{N}}}
\nc{\CO}{{\mathcal{O}}} \nc{\CP}{{\mathcal{P}}}
\nc{\CQ}{{\mathcal{Q}}} \nc{\CR}{{\mathcal{R}}}
\nc{\CS}{{\mathcal{S}}} \nc{\CT}{{\mathcal{T}}}
\nc{\CU}{{\mathcal{U}}} \nc{\CV}{{\mathcal{V}}}
\nc{\CW}{{\mathcal{W}}} \nc{\CX}{{\mathcal{X}}}
\nc{\CY}{{\mathcal{Y}}} \nc{\CZ}{{\mathcal{Z}}}

\nc{\fa}{{\mathfrak{a}}}
\nc{\fb}{{\mathfrak{b}}}
\nc{\fg}{{\mathfrak{g}}}
\nc{\fgl}{{\mathfrak{gl}}}
\nc{\fh}{{\mathfrak{h}}}
\nc{\fj}{{\mathfrak{j}}}
\nc{\fl}{{\mathfrak{l}}}
\nc{\fm}{{\mathfrak{m}}}
\nc{\fn}{{\mathfrak{n}}}
\nc{\fu}{{\mathfrak{u}}}
\nc{\fp}{{\mathfrak{p}}}
\nc{\frr}{{\mathfrak{r}}}
\nc{\fs}{{\mathfrak{s}}}
\nc{\ft}{{\mathfrak{t}}}
\nc{\fw}{{\mathfrak{w}}}
\nc{\fz}{{\mathfrak{z}}}

\nc{\fA}{{\mathfrak{A}}}
\nc{\fB}{{\mathfrak{B}}}
\nc{\fD}{{\mathfrak{D}}}
\nc{\fE}{{\mathfrak{E}}}
\nc{\fF}{{\mathfrak{F}}}
\nc{\fG}{{\mathfrak{G}}}
\nc{\fI}{{\mathfrak{I}}}
\nc{\fJ}{{\mathfrak{J}}}
\nc{\fK}{{\mathfrak{K}}}
\nc{\fL}{{\mathfrak{L}}}
\nc{\fM}{{\mathfrak{M}}}
\nc{\fN}{{\mathfrak{N}}}
\nc{\frP}{{\mathfrak{P}}}
\nc{\fQ}{{\mathfrak Q}}
\nc{\fR}{{\mathfrak R}}
\nc{\fS}{{\mathfrak S}}
\nc{\fT}{{\mathfrak{T}}}
\nc{\fU}{{\mathfrak{U}}}
\nc{\fW}{{\mathfrak{W}}}
\nc{\fY}{{\mathfrak{Y}}}
\nc{\fZ}{{\mathfrak{Z}}}

\nc{\ba}{{\mathbf{a}}}
\nc{\bb}{{\mathbf{b}}}
\nc{\bc}{{\mathbf{c}}}
\nc{\bd}{{\mathbf{d}}}
\nc{\be}{{\mathbf{e}}}
\nc{\bi}{{\mathbf{i}}}
\nc{\bj}{{\mathbf{j}}}
\nc{\bn}{{\mathbf{n}}}
\nc{\bp}{{\mathbf{p}}}
\nc{\bq}{{\mathbf{q}}}
\nc{\bu}{{\mathbf{u}}}
\nc{\bv}{{\mathbf{v}}}
\nc{\bw}{{\mathbf{w}}}
\nc{\bx}{{\mathbf{x}}}
\nc{\by}{{\mathbf{y}}}
\nc{\bz}{{\mathbf{z}}}

\nc{\bA}{{\mathbf{A}}}
\nc{\bB}{{\mathbf{B}}}
\nc{\bC}{{\mathbf{C}}}
\nc{\bD}{{\mathbf{D}}}
\nc{\bE}{{\mathbf{E}}}
\nc{\bI}{{\mathbf{I}}}
\nc{\bK}{{\mathbf{K}}}
\nc{\bH}{{\mathbf{H}}}
\nc{\bM}{{\mathbf{M}}}
\nc{\bN}{{\mathbf{N}}}
\nc{\bO}{{\mathbf{O}}}
\nc{\bQ}{{\mathbf Q}}
\nc{\bS}{{\mathbf{S}}}
\nc{\bT}{{\mathbf{T}}}
\nc{\bV}{{\mathbf{V}}}
\nc{\bW}{{\mathbf{W}}}
\nc{\bX}{{\mathbf{X}}}
\nc{\bP}{{\mathbf{P}}}
\nc{\bY}{{\mathbf{Y}}}
\nc{\bZ}{{\mathbf{Z}}}

\nc{\sA}{{\mathsf{A}}}
\nc{\sB}{{\mathsf{B}}}
\nc{\sC}{{\mathsf{C}}}
\nc{\sD}{{\mathsf{D}}}
\nc{\sF}{{\mathsf{F}}}
\nc{\sK}{{\mathsf{K}}}
\nc{\sM}{{\mathsf{M}}}
\nc{\sO}{{\mathsf{O}}}
\nc{\sQ}{{\mathsf{Q}}}
\nc{\sP}{{\mathsf{P}}}
\nc{\sT}{{\mathsf{T}}}
\nc{\sV}{{\mathsf{V}}}
\nc{\sW}{{\mathsf{W}}}
\nc{\sX}{{\mathsf{X}}}
\nc{\sZ}{{\mathsf{Z}}}

\nc{\sfb}{{\mathsf{b}}}
\nc{\sfc}{{\mathsf{c}}}
\nc{\sd}{{\mathsf{d}}}
\nc{\sg}{{\mathsf{g}}}
\nc{\sk}{{\mathsf{k}}}
\nc{\sfl}{{\mathsf{l}}}
\nc{\sfp}{{\mathsf{p}}}
\nc{\sr}{{\mathsf{r}}}
\nc{\st}{{\mathsf{t}}}
\nc{\sfu}{{\mathsf{u}}}
\nc{\sw}{{\mathsf{w}}}
\nc{\sz}{{\mathsf{z}}}
\nc{\sx}{{\mathsf{x}}}

\nc{\bLambda}{{\boldsymbol{\Lambda}}}
\nc{\vv}{{\boldsymbol{v}}}
\nc{\Fl}{{{\mathcal F}\ell}}
\nc{\Gr}{{\on{Gr}}}
\nc{\CHH}{{\CH\!\!\CH}}
\nc{\lambdavee}{{\lambda^{\!\scriptscriptstyle\vee}}}
\nc{\alphavee}{\alpha^{\!\scriptscriptstyle\vee}}
\nc{\rhovee}{{\rho^{\!\scriptscriptstyle\vee}}}
\newcommand\iso{\,\vphantom{j^{X^2}}\smash{\overset{\sim}{\vphantom{\rule{0pt}{0.20em}}\smash{\longrightarrow}}}\,}
\nc{\oQM}{\vphantom{j^{X^2}}\smash{\overset{\circ}{\vphantom{\vstretch{0.7}{A}}\smash{\QM}}}}
\nc{\oZ}{{}^\dagger\!\vphantom{j^{X^2}}\smash{\overset{\circ}{\vphantom{\vstretch{0.7}{A}}\smash{Z}}}}
\nc{\odZ}{{}^\dagger\!\vphantom{j^{X^2}}\smash{\overset{\circ}{\vphantom{\vstretch{0.7}{A}}\smash{\mathfrak Z}}}^{c',c}}
\nc{\bdZ}{{}^\dagger\!\vphantom{j^{X^2}}\smash{\overset{\bullet}{\vphantom{\vstretch{0.7}{A}}\smash{\mathfrak Z}}}^{c',c}}
\nc{\oS}{\vphantom{j^{X^2}}\smash{\overset{\circ}{\vphantom{\vstretch{0.7}{A}}\smash{S}}}}
\nc{\buM}{\vphantom{j^{X^2}}\smash{\overset{\bullet}{\vphantom{\vstretch{0.7}{A}}\smash{M}}}}
\nc{\dW}{{}^\dagger\ol\CW{}}
\nc{\hW}{{}^\dagger\hat\CW{}}
\nc{\wW}{{}^\dagger\widetilde\CW{}}
\nc{\dZ}{{}^\dagger\!\fZ^{c',c}}
\nc{\dZc}{{}^\dagger\!\fZ^{c,c}}
\nc{\tZ}{{}^\dagger\!\tilde{Z}{}}
\nc{\hZ}{{}^\dagger\!\hat{Z}{}}

\nc{\ssl}{\mathfrak{sl}} \nc{\gl}{\mathfrak{gl}}
\nc{\wt}{\widetilde} \nc{\Sym}{\mathrm{Sym}} \nc{\Res}{\mathrm{Res}}
\nc{\sE}{{\mathsf{E}}} \nc{\bs}{{\mathbf{s}}}
\nc{\trig}{\mathrm{trig}} \nc{\rat}{\mathrm{rat}}
\nc{\sign}{\mathrm{sign}} \nc{\sL}{{\mathsf{L}}}
\nc{\fv}{{\mathfrak{v}}} \nc{\ad}{\mathrm{ad}}
\nc{\spsi}{{\mathsf{\psi}}} \nc{\sh}{{\mathsf{h}}}
\nc{\rtt}{\mathrm{rtt}} \nc{\qdet}{\mathrm{qdet}}
\nc{\M}{\mathrm{M}} \nc{\Ker}{\mathrm{Ker}} \nc{\ssc}{\mathrm{sc}}
\nc{\loc}{\mathrm{loc}} \nc{\fra}{\mathrm{frac}}
\nc{\ddj}{\mathrm{DJ}}
\nc{\rk}{\mathrm{rk}}
\nc{\g}{\mathfrak{g}}
\nc{\h}{\mathfrak{h}}
\nc{\BT}{{\mathbb{T}}}
\nc{\fk}{{\mathfrak{k}}}
\nc{\Or}{\mathrm{Or}}
\nc{\tr}{\mathrm{tr}}
\nc{\op}{\mathrm{op}}
\nc{\D}{\mathbf{D}}
\nc{\HH}{\mathsf{H}}
\nc{\BH}{{\mathbb{H}}}
\nc{\n}{\mathsf{n}}
\nc{\m}{\mathsf{m}}
\nc{\evv}{\mathsf{ev}}
\nc{\NN}{\mathsf{N}}
\nc{\End}{\mathrm{End}}
\nc{\id}{\mathrm{id}}
\nc{\Dyn}{\mathrm{Dyn}}

\allowdisplaybreaks[1]

\begin{document}
\title[Modified quantum difference Toda systems]
{On Sevostyanov's construction of quantum difference Toda lattices}

\author{Roman Gonin}
 \address{R.G.:
  National Research University Higher School of Economics, Department of Mathematics,
  6 Usacheva st., Moscow 119048, Russian Federation;
  Center for Advanced Studies, Skoltech}
 \email{roma-gonin@yandex.ru}

\author[Alexander Tsymbaliuk]{Alexander Tsymbaliuk}
 \address{A.T.: Yale University, Department of Mathematics, New Haven, CT 06511, USA}
 \email{sashikts@gmail.com}

\dedicatory{To our teacher and dear friend Michael Finkelberg, with admiration and gratitude}

\begin{abstract}
We propose a natural generalization of the construction of the quantum
difference Toda lattice~\cite{E,S2} associated to a simple Lie algebra $\g$.
Our construction depends on two orientations of the Dynkin diagram of $\g$ and
some other data (which we refer to as a pair of \emph{Sevostyanov triples}). In types
$A$ and $C$, we provide an alternative construction via Lax matrix formalism, cf.~\cite{KuT}.
We also show that the generating function of the pairing of Whittaker vectors in
the Verma modules is an eigenfunction of the corresponding modified quantum difference
Toda system and derive fermionic formulas for the former in spirit of~\cite{FFJMM}.
We give a geometric interpretation of all Whittaker vectors in type $A$ via line
bundles on the Laumon moduli spaces and obtain an edge-weight path model
for them, generalizing the construction of~\cite{DFKT}.
\end{abstract}
\maketitle


\section{Introduction}

In a recent work~\cite{FT} of M.~Finkelberg and the second author, a family of $3^{n-1}$
commutative subalgebras in the algebra of difference operators on $(\BC^\times)^{n+1}$
was constructed, generalizing the type $A$ quantum difference Toda lattice of~\cite{E,S2}.
In this paper, we show how the construction of~\cite{S2} for an arbitrary semisimple Lie algebra
$\g$ can be generalized to produce $3^{\rk(\g)-1}$ \emph{integrable systems}, thus answering
a question of P.~Etingof. In types $A$ and $C$, we identify these systems with the ones obtained
via the Lax matrix formalism. We also discuss some generalizations of the basic results on the
quantum difference Toda system to the current setting.

The importance of our generalization of $q$-Toda systems of~\cite{E,S2} is two-fold.
First of all, as emphasized in~\cite{FT} (historically this goes back at least
to~\cite{stss}), already the quasi-classical limit of this construction (known as the
\emph{relativistic open Toda system}) crucially depends on a choice of a pair of
Coxeter elements in the Weyl group of $G$ (simply-connected algberaic group
associated with $\g$).
One of our main results, Theorem~\ref{Main Conj}, gives an upper bound on the number
of different integrable systems we obtain this way in the quantum case.
Another motivation arises from the geometric representation theory, where
Whittaker vectors (closely related to the Toda systems due to Theorem~\ref{J-eigenfunction})
often have natural geometric interpretations which unveil additional symmetry.
We illustrate this in Section~\ref{section geometric}, where the universal Verma module
over $U_v(\ssl_n)$ is realized as the equivariant $K$-theory of Laumon spaces due to~\cite{BF}
(see Theorem~\ref{Br-Fin}), one of the Whittaker vectors is realized as a sum of the structure sheaves
(see~(\ref{simplest Whittaker}) and Proposition~\ref{eigen-property}(a)), while an extra symmetry noticed
in Proposition~\ref{eigen-property}(b) gives rise to a family of Whittaker vectors
(see Theorem~\ref{geometric Whittaker} and Proposition~\ref{Feigin's relation continued}).

\medskip

This paper is organized as follows:

$\bullet$
In Section~\ref{section basic}, we construct the \emph{modified quantum difference Toda systems}
depending on a pair of \emph{Sevostyanov triples} (following~\cite{S1,S2}) and generalizing the
$q$-Toda systems of~\cite{E,S2}.

$\bullet$
In Section~\ref{section power of 3}, we explain how to compute explicitly the
corresponding hamiltonians using~\cite{KT}. We write down the formulas for the
hamiltonians corresponding to the first fundamental representation in the classical
types and $G_2$. In the latter case of $G_2$, our formula seems to be new even in the
simplest setup of the standard $q$-Toda system of~\cite{E}, see~(\ref{standard qToda for G2}).

One of the key results of this section is that there are at most $3^{\rk(\g)-1}$
different modified quantum difference Toda systems, see Theorem~\ref{Main Conj}.
For the classical types and $G_2$, see Theorem~\ref{Main Thm A}, whose proof is more elementary and
relies on Propositions~\ref{bound for A},~\ref{bound for C},~\ref{bound for D},~\ref{bound for B},~\ref{bound for G2}.
We also show that these are maximal commutative subalgebras, determined by their first hamiltonians,
see Theorem~\ref{Main Thm B}.

We also prove that in type $A$ these integrable systems exactly match those of~\cite[11(ii, iii)]{FT},
see Theorem~\ref{Lax for type A}. This generalizes the Lax matrix realization of
the type $A$ $q$-Toda system, due to~\cite{KuT}. In Theorem~\ref{Lax for type C}, we also
provide a similar Lax matrix realization of the type $C$ modified quantum difference Toda systems.
Noticing that the \emph{periodic} counterparts of these two constructions in the classical case
(that is, for $\vec{k}=\vec{0}$ in the notations of \emph{loc.cit.}) match up with the hamiltonians
of the affine $q$-Toda lattice of~\cite{E}, see
formulas~(\ref{affine Toda A},~\ref{affine Toda C},~\ref{affine Toda D},~\ref{affine Toda B},~\ref{affine Toda G2}),
we propose a periodic analogue of the modified quantum difference Toda systems
in types $A, C$, see Propositions~\ref{Closed Toda for A},~\ref{Closed Toda for C}
and Remarks~\ref{affine Toda 1}(b,c),~\ref{affine Toda 2}(b).

$\bullet$
In Section~\ref{section properties}, we study the Shapovalov pairing between a pair of Whittaker vectors
(determined by Sevostyanov triples) in Verma modules. We obtain fermionic formulas for those in spirit
of~\cite{FFJMM}, see Theorems~\ref{Fermionic one},~\ref{Fermionic two}. We also prove that their generating
function is naturally an eigenfunction of the corresponding modified quantum difference Toda system,
see Theorem~\ref{J-eigenfunction}.

$\bullet$
In Section~\ref{section geometric}, we provide a geometric interpretation of all type $A$
Whittaker vectors and their Shapovalov pairing via the geometry of the Laumon moduli spaces,
generalizing~\cite{BF}, see Theorems~\ref{geometric Whittaker},~\ref{geometric J-function}.

Following a suggestion of B.~Feigin, we relate this family of Whittaker vectors to an
eigen-property of the (geometrically) simplest one~(\ref{simplest Whittaker}) with respect
to the action of the quantum loop algebra $U_v(L\ssl_n)$ (via the evaluation homomorphism),
see Propositions~\ref{eigen-property},~\ref{Feigin's relation continued}, and
Corollary~\ref{Feigin's relation}. This viewpoint also provides an edge-weight path model
for a general type $A$ Whittaker vector, generalizing the path model of~\cite{DFKT} for a
particular choice of a Sevostyanov triple, see Propositions~\ref{Path model 1},~\ref{Path model 2}.

$\bullet$
In Appendices, we prove Proposition~\ref{bound for A} and
Theorems~\ref{Main Conj},~\ref{Main Thm A},~\ref{Main Thm B},~\ref{Lax for type A}.


\subsection*{Acknowledgments}
\

We are deeply grateful to P.~Etingof for suggesting this problem, stimulating discussions,
and proposing the elegant argument used in the proofs of Theorems~\ref{Main Thm A},~\ref{Main Thm B};
to A.~Braverman for the insight that our proof of Theorem~\ref{Main Thm A} combined with
Theorem~\ref{Fermionic two} provides an immediate proof of Theorem~\ref{Main Conj};
to B.~Feigin for enlightening discussions on the Whittaker vectors and
fermionic formulas; to P.~Di~Francesco, R.~Kedem, A.~Sevostyanov, and M.~Yakimov for a useful
correspondence; to M.~Bershtein and A.~Marshakov for their interest in this work.
We also thank the anonymous referee for useful suggestions which improved the exposition.

Last, but not least, we are extremely grateful to M.~Finkelberg for his encouragement,
numerous discussions, and constant support. This paper would not appear without his
interest in the current project and inspiring discussions the authors had with him.
We would like to use this opportunity to express our gratitude to M.~Finkelberg for all
the beautiful mathematics he has taught us for more than a decade.

The work of R.G.\ has been funded by the  Russian Academic Excellence Project `5-100'.
A.T.\ gratefully acknowledges support from Yale University at which most of the research
for this paper was performed, and is extremely grateful to IHES (Bures-sur-Yvette, France)
for the hospitality and wonderful working conditions in the summer 2017 when this project
was initiated. A.T.\ was partially supported by the NSF Grants DMS--1502497, DMS--1821185.


\section{Sevostyanov triples and Whittaker functions}\label{section basic}


\subsection{Quantum groups}
\

We fix the notations as follows. Let $G$ be a simply-connected complex
algebraic group with a semisimple Lie algebra $\g$. We denote by $H\subset B$
a pair of a Cartan torus and a Borel subgroup.
The Cartan subalgebra $\h\subset \g$ is defined as the Lie algebra of $H$,
$\Delta$ denotes the set of roots of $(\g,\h)$, $\Delta_+\subset \Delta$ denotes the set
of positive roots corresponding to $B$. Let $n=\rk(\g)$ be the rank of $\g$,
$\alpha_1,\ldots,\alpha_n$ be the simple positive roots, and $\omega_1,\ldots,\omega_n$ be
the fundamental weights. Let $P:=\oplus_{i=1}^n \BZ\omega_i$ be the weight lattice,
$Q:=\oplus_{i=1}^n \BZ\alpha_i$ be the root lattice, and set
$P_+:=\oplus_{i=1}^n \BZ_{\geq 0}\omega_i,\ Q_+:=\oplus_{i=1}^n \BZ_{\geq 0}\alpha_i$.
We write $\beta\geq \gamma$ if $\beta-\gamma\in Q_+$.
We fix a non-degenerate invariant bilinear form $(\cdot,\cdot)\colon \h\times\h\to \BC$
and identify $\h^*$ with $\h$ via $(\cdot,\cdot)$. We set $\sd_i:=\frac{(\alpha_i,\alpha_i)}{2}$.
The choice of $(\cdot,\cdot)$ is such that $\sd_i=1$ for short roots $\alpha_i$,
in particular, $\sd_i\in \{1,2,3\}$  for any $i$. We also define $\omega_i^\vee:=\omega_i/\sd_i$
so that $(\omega_i^\vee,\alpha_j)=\delta_{i,j}$, and
$\rho:=\sum_{i=1}^n \omega_i=\frac{1}{2}\sum_{\gamma\in \Delta_+}\gamma\in P$.
Let $(a_{ij})_{i,j=1}^n$ be the corresponding Cartan matrix with
$a_{ij}=\frac{2(\alpha_i,\alpha_j)}{(\alpha_i,\alpha_i)}$. We define
$b_{ij}=\sd_ia_{ij}=(\alpha_i,\alpha_j)$, so that $(b_{ij})_{i,j=1}^n$ is symmetric.

Choose $\NN\in \BZ_{>0}$ so that $(P,P)\subset \frac{1}{\NN}\BZ$.
The quantum group (of \emph{adjoint type} in the terminology of~\cite{L}) $U_v(\g)$
is the unital associative $\BC(v^{1/\NN})$-algebra generated by
$\{E_i,F_i,K_\mu\}_{1\leq i\leq n}^{\mu\in P}$ with the following defining relations:
\begin{align*}
  & K_{\mu} K_{\mu'}=K_{\mu+\mu'},\ K_0=1,\\
  & K_\mu E_i K^{-1}_\mu=v^{(\mu,\alpha_i)}E_i,\
    K_\mu F_i K^{-1}_\mu=v^{-(\mu,\alpha_i)}F_i,\
    [E_i,F_j]=\delta_{i,j}\frac{K_i-K_i^{-1}}{v_i-v_i^{-1}},\\
  & \sum_{r=0}^{1-a_{ij}}(-1)^r {1-a_{ij}\brack r}_{v_i} E_i^{1-a_{ij}-r}E_jE_i^{r}=0,\
    \sum_{r=0}^{1-a_{ij}}(-1)^r {1-a_{ij}\brack r}_{v_i} F_i^{1-a_{ij}-r}F_jF_i^{r}=0\ (i\ne j),
\end{align*}
where
 $K_i:=K_{\alpha_i},\ v_i:=v^{\sd_i},\ [r]_v:=\frac{v^r-v^{-r}}{v-v^{-1}},\
  [r]_v!:=[1]_v\cdots [r]_v,\ {m\brack r}_v:=\frac{[m]_v!}{[r]_v!\cdot [m-r]_v!}$.

Set $L_i:=K_{\omega_i}$. Since $P=\oplus_{i=1}^n \BZ\omega_i$, we will alternatively view
$U_v(\g)$ as the $\BC(v^{1/\NN})$-algebra generated by $\{E_i,F_i,L_i^{\pm 1}\}_{i=1}^n$
with the corresponding defining relations. In particular,
\begin{equation*}
  L_iE_jL_i^{-1}=v_i^{\delta_{i,j}}E_j,\
  L_iF_jL_i^{-1}=v_i^{-\delta_{i,j}}F_j,\
  K_i=\prod_{j=1}^n L_j^{a_{ji}}.
\end{equation*}


\subsection{Sevostyanov triples}\label{Sevostyanov triples}
\

Let $\Dyn(\g)$ be the graph obtained from the Dynkin diagram of $\g$ by replacing
all multiple edges by simple ones, e.g.,
  $\Dyn(\mathfrak{sp}_{2n})=\Dyn(\mathfrak{so}_{2n+1})=\Dyn(\ssl_{n+1})=A_n$.
Given an orientation $\Or$ of $\Dyn(\g)$, define the associated matrix
$\epsilon=(\epsilon_{ij})_{i,j=1}^n$ via
\begin{equation*}
  \epsilon_{ij}=
  \begin{cases}
     0, & \mathrm{if}\ a_{ij}=0\ \mathrm{or}\ i=j,\\
     1, & \mathrm{if}\ a_{ij}<0\ \mathrm{and\ the\ edge\ is\ oriented}\ i\rightarrow j\ \mathrm{in}\ \Or,\\
     -1, & \mathrm{if}\ a_{ij}<0\ \mathrm{and\ the\ edge\ is\ oriented}\ i\leftarrow j\ \mathrm{in}\ \Or.
  \end{cases}
\end{equation*}

\begin{Def}
A \textbf{Sevostyanov triple} is a collection of the following data:

(a) an orientation $\Or$ of $\Dyn(\g)$,

(b) an integer matrix $\n=(\n_{ij})_{i,j=1}^n$ satisfying
$\sd_j\n_{ij}-\sd_i\n_{ji}=\epsilon_{ij}b_{ij}$ for any $i,j$,

(c) a collection $c=(c_i)_{i=1}^n\in (\BC(v^{1/\NN})^\times)^n$.

\noindent
We refer to this \textbf{Sevostyanov triple} by $(\epsilon, \n, c)$.
\end{Def}

Fix a pair of integer matrices $\n^\pm=(\n^\pm_{ij})_{i,j=1}^n$ and collections
$c^\pm=(c^\pm_i)_{i=1}^n\in (\BC(v^{1/\NN})^\times)^n$. Set
  $e_i:=E_i\cdot \prod_{p=1}^n L_p^{\n^+_{ip}},\
   f_i:=\prod_{p=1}^n L_p^{-\n^-_{ip}}\cdot F_i$,
and let $U^{+}_{\n^+}(\g), U^{-}_{\n^-}(\g)$ be the $\BC(v^{1/\NN})$-subalgebras of $U_v(\g)$
generated by $\{e_i\}_{i=1}^n$ and $\{f_i\}_{i=1}^n$, respectively.

The following simple, but very important, observation is essentially due to~\cite{S1}:

\begin{Lem}\label{Sev twist}
(a) The assignment $e_i\mapsto c^+_i\ (1\leq i\leq n)$ extends to an algebra homomorphism
$\chi^+\colon U^{+}_{\n^+}(\g)\to \BC(v^{1/\NN})$ if and only if there exists an orientation
$\Or^+$ of $\Dyn(\g)$ with an associated matrix $\epsilon^+$, such that $(\epsilon^+,\n^+,c^+)$
is a Sevostyanov triple.

\noindent
(b) The assignment $f_i\mapsto c^-_i\ (1\leq i\leq n)$ extends to an algebra homomorphism
$\chi^-\colon U^{-}_{\n^-}(\g)\to \BC(v^{1/\NN})$ if and only if there exists an orientation
$\Or^-$ of $\Dyn(\g)$ with an associated matrix $\epsilon^-$, such that $(\epsilon^-,\n^-,c^-)$
is a Sevostyanov triple.
\end{Lem}

\begin{proof}
(a) As the ``if'' part is proved in~\cite[Theorem 4]{S1}, let us now prove the ``only if''
part following similar arguments. Due to the triangular decomposition of $U_v(\g)$, the algebra
$U^{+}_{\n^+}(\g)$ is generated by $\{e_i\}_{i=1}^n$ subject to
  $\sum_{r=0}^{1-a_{ij}}(-1)^rv^{r(\sd_j\n^+_{ij}-\sd_i\n^+_{ji})}
   {1-a_{ij}\brack r}_{v_i} e_i^{1-a_{ij}-r}e_je_i^{r}=0$
for $i\ne j$. Hence, there is a character
  $\chi^+\colon U^{+}_{\n^+}(\g)\to \BC(v^{1/\NN})$ with $\chi^+(e_i)\ne 0$
if and only if
  $\sum_{r=0}^{1-a_{ij}}(-1)^rv^{r(\sd_j\n^+_{ij}-\sd_i\n^+_{ji})} {1-a_{ij}\brack r}_{v_i}=0$
for any $i\ne j$. If $a_{ij}=0$, then we immediately get $\sd_j\n^+_{ij}-\sd_i\n^+_{ji}=0$.
If $a_{ij}=-1$, then we recover $\sd_j\n^+_{ij}-\sd_i\n^+_{ji}\in \{\pm \sd_i\}=\{\pm b_{ij}\}$
and $\epsilon_{ij}\in \{\pm 1\}$. Finally, if $a_{ij}<-1$, then $a_{ji}=-1$ and we can apply
the previous case.
(b) Analogous.
\end{proof}

\subsection{Whittaker functions}\label{section Whittaker functions}
\

From now on, we fix a pair of Sevostyanov triples $(\epsilon^\pm,\n^\pm,c^\pm)$,
which give rise to the subalgebras $U^{\pm}_{\n^\pm}(\g)$ of $U_v(\g)$ and the corresponding
characters $\chi^\pm\colon U^{\pm}_{\n^\pm}(\g)\to \BC(v^{1/\NN})$ of Lemma~\ref{Sev twist}.
We consider the quantum function algebra $\CO_v(G)$ spanned by the matrix coefficients
of integrable $U_v(\g)$-modules (with the highest weights in $P_+$). Let ${\mathcal D}_v(G)$
denote the corresponding Heisenberg double~\cite[Section 3]{sts}. It acts on $\CO_v(G)$.
It is equipped with a homomorphism $\mu_v\colon U_v(\g)\otimes U_v(\g)\to{\mathcal D}_v(G)$.
Let $\CO_v(Bw_0B)$ stand for the quantized coordinate ring of the big Bruhat cell~\cite[8.2]{GY}
(a localization of $\CO_v(G)$).
The action of ${\mathcal D}_v(G)$ on $\CO_v(G)$ extends to the action on $\CO_v(Bw_0B)$.
In particular, $U^-_{\n^-}(\g)\otimes U^+_{\n^+}(\g)\subset U_v(\g)\otimes U_v(\g)$ acts
on $\CO_v(Bw_0B)$. According to~\cite[(3.22), Theorem~4.7, Proposition~8.3]{GY}, there are
subalgebras $S^\pm_v$ of $\CO_v(Bw_0B)$\footnote{We note that $\CO_v(G), \CO_v(Bw_0B),S^+_v$
are denoted by $R_v[G], R_v[Bw_0B],\overline{S}^\mp_{w_0}$, respectively, in~\cite{GY}.}
such that $\CO_v(Bw_0B)\simeq S^-_v\otimes \CO_v(H)\otimes S^+_v$ (as vector spaces)
and $S^\pm_v\simeq U^\pm_v(\g)$, where $U^-_v(\g), U^+_v(\g)$ are the subalgebras of $U_v(\g)$
generated by $\{F_i\}_{i=1}^n$ and $\{E_i\}_{i=1}^n$, respectively. Hence, there is
a (vector space) isomorphism
\begin{equation}\label{triangular}
  \CO_v(Bw_0B)\simeq U^-_{\n^-}(\g)\otimes \CO_v(H)\otimes U^+_{\n^+}(\g),
\end{equation}
under which the above actions of $U^-_{\n^-}(\g),U^+_{\n^+}(\g)$ on $\CO_v(Bw_0B)$
are via the left and the right multiplications. Let $U^\pm_{\n^\pm}(\g)^\wedge$ denote
the completions of $U^\pm_{\n^\pm}(\g)$ with respect to the natural gradings with
$\deg(e_i)=1$ and $\deg(f_i)=1$. In view of the identification~(\ref{triangular}),
we define the completion of $\CO_v(Bw_0B)$ via
  $\CO_v(Bw_0B)^\wedge\simeq U^-_{\n^-}(\g)^\wedge\otimes \CO_v(H)\otimes U^+_{\n^+}(\g)^\wedge$.
Hence, the subspace of semi-invariants
  $\left(\CO_v(Bw_0B)^\wedge\right)^{U^-_{\n^-}(\g)\otimes U^+_{\n^+}(\g),\chi^-\otimes\chi^+}$
projects isomorphically onto $\CO_v(H)$ under the restriction projection
$\CO_v(Bw_0B)^\wedge\to\CO_v(H)$. We denote this projection by $\phi\mapsto\phi_{\mid_H}$.

\begin{Def}\label{Whit function}
A \textbf{Whittaker function} is an element of
  $\left(\CO_v(Bw_0B)^\wedge\right)^{U^-_{\n^-}(\g)\otimes U^+_{\n^+}(\g),\chi^-\otimes\chi^+}$.
\end{Def}

\begin{Rem}\label{formal}
Following~\cite{E}, we could alternatively work with the dual quantum formal group
$\CA_\hbar(\g)=U_\hbar(\g)^*$, defined as the space of linear functions on $U_\hbar(\g)$.
Here, the quantum group $U_\hbar(\g)$ is defined over $\BC[[\hbar]]$ with $v$ replaced
by $e^\hbar$. In this setup, a \textbf{Whittaker function} is an element
$\phi\in \CA_\hbar(\g)$ such that $\phi(x^-xx^+)=\chi^-(x^-)\chi^+(x^+)\phi(x)$ for any
$x^\pm \in U^{\pm}_{\n^\pm}(\g), x\in U_\hbar(\g)$. Let us point out that this differs
from the notion of Whittaker functions as defined in~\emph{loc.cit.}
\end{Rem}

We note that the character lattice $X^*(H)=P$ and the pairing $(Q,P)\subset \BZ$,
hence, we have the natural embedding of $Q$ into the cocharacter lattice $X_*(H)$.
Thus, for every $\lambda\in Q$ we can define the difference operators $T_\lambda$
acting on $\CO_v(H)$ via $(T_\lambda f)(x)=f(x\cdot v^\lambda)$. Moreover, since
$v^{(P,P)}\subset \BC(v^{1/\NN})$, the difference operators $T_\lambda$ are also
well-defined for $\lambda\in P$. Let $\wt{\CalD}_v(H)$ be the algebra generated by
$\{e^\lambda,T_{\mu}|\lambda,\mu \in P\}$, and $\CalD_v(H^\ad)$ be its subalgebra
generated by $\{e^\lambda,T_{\mu}|\lambda\in Q,\mu \in P\}$.
The following is completely analogous to~\cite[Proposition 3.2]{E}:

\begin{Lem}\label{commuting operators}
(a) For any $Y\in U_v(\g)$, there exists a unique difference operator
$\tilde{\D}_Y=\tilde{\D}_Y(\epsilon^\pm,\n^\pm,c^\pm)\in \wt{\CalD}_v(H)$ such that
$(Y\phi)_{\mid_H}=\tilde{\D}_Y(\phi_{\mid_H})$ for any Whittaker function $\phi$.

\noindent
(b) $\tilde{\D}_Y$ is an element of $\CalD_v(H^\ad)\subset \wt{\CalD}_v(H)$.

\noindent
(c) If $Y_1$ and $Y_2$ are central elements of $U_v(\g)$, then
$\tilde{\D}_{Y_1Y_2}=\tilde{\D}_{Y_1}\tilde{\D}_{Y_2}$.
\end{Lem}

Recall the element $\Theta\in (U_v(\g)\otimes U_v(\g))^\wedge$ of the completion
of the vector space $U_v(\g)\otimes U_v(\g)$ as defined in~\cite[4.1.1]{L}. Loosely
speaking, the universal $R$-matrix is given by $R=\Theta^\op \cdot R^0$, where $R^0=v^T$
and $T\in \h\otimes \h$ stands for the canonical element. Let $\pi_V\colon U_v(\g)\to \End(V)$
be a finite-dimensional representation, $\{w_k\}_{k=1}^N$ be a weight basis of $V$,
and $\mu_k\in P$ be the weight of $w_k$. First, we note that though $\Theta,\Theta^\op$
are defined as infinite sums, their images
  $(\id\otimes \pi_V)(\Theta), (\id\otimes \pi_V)(\Theta^\op)\in U_v(\g)\otimes \End(V)$
are well-defined. Second, the image
  $(\id\otimes \pi_V)(R^0)=(\id\otimes \pi_V)((R^0)^\op)\in U_v(\g)\otimes \End(V)$
is also well-defined via
  $(\id\otimes \pi_V)(R^0)=\sum_{k=1}^N K_{\mu_k}\otimes E_{k,k}$
with $E_{k,k}\in \End(V)$ given by $E_{k,k}(w_{k'})=\delta_{k,k'}w_{k'}$
(this does not depend on the choice of a weight basis $\{w_k\}$).
Hence, working over $\BC(v^{1/\NN})$ (rather than in the formal setting
$\BC[[\hbar]]$ as in~\cite{E,S2}), the elements
$(\id\otimes \pi_V)(R)$ and $(\id\otimes \pi_V)(R^\op)$ are still well-defined.

Due to~\cite{D,R}, the center of $U_v(\g)$ is spanned by elements $C_V$ corresponding to
finite-dimensional $U_v(\g)$-representations $V$ via the formula
\begin{equation}\label{quantum center}
  C_V=\tr_V(\id\otimes \pi_V)\left(R^\op R(1\otimes K_{2\rho})\right).
\end{equation}
We define $\tilde{\D}_V, \D_V\in \CalD_v(H^\ad)$ via $\tilde{\D}_V:=\tilde{\D}_{C_V}$
and $\D_V:=e^\rho\tilde{\D}_Ve^{-\rho}$.
Consider the fundamental representations $\{V_i\}_{i=1}^n$ of $U_v(\g)$ and set
$\tilde{\D}_i:=\tilde{\D}_{V_i}, \D_i:=\D_{V_i}$. According to Lemma~\ref{commuting operators},
$\{\tilde{\D}_i\}_{i=1}^n$ and therefore $\{\D_i\}_{i=1}^n$ are families of pairwise
commuting elements of $\CalD_v(H^\ad)$.

\begin{Def}
A \textbf{modified quantum difference Toda system} is the commutative subalgebra
$\CT=\CT(\epsilon^\pm,\n^\pm,c^\pm)$ of $\CalD_v(H^\ad)$ generated by $\{\D_i\}_{i=1}^n$.
\end{Def}

Due to Theorem~\ref{Main Thm B}(c) below, $\D_V\in \CT$ for any finite-dimensional
$U_v(\g)$-representation $V$.

\begin{Rem}
This construction is a $q$-deformed version of the Kazhdan-Kostant approach to the
classical Toda system. In case the two Sevostyanov triples coincide, we recover the
original construction of~\cite{S2}. Let us point out right away that we do not know
how to generalize an alternative approach of~\cite{E} to obtain our modified quantum
difference Toda systems.
\end{Rem}


\section{First hamiltonians, classification, and Lax realization in types A,C}\label{section power of 3}

The main result of this section is:

\begin{Thm}\label{Main Conj}
There are at most $3^{n-1}$ different modified quantum difference Toda systems,
up to algebra automorphisms of $\CalD_v(H^\ad)$.
\end{Thm}

The proof of this result is presented in Appendix~\ref{Proof of Main Conjecture}
and crucially relies on Theorem~\ref{Fermionic two}.

We also provide a more straightforward proof for the classical types $A_n,B_n,C_n,D_n$ as well as
the exceptional type $G_2$. To state the result, we label the simple roots $\{\alpha_i\}_{i=1}^n$
as in~\cite[Chapter VI, \S4]{B} (here $n=2$ for the type $G_2$).
Given a pair of Sevostyanov triples $(\epsilon^\pm,\n^\pm,c^\pm)$, we define
$\vec{\epsilon}=(\epsilon_{n-1},\ldots,\epsilon_1)\in \{-1,0,1\}^{n-1}$ via
  $\epsilon_i:=\begin{cases}
      \frac{\epsilon^+_{n-2,n}-\epsilon^-_{n-2,n}}{2}, & \mathrm{if}\ i=n-1\ \mathrm{in\ type}\ D_n,\\
      \frac{\epsilon^+_{i,i+1}-\epsilon^-_{i,i+1}}{2}, & \mathrm{otherwise}.
   \end{cases}$

\begin{Thm}\label{Main Thm A}
If $\g$ is of type $A_n,B_n,C_n,D_n$ or $G_2$, then up to algebra automorphisms
of $\CalD_v(H^\ad)$, the modified quantum difference Toda system
$\CT(\epsilon^\pm,\n^\pm,c^\pm)$ depends only on $\vec{\epsilon}$.
\end{Thm}

We present the proof of this result in Appendix~\ref{Proof of Main Theorem A}.
The key ingredient in our proof is that the first hamiltonian $\D_1$ depends
only on $\vec{\epsilon}\in \{-1,0,1\}^{n-1}$ up to an algebra automorphism of
$\CalD_v(H^\ad)$, which is established case-by-case in
Propositions~\ref{bound for A},~\ref{bound for C},~\ref{bound for D},~\ref{bound for B},~\ref{bound for G2}.
Following an elegant argument of P.~Etingof, we show in Appendix~\ref{Proof of Main Theorem A}
that the other hamiltonians $\D_i$ match as well under the same automorphism.

Let $\CalD^\leq_v(H^\ad)$ be the subalgebra of $\CalD_v(H^\ad)$, generated
by $\{e^{-\alpha_i},T_{\mu}|1\leq i\leq n, \mu \in P\}$. It follows from the
construction that $\D_i\in \CalD^\leq_v(H^\ad)$, so that $\CT\subset \CalD^\leq_v(H^\ad)$.
Applying ideas similar to those from the proof of Theorem~\ref{Main Thm A},
we get another important result:

\begin{Thm}\label{Main Thm B}
Consider a modified quantum difference Toda system $\CT=\CT(\epsilon^\pm,\n^\pm,c^\pm)$.

\noindent
(a) The difference operators $\{\D_i\}_{i=1}^n\subset \CT(\epsilon^\pm,\n^\pm,c^\pm)$
are algebraically independent.

\noindent
(b) The centralizer of $\D_1$ in $\CalD^\leq_v(H^\ad)$ coincides with $\CT(\epsilon^\pm,\n^\pm,c^\pm)$.

\noindent
(c) We have $\D_V(\epsilon^\pm,\n^\pm,c^\pm)\in \CT(\epsilon^\pm,\n^\pm,c^\pm)$ for any
finite-dimensional $U_v(\g)$-module $V$.
\end{Thm}

The proof of Theorem~\ref{Main Thm B} is presented in Appendix~\ref{Proof of Main Theorem B}.


\subsection{R-matrix and convex orderings}\label{subsection technicalities}
\

Our computations are based on the explicit formula for the universal $R$-matrix $R$,
due to~\cite{KT}. First, let us recall the construction of Cartan-Weyl root elements
$\{E_\gamma,F_\gamma\}_{\gamma\in \Delta_+}$, which is crucially based on the notion of
a convex ordering on $\Delta_+$.

\begin{Def}
An ordering $\prec$ on the set of positive roots $\Delta_+$ is called
\textbf{convex}\footnote{We note that such orderings are called  \emph{normal} in~\cite{KT,Tol}.}
if for any three roots $\alpha,\beta,\gamma\in \Delta_+$ such that $\gamma=\alpha+\beta$,
we have either $\alpha\prec \gamma\prec \beta$ or $\beta\prec \gamma\prec \alpha$.
\end{Def}

Fix a convex ordering $\prec$ on $\Delta_+$.
For a simple root $\alpha_i\ (1\leq i\leq n)$, set $E_{\alpha_i}:=E_i, F_{\alpha_i}:=F_i$.
To construct the remaining root vectors, we apply the following inductive algorithm.
Let $\alpha,\beta,\gamma\in \Delta_+$ be such that $\gamma=\alpha+\beta, \alpha\prec\beta$,
and there are no $\alpha\npreceq \alpha'\prec\beta'\npreceq\beta$ satisfying
$\gamma=\alpha'+\beta'$. Suppose that $E_\alpha,F_\alpha,E_\beta,F_\beta$ have been
already constructed. Then, we define
\begin{equation*}
  E_\gamma:=E_\alpha E_\beta-v^{(\alpha,\beta)}E_\beta E_\alpha,\
  F_\gamma:=F_\beta F_\alpha-v^{-(\alpha,\beta)}F_\alpha F_\beta.
\end{equation*}
According to~\cite{KT}, we have
  $[E_\gamma,F_\gamma]=a(\gamma)\frac{K_\gamma-K_\gamma^{-1}}{v_\gamma-v_\gamma^{-1}}$
for certain constants $a(\gamma)\in \BC(v^{1/\NN})$, where $v_\gamma:=v^{(\gamma,\gamma)/2}$
(note that $a(\alpha_i)=1$ and $v_{\alpha_i}=v_i$). For $\gamma\in \Delta_+$, define
\begin{equation*}
  R_\gamma:=\mathrm{exp}_{v_\gamma^{-1}}
  \left(\frac{v_\gamma-v_\gamma^{-1}}{a(\gamma)}E_\gamma\otimes F_\gamma\right),
\end{equation*}
where
  $\mathrm{exp}_v(x):=\sum_{r=0}^\infty \frac{x^r}{(r)_v!},
   (r)_v!:=(1)_v\cdots(r)_v, (r)_v:=\frac{1-v^r}{1-v}$.
The following is due to~\cite{KT}:

\begin{Thm}[\cite{KT}]\label{KT decomposition}
Fix a convex ordering $\prec$ on $\Delta_+$. Then,
  $\Theta^\op=\prod_{\gamma\in \Delta_+}R_\gamma$,
where the order in the product coincides with the ordering $\prec$.
\end{Thm}

The explicit computations of $\D_1$ below are based on the special choice of convex orderings.
We choose two convex orderings $\prec_\pm$ on $\Delta_+$ in such a way that
  $\epsilon^\pm_{ij}=-1\Rightarrow \alpha_i\prec_\pm \alpha_j$\footnote{As shown in~\cite{Tol},
any ordering on simple positive roots can be extended to a convex ordering on $\Delta_+$.}.
This choice is motivated by Proposition~\ref{reduced central elt} below.
To state the result, define
\begin{equation}\label{C'}
  C'_V=\tr_V(\id\otimes \pi_V)\left(\prod R_{\alpha_i}^\op\cdot (R^0)^\op\cdot
  \prod R_{\alpha_i}\cdot R^0\cdot (1\otimes K_{2\rho})\right),
\end{equation}
where the first and the second products are over all simple positive roots ordered
according to $\prec_-$ and $\prec_+$, respectively, whereas
$(\id\otimes \pi_V)(R^0)=(\id\otimes \pi_V)((R^0)^\op)$ are understood as before.
We define $\bar{\D}_{V}:=\tilde{\D}_{C'_V}$.

\begin{Prop}\label{reduced central elt}
We have $\tilde{\D}_V=\bar{\D}_V$.
\end{Prop}

\begin{proof}
For $\gamma=\sum_{i=1}^n m_i\alpha_i\in \Delta_+\ (m_i\in \BZ_{\geq 0})$, define
$e_\gamma,f_\gamma\in U_v(\g)$ via $e_\gamma:=E_\gamma\cdot \prod_{i,k=1}^n L_k^{m_i\n^+_{ik}}$
and $f_\gamma:=\prod_{i,k=1}^n L_k^{-m_i\n^-_{ik}}\cdot F_\gamma$, so that
$e_{\alpha_i}=e_i, f_{\alpha_i}=f_i$ as defined in Section~\ref{Sevostyanov triples}.
The proof of Proposition~\ref{reduced central elt} is based on the following properties
of these elements $\{e_\gamma, f_\gamma\}_{\gamma\in \Delta_+}$ established
in~\cite[Propositions 2.2.4, 2.2.5]{S2}.

\begin{Lem}\label{nonessential}
(a) For $\gamma\in \Delta_+$, we have $e_\gamma\in U^+_{\n^+}(\g)$ and $f_\gamma\in U^-_{\n^-}(\g)$.

\noindent
(b) If $\gamma\in \Delta_+$ is not a simple root, then $\chi^+(e_\gamma)=0$ and $\chi^-(f_\gamma)=0$.
\end{Lem}

We recall the proof of this Lemma to make our exposition self-contained.

\begin{proof}
(a) The proof is by induction used in the above definition of the root vectors $E_\gamma,F_\gamma$.
The claim is trivial when $\gamma$ is a simple root. For the remaining cases, let
$\alpha,\beta,\gamma\in \Delta_+$ be as above and assume that we have already established
the inclusions $e_\alpha,e_\beta\in U^+_{\n^+}(\g)$ and $f_\alpha,f_\beta\in U^-_{\n^-}(\g)$.
Let us write $\alpha=\sum_{i=1}^n m_i\alpha_i, \beta=\sum_{i=1}^n m'_i\alpha_i$. Then
\begin{align*}
  &  E_\gamma=\left(v^{-\sum_{i,k=1}^n m_i\n^+_{ik}(\omega_k,\beta)}e_\alpha e_\beta-
     v^{(\alpha,\beta)-\sum_{i,k=1}^n m'_i\n^+_{ik}(\omega_k,\alpha)}e_\beta e_\alpha\right)
     \cdot \prod_{i,k=1}^n L_k^{-(m_i+m'_i)\n^+_{ik}},\\
  &  F_\gamma=\prod_{i,k=1}^n L_k^{(m_i+m'_i)\n^-_{ik}}\cdot
     \left(v^{\sum_{i,k=1}^n m_i\n^-_{ik}(\omega_k,\beta)}f_\beta f_\alpha-v^{-(\alpha,\beta)+
     \sum_{i,k=1}^n m'_i\n^-_{ik}(\omega_k,\alpha)}f_\alpha f_\beta\right),
\end{align*}
so that
\begin{equation}\label{root basis e}
  e_\gamma=
  v^{-\sum_{i,k=1}^n m_i\n^+_{ik}(\omega_k,\beta)}e_\alpha e_\beta-v^{(\alpha,\beta)-
  \sum_{i,k=1}^n m'_i\n^+_{ik}(\omega_k,\alpha)}e_\beta e_\alpha
\end{equation}
and
\begin{equation}\label{root basis f}
  f_\gamma=
  v^{\sum_{i,k=1}^n m_i\n^-_{ik}(\omega_k,\beta)}f_\beta f_\alpha-
  v^{-(\alpha,\beta)+\sum_{i,k=1}^n m'_i\n^-_{ik}(\omega_k,\alpha)}f_\alpha f_\beta.
\end{equation}
Thus, $e_\gamma\in U^+_{\n^+}(\g)$ and $f_\gamma\in U^-_{\n^-}(\g)$, which completes
our inductive step. Part (a) follows.

(b) Due to the formulas~(\ref{root basis e},~\ref{root basis f}), it suffices
to prove $\chi^+(e_\gamma)=0$ and $\chi^-(f_\gamma)=0$ for $\gamma=\alpha+\beta$
with $\alpha=\alpha_i, \beta=\alpha_j$.

In the former case, we get
\begin{equation*}
  e_\gamma=
  v^{-\sd_j\n^+_{ij}}e_ie_j-v^{b_{ij}-\sd_i\n^+_{ji}}e_je_i=
  v^{-\sd_j\n^+_{ij}}[e_i,e_j],
\end{equation*}
since $\sd_j\n^+_{ij}-\sd_i\n^+_{ji}=\epsilon^+_{ij}b_{ij}=-b_{ij}$ as
$\alpha_i\prec_+ \alpha_j$. Hence,
$\chi^+(e_\gamma)=v^{-\sd_j\n^+_{ij}}[\chi^+(e_i),\chi^+(e_j)]=0$.

In the latter case, we get
\begin{equation*}
  f_\gamma=
  v^{\sd_j\n^-_{ij}}f_jf_i-v^{-b_{ij}+\sd_i\n^-_{ji}}f_if_j=
  v^{\sd_j\n^-_{ij}}[f_j,f_i],
\end{equation*}
since $\sd_j\n^-_{ij}-\sd_i\n^-_{ji}=\epsilon^-_{ij}b_{ij}=-b_{ij}$ as
$\alpha_i\prec_- \alpha_j$. Thus,
$\chi^-(f_\gamma)=v^{\sd_j\n^-_{ij}}[\chi^-(f_j),\chi^-(f_i)]=0$.
\end{proof}

Tracing back the definition of $\tilde{\D}_V$, Lemma~\ref{nonessential} implies
that $R_\gamma,R_\gamma^\op$ give trivial contributions to $\tilde{\D}_V$ unless
$\gamma\in \Delta_+$ is a simple root, cf.~\cite[Lemma 5.2]{E}
and~\cite[Proposition 3.6]{FFJMM}. Hence, the equality $\tilde{\D}_V=\bar{\D}_V$.
\end{proof}


\subsection{Explicit formulas and classification in type $A_{n-1}$}\label{explicit hamiltonians A}
\

Recall explicit formulas for the action of $U_v(\ssl_n)$ on its first fundamental
representation $V_1$. The space $V_1$ has a basis $\{w_1,\ldots,w_n\}$, in which
the action is given by the following formulas:
\begin{equation*}
  E_i(w_j)=\delta_{j,i+1}w_{j-1},\ F_i(w_j)=\delta_{j,i}w_{j+1},\
  L_i(w_j)=v^{-\frac{i}{n}+\delta_{j\leq i}}w_j,\
   K_i(w_j)=v^{\delta_{j,i}-\delta_{j,i+1}}w_j
\end{equation*}
for any $1\leq i<n, 1\leq j\leq n$. Let $\varpi_1,\ldots,\varpi_n$ be the weights of
$w_1,\ldots,w_n$, respectively, so that $(\varpi_i,\varpi_j)=\delta_{i,j}-1/n$.
Recall that the simple roots are given by
$\alpha_i=\varpi_i-\varpi_{i+1}\ (1\leq i\leq n-1)$, while
$\rho=\sum_{j=1}^n \frac{n+1-2j}{2}\varpi_j$. According to
Proposition~\ref{reduced central elt}, to compute $\D_1$ explicitly, we should:

$\bullet$
evaluate $C'_{V_1}$,

$\bullet$
replace $E_i,F_i$ by $e_i,f_i$ and $L_p$, moving the latter to the middle part,

$\bullet$
apply $\chi^\pm$ as in~\cite{E} to obtain the difference operator
$\tilde{\D}_1=\bar{\D}_{V_1}$,

$\bullet$
conjugate by $e^\rho$.

Note that the operators $\{E_i^r,F_i^r\}^{r>1}_{1\leq i\leq n-1}$ act trivially on $V_1$.
Hence, applying formula~(\ref{C'}), we can replace $R_{\alpha_i}$ by
  $\bar{R}_{\alpha_i}:=1+(v-v^{-1})E_i\otimes F_i$.
Let us now compute all the non-zero terms contributing to $C'_{V_1}$:

$\bullet$
Picking $1$ out of each $\bar{R}^\op_{\alpha_i}, \bar{R}_{\alpha_i}$, we recover
  $\sum_{j=1}^n v^{n+1-2j}\cdot K_{2\varpi_j}$.

$\bullet$
Picking non-trivial terms only at $\bar{R}_{\alpha_j}^\op,\bar{R}_{\alpha_i}$,
the result does not depend on $\Or^\pm$ (hence, the orderings $\prec_\pm$) and
the total contribution is
  $\sum_{i=1}^{n-1} (v-v^{-1})^2 v^{n+1-2i}\cdot F_iK_{\varpi_{i+1}}E_i K_{\varpi_i}$.
Rewriting in terms of $e_i,f_i$ and $L_p$, we get
  $(v-v^{-1})^2\sum_{i=1}^{n-1} v^{n-2i+(\n^+_{ii}-\n^-_{ii})}\cdot
   f_i K_{\varpi_i+\varpi_{i+1}}\prod_{p=1}^{n-1} L_p^{\n^-_{ip}-\n^+_{ip}} e_i$.

$\bullet$
The computation of the remaining terms is based on the following obvious formulas:
\begin{align*}
  & F_{i_k}\cdots F_{i_2}F_{i_1}(w_i)=\delta_{i_1,i}\delta_{i_2,i_1+1}\cdots\delta_{i_k,i_{k-1}+1}w_{i+k},\\
  & E_{j_k}\cdots E_{j_2}E_{j_1}(w_j)=\delta_{j_1,j-1}\delta_{j_2,j_1-1}\cdots\delta_{j_k,j_{k-1}-1}w_{j-k}.
\end{align*}
Hence, picking non-trivial terms only at
  $\bar{R}^\op_{\alpha_{j_1}},\ldots, \bar{R}^\op_{\alpha_{j_{k'}}},
   \bar{R}_{\alpha_{i_k}},\ldots,\bar{R}_{\alpha_{i_1}}$
(in the order listed) is possible only if
  $i_k\prec_+ \cdots \prec_+ i_1,\ j_1\prec_- \cdots \prec_- j_{k'}$,
and  gives a non-zero contribution to $C'_{V_1}$ if and only if
$k=k',\ i_k=i_{k-1}+1=\ldots=i_1+k-1$, and $i_a=j_a$ for $1\leq a\leq k$.
Thus, the remaining terms contributing to $C'_{V_1}$ depend on $\prec_\pm$
(only on $\Or^\pm$) and give in total
\begin{equation*}
  \sum_{1\leq i<j-1\leq n-1}^{\epsilon^\pm_{i,i+1}=\ldots=\epsilon^\pm_{j-2,j-1}=\pm 1}
  (v-v^{-1})^{2(j-i)}v^{n+1-2i}\cdot F_i\cdots F_{j-1}K_{\varpi_j}E_{j-1}\cdots E_i K_{\varpi_i}.
\end{equation*}
Rewriting this in terms of $e_i,f_i$ and $L_p$, and moving the latter to the middle, we get
\begin{equation*}
  \sum (v-v^{-1})^{2(j-i)}v^{n-2i+\sum_{i\leq a\leq b\leq j-1}(\n^+_{ab}-\n^-_{ab})}\cdot
  f_i\cdots f_{j-1}\cdot K_{\varpi_i+\varpi_j}\prod_{p=1}^{n-1} L_p^{\sum_{s=i}^{j-1} (\n^-_{sp}-\n^+_{sp})}\cdot
  e_{j-1}\cdots e_i,
\end{equation*}
where the sum is over all $1\leq i<j-1\leq n-1$ such that
  $\epsilon^\pm_{i,i+1}=\ldots=\epsilon^\pm_{j-2,j-1}=\pm 1$.
Note that $L_p=K_{\varpi_1+\ldots+\varpi_p}$.
Set $\m_{ik}:=\sum_{p=k}^{n-1} (\n^-_{ip}-\n^+_{ip})$.
Then, the Cartan part above equals
  $K_{\varpi_i+\varpi_j}\prod_{p=1}^{n-1} L_p^{\sum_{s=i}^{j-1} (\n^-_{sp}-\n^+_{sp})}=
   K_{\sum_{k=1}^n (\sum_{s=i}^{j-1} \m_{sk}+\delta_{k,i}+\delta_{k,j})\varpi_k}$.

We have listed all the non-zero terms contributing to $C'_{V_1}$. To obtain
the desired difference operator $\tilde{\D}_1$, apply the characters $\chi^\pm$
with $\chi^+(e_i)=c^+_i, \chi^-(f_i)=c^-_i$ as in~\cite[Lemma 5.2]{E}.
Set $b_i:=(v-v^{-1})^2v^{\n^+_{ii}-\n^-_{ii}}c^+_ic^-_i$. Then, we have
\begin{multline}\label{D1 for A-type}
  \tilde{\D}_1=\sum_{j=1}^n v^{n+1-2j} T_{2\varpi_j}+
  \sum_{i=1}^{n-1} b_iv^{n-2i}\cdot e^{-\alpha_i}T_{\sum_{k=1}^n (\m_{ik}+\delta_{k,i}+\delta_{k,i+1})\varpi_k}+\\
  \sum_{1\leq i<j-1\leq n-1}^{\epsilon^\pm_{i,i+1}=\ldots=\epsilon^\pm_{j-2,j-1}=\pm 1}
  b_i\cdots b_{j-1} v^{n-2i+\sum_{i\leq a<b\leq j-1} (\n^+_{ab}-\n^-_{ab})}\times\\
  e^{-\sum_{s=i}^{j-1} \alpha_s}T_{\sum_{k=1}^n (\sum_{s=i}^{j-1} \m_{sk}+\delta_{k,i}+\delta_{k,j})\varpi_k}.
\end{multline}

Conjugating this by $e^\rho$, we finally obtain the explicit formula for the first
hamiltonian $\D_1$ of the type $A_{n-1}$ modified quantum difference Toda system:
\begin{multline}\label{FinalD for A-type}
  \D_1=\sum_{j=1}^n T_{2\varpi_j}+
  \sum_{i=1}^{n-1} b_iv^{\sum_{k=1}^n \frac{2k-n-1}{2}\m_{ik}}\cdot e^{-\alpha_i}T_{\sum_{k=1}^n (\m_{ik}+\delta_{k,i}+\delta_{k,i+1})\varpi_k}+\\
  \sum_{1\leq i<j-1\leq n-1}^{\epsilon^\pm_{i,i+1}=\ldots=\epsilon^\pm_{j-2,j-1}=\pm 1}
  b_i\cdots b_{j-1} v^{j-i-1+\sum_{i\leq a<b\leq j-1} (\n^+_{ab}-\n^-_{ab})+\sum_{k=1}^n\sum_{s=i}^{j-1}\frac{2k-n-1}{2}\m_{sk}}\times\\
  e^{-\sum_{s=i}^{j-1} \alpha_s}T_{\sum_{k=1}^n (\sum_{s=i}^{j-1} \m_{sk}+\delta_{k,i}+\delta_{k,j})\varpi_k}.
\end{multline}

\begin{Rem}\label{Etingof's formula for A}
If $\epsilon^+=\epsilon^-$, then the last sum is vacuous. If we also set $\n^+=\n^-$
and $c^\pm_i=\pm 1$ for all $i$, then we recover the formula~\cite[(5.7)]{E}
for the first hamiltonian of the type $A_{n-1}$ quantum difference Toda lattice:
\begin{equation}\label{standard qToda for A}
  \D_1=\sum_{j=1}^n T_{2\varpi_j}-(v-v^{-1})^2\sum_{i=1}^{n-1} e^{-\alpha_i}T_{\varpi_i+\varpi_{i+1}}.
\end{equation}
\end{Rem}

Let $\CA_n$ be the associative $\BC(v^{1/\NN})$-algebra generated by
$\{\sw^{\pm 1}_j, \sD^{\pm 1}_j\}_{j=1}^n$ with the defining relations
\begin{equation}\label{Algebra A}
  [\sw_i,\sw_j]=[\sD_i,\sD_j]=0,\
  \sw^{\pm 1}_i\sw^{\mp 1}_i=\sD^{\pm 1}_i\sD^{\mp 1}_i=1,\
  \sD_i\sw_j=v^{\delta_{i,j}}\sw_j\sD_i.
\end{equation}
Define $\bar{\CA}_n$ as the quotient of the $\BC(v^{1/\NN})$-subalgebra generated by
  $\{\sw_j^{\pm 1}, (\sD_i/\sD_{i+1})^{\pm 1}\}_{1\leq i<n}^{1\leq j\leq n}$
by the relation $\sw_1\cdots \sw_n=1$. Consider the anti-isomorphism from
the algebra $\bar{\CA}_n$ to the algebra $\CalD_v(H^\ad_{\ssl_n})$ of
Section~\ref{section Whittaker functions}, sending
  $\sw_j\mapsto T_{-\varpi_j}, \sD_i/\sD_{i+1}\mapsto e^{-\alpha_i}$.
Then, the hamiltonian $\D_1$ is the image of the following element
$\HH=\HH(\epsilon^\pm, \n^\pm, c^\pm)$ of $\bar{\CA}_n$:
\begin{multline}\label{H for A-type}
  \HH(\epsilon^\pm, \n^\pm, c^\pm)=
  \sum_{j=1}^n \sw_j^{-2} +
  \sum_{i=1}^{n-1} b_i v^{\sum_{k=1}^n \frac{2k-n-1}{2}\m_{ik}}\cdot
  \prod_{k=1}^n \sw_k^{-\m_{ik}-\delta_{k,i}-\delta_{k,i+1}}\cdot \frac{\sD_i}{\sD_{i+1}}+\\
  \sum_{1\leq i<j-1\leq n-1}^{\epsilon^\pm_{i,i+1}=\ldots=\epsilon^\pm_{j-2,j-1}=\pm 1}
  b_i\cdots b_{j-1} v^{j-i-1+\sum_{i\leq a<b\leq j-1} (\n^+_{ab}-\n^-_{ab})+
  \sum_{k=1}^n\sum_{s=i}^{j-1}\frac{2k-n-1}{2}\m_{sk}}\times\\
  \prod_{k=1}^n \sw_k^{-\sum_{s=i}^{j-1} \m_{sk}-\delta_{k,i}-\delta_{k,j}}\cdot \frac{\sD_i}{\sD_j}.
\end{multline}

The following is the key property of $\HH(\epsilon^\pm,\n^\pm,c^\pm)$ in type $A$:

\begin{Prop}\label{bound for A}
$\HH(\epsilon^\pm,\n^\pm,c^\pm)$ depends only on
$\vec{\epsilon}=(\epsilon_{n-2},\ldots,\epsilon_1)\in \{-1,0,1\}^{n-2}$
with $\epsilon_i:=\frac{\epsilon^+_{i,i+1}-\epsilon^-_{i,i+1}}{2}$,
up to algebra automorphisms of $\bar{\CA}_n$.
\end{Prop}

This result implies that given two pairs of Sevostyanov triples
$(\epsilon^\pm,\n^\pm,c^\pm)$ and $(\tilde{\epsilon}^\pm,\tilde{\n}^\pm,\tilde{c}^\pm)$ with
  $\epsilon^+_{i,i+1}-\epsilon^-_{i,i+1}=\tilde{\epsilon}^+_{i,i+1}-\tilde{\epsilon}^-_{i,i+1}$,
there exists an algebra automorphism of $\CalD_v(H^\ad_{\ssl_n})$ which maps the first hamiltonian
$\D_1(\epsilon^\pm,\n^\pm,c^\pm)$ to $\D_1(\tilde{\epsilon}^\pm,\tilde{\n}^\pm,\tilde{c}^\pm)$.
As we will see in Appendix~\ref{Proof of Main Theorem A}, the same automorphism maps the modified
quantum Toda system $\CT(\epsilon^\pm,\n^\pm,c^\pm)$ to $\CT(\tilde{\epsilon}^\pm,\tilde{\n}^\pm,\tilde{c}^\pm)$.

We present the proof of Proposition~\ref{bound for A} in Appendix~\ref{Proof of bound for A}.


\subsection{Explicit formulas and classification in type $C_n$}\label{explicit hamiltonians C}
\

Recall explicit formulas for the action of $U_v(\mathfrak{sp}_{2n})$ on its
first fundamental representation $V_1$. The space $V_1$ has a basis
$\{w_1,\ldots,w_n,w_{n+1},\ldots,w_{2n}\}$, in which the action is given via
\begin{align*}
  &  E_i(w_j)=\delta_{j,i+1}w_{j-1},\ E_i(w_{n+j})=\delta_{j,i}w_{n+j+1},\ E_n(w_j)=0,\ E_n(w_{n+j})=\delta_{j,n}w_n,\\
  &  F_i(w_j)=\delta_{j,i}w_{j+1},\ F_i(w_{n+j})=\delta_{j,i+1}w_{n+j-1},\ F_n(w_j)=\delta_{j,n}w_{2n},\  F_n(w_{n+j})=0,\\
  &  L_i(w_j)=v^{\delta_{j\leq i}}w_j,\  L_i(w_{n+j})=v^{-\delta_{j\leq i}} w_j,\ L_n(w_j)=vw_j,\ L_n(w_{n+j})=v^{-1}w_{n+j},\\
  &  K_i(w_j)=v^{\delta_{j,i}-\delta_{j,i+1}}w_j,\ K_i(w_{n+j})=v^{-\delta_{j,i}+\delta_{j,i+1}}w_{n+j},\\
  &  K_n(w_j)=v^{2\delta_{j,n}}w_j,\ K_n(w_{n+j})=v^{-2\delta_{j,n}}w_{n+j}
\end{align*}
for any $1\leq i<n, 1\leq j\leq n$. Let $\varpi_j$ be the weight of $w_j\ (1\leq j\leq n)$,
so that the weight of $w_{n+j}$ equals $-\varpi_j$, while $(\varpi_i,\varpi_j)=\delta_{i,j}$.
Recall that the simple positive roots are given by $\alpha_i=\varpi_i-\varpi_{i+1}\ (1\leq i\leq n-1)$
and $\alpha_n=2\varpi_n$, while $\rho=\sum_{i=1}^n (n+1-i)\varpi_i$ and
$\sd_1=\ldots=\sd_{n-1}=1, \sd_n=2$.

To compute $\D_1$ explicitly, we use the same strategy as in type $A$. Note that
the operators $\{E_i^r,F_i^r\}^{r>1}_{1\leq i\leq n}$ act trivially on $V_1$.
Therefore, applying formula~(\ref{C'}), we can replace $R_{\alpha_i}$ by
  $\bar{R}_{\alpha_i}:=1+(v_i-v_i^{-1})E_i\otimes F_i$.
Let us now compute all the non-zero terms contributing to $C'_{V_1}$:

$\bullet$
Picking $1$ out of each $\bar{R}^\op_{\alpha_i}, \bar{R}_{\alpha_i}$, we recover
$\sum_{i=1}^n \left(v^{2(n+1-i)}\cdot K_{2\varpi_i}+v^{-2(n+1-i)}\cdot K_{-2\varpi_i}\right)$.

$\bullet$
Picking non-trivial terms only at $\bar{R}_{\alpha_j}^\op, \bar{R}_{\alpha_i}$, the result
does not depend on $\Or^\pm$ (hence, the orderings $\prec_\pm$) and the total contribution
of the non-zero terms equals
\begin{align*}
  & \sum_{i=1}^{n-1} (v-v^{-1})^2\left(v^{2(n+1-i)}F_iK_{\varpi_{i+1}}E_iK_{\varpi_i}+
    v^{-2(n-i)}F_iK_{-\varpi_i}E_iK_{-\varpi_{i+1}}\right)+\\
  & (v^2-v^{-2})^2v^2F_nK_{-\varpi_n}E_nK_{\varpi_n}.
\end{align*}

$\bullet$
The other terms contributing to $C'_{V_1}$ depend on $\prec_\pm$ (only on $\Or^\pm$).
Picking non-trivial terms only at
  $\bar{R}^\op_{\alpha_{j_1}},\ldots, \bar{R}^\op_{\alpha_{j_{k'}}},
   \bar{R}_{\alpha_{i_k}},\ldots,\bar{R}_{\alpha_{i_1}}$
(in the order listed) is possible only if
  $i_k\prec_+ \cdots \prec_+ i_1,\ j_1\prec_- \cdots \prec_- j_{k'}$,
and gives a non-zero contribution to $C'_{V_1}$ if and only if
$k=k', i_k=i_{k-1}\pm 1=\ldots=i_1\pm (k-1)$ (the sign stays the same everywhere),
and $i_a=j_a$ for $1\leq a\leq k$. When computing these contributions,
we shall distinguish between the two cases: $\max(i_1,i_k)=n$ and $\max(i_1,i_k)<n$.
The total contribution of such terms with $k>1$ equals
\begin{align*}
  &  \sum_{1\leq i<j<n}^{\epsilon^\pm_{i,i+1}=\ldots=\epsilon^\pm_{j-1,j}=\pm 1}
  (v-v^{-1})^{2(j-i+1)}v^{2(n+1-i)}\cdot F_i\cdots F_{j}K_{\varpi_{j+1}}E_{j}\cdots E_i K_{\varpi_i}+\\
  &  \sum_{1\leq i<n}^{\epsilon^\pm_{i,i+1}=\ldots=\epsilon^\pm_{n-1,n}=\pm 1}
  (v-v^{-1})^{2(n-i)}(v^2-v^{-2})^{2}v^{2(n+1-i)}\cdot F_i\cdots F_{n}K_{-\varpi_n}E_{n}\cdots E_i K_{\varpi_i}+\\
  &  \sum_{1\leq i<j<n}^{\epsilon^\pm_{i,i+1}=\ldots=\epsilon^\pm_{j-1,j}=\mp 1}
  (v-v^{-1})^{2(j-i+1)}v^{-2(n-j)}\cdot F_j\cdots F_{i}K_{-\varpi_{i}}E_{i}\cdots E_j K_{-\varpi_{j+1}}+\\
  &  \sum_{1\leq i<n}^{\epsilon^\pm_{i,i+1}=\ldots=\epsilon^\pm_{n-1,n}=\mp 1}
  (v-v^{-1})^{2(n-i)}(v^2-v^{-2})^{2}v^{2}\cdot F_n\cdots F_{i}K_{-\varpi_i}E_{i}\cdots E_n K_{\varpi_n}.
\end{align*}

We have listed all the non-zero terms contributing to $C'_{V_1}$. To obtain
$\tilde{\D}_1=\bar{\D}_{V_1}$, we should rewrite the above formulas via $e_i,f_i$
and $L_p=K_{\varpi_1+\ldots+\varpi_p}\ (1\leq p\leq n)$, moving all the Cartan terms to
the middle, and then apply the characters $\chi^\pm$ with $\chi^+(e_i)=c^+_i, \chi^-(f_i)=c^-_i$.
Conjugating further by $e^\rho$, we obtain the explicit formula for the first hamiltonian
$\D_1$ of the type $C_n$ modified quantum difference Toda system. To write it down,
define constants $b_i,\m_{ik}$ via $b_i:=(v_i-v_i^{-1})^2v_i^{\n^+_{ii}-\n^-_{ii}}c^+_ic^-_i$
and $\m_{ik}:=\sum_{p=k}^n(\n^-_{ip}-\n^+_{ip})$. Then, we have
\begin{multline}\label{FinalD for C-type}
  \D_1=
  \sum_{i=1}^n (T_{2\varpi_i}+T_{-2\varpi_i})+
  b_n v^{\sum_{k=1}^n (k-n-1) \m_{nk}}\cdot e^{-\alpha_n}T_{\sum_{k=1}^n \m_{nk}\varpi_k}+\\
  \sum_{i=1}^{n-1} b_iv^{\sum_{k=1}^n (k-n-1) \m_{ik}}\cdot
  e^{-\alpha_i}\left(T_{\sum_{k=1}^n (\m_{ik}+\delta_{k,i}+\delta_{k,i+1})\varpi_k}+T_{\sum_{k=1}^n (\m_{ik}-\delta_{k,i}-\delta_{k,i+1})\varpi_k}\right)+\\
  \sum_{1\leq i<j<n}^{\epsilon^\pm_{i,i+1}=\ldots=\epsilon^\pm_{j-1,j}=\pm 1}
  b_i\cdots b_j v^{j-i+\sum_{i\leq a<b\leq j} (\n^+_{ab}-\n^-_{ab})+\sum_{k=1}^n\sum_{s=i}^j (k-n-1) \m_{sk}}\times\\
  e^{-\alpha_i-\ldots-\alpha_j}T_{\sum_{k=1}^n(\sum_{s=i}^j \m_{sk}+\delta_{k,i}+\delta_{k,j+1})\varpi_k}+\\
  \sum_{1\leq i<n}^{\epsilon^\pm_{i,i+1}=\ldots=\epsilon^\pm_{n-1,n}=\pm 1}
  b_i\cdots b_n v^{n+1-i+\sum_{i\leq a<b\leq n}(\n^+_{ab}-\n^-_{ab})(1+\delta_{b,n})+\sum_{k=1}^n\sum_{s=i}^n (k-n-1) \m_{sk}}\times\\
  e^{-\alpha_i-\ldots-\alpha_n}T_{\sum_{k=1}^n(\sum_{s=i}^n \m_{sk}+\delta_{k,i}-\delta_{k,n})\varpi_k}+\\
  \sum_{1\leq i<j<n}^{\epsilon^\pm_{i,i+1}=\ldots=\epsilon^\pm_{j-1,j}=\mp 1}
  b_i\cdots b_j v^{j-i+\sum_{i\leq a<b\leq j}(\n^+_{ba}-\n^-_{ba})+\sum_{k=1}^n\sum_{s=i}^j (k-n-1) \m_{sk}}\times\\
  e^{-\alpha_i-\ldots-\alpha_j}T_{\sum_{k=1}^n(\sum_{s=i}^j \m_{sk}-\delta_{k,i}-\delta_{k,j+1})\varpi_k}+\\
  \sum_{1\leq i<n}^{\epsilon^\pm_{i,i+1}=\ldots=\epsilon^\pm_{n-1,n}=\mp 1}
  b_i\cdots b_n v^{n+1-i+\sum_{i\leq a<b\leq n} (\n^+_{ba}-\n^-_{ba})+\sum_{k=1}^n\sum_{s=i}^n (k-n-1) \m_{sk}}\times\\
  e^{-\alpha_i-\ldots-\alpha_n}T_{\sum_{k=1}^n(\sum_{s=i}^n \m_{sk}-\delta_{k,i}+\delta_{k,n})\varpi_k}.
\end{multline}

\begin{Rem}\label{Etingof's formula for C}
If $\epsilon^+=\epsilon^-$, then the last four sums are vacuous. If we also set
$\n^+=\n^-$ and $c^\pm_i=\pm 1$ for all $i$, then we obtain the formula for the first
hamiltonian of the type $C_n$ quantum difference Toda lattice as defined in~\cite{E}
(we write down this formula as we could not find it in the literature, even though
it can be derived completely analogously to~\cite[(5.7)]{E}, cf.~\cite[the end of Section 3]{FFJMM}):
\begin{equation}\label{standard qToda for C}
  \D_1=
  \sum_{i=1}^n (T_{2\varpi_i}+T_{-2\varpi_i})-
  (v-v^{-1})^2\sum_{i=1}^{n-1}  e^{-\alpha_i}\left(T_{\varpi_i+\varpi_{i+1}}+T_{-\varpi_i-\varpi_{i+1}}\right)-
  (v^2-v^{-2})^2 e^{-\alpha_n}.
\end{equation}
\end{Rem}

Let $\CalC_n$ be the $\BC(v^{1/\NN})$-subalgebra of $\CA_n$ generated by
  $\{\sw^{\pm 1}_j, (\sD_i/\sD_{i+1})^{\pm 1}, \sD_n^{\pm 2}\}_{1\leq i<n}^{1\leq j\leq n}$.\footnote{Note
that $\CalC_n$ can be abstractly defined as the associative algebra generated by
$\{\tilde{\sw}_i^{\pm 1},\tilde{\sD}_i^{\pm 1}\}_{i=1}^n$ with the defining relations
  $[\tilde{\sw}_i,\tilde{\sw}_j]=[\tilde{\sD}_i,\tilde{\sD}_j]=0,
   \tilde{\sw}^{\pm 1}_i\tilde{\sw}^{\mp 1}_i=\tilde{\sD}^{\pm 1}_i\tilde{\sD}^{\mp 1}_i=1,
   \tilde{\sD}_i\tilde{\sw}_j=v^{\delta_{i,j}\sd_i}\tilde{\sw}_j\tilde{\sD}_i,$
where $\sd_i=1+\delta_{i,n}$.}
Consider the anti-isomorphism from $\CalC_n$ to the algebra $\CalD_v(H^\ad_{\mathfrak{sp}_{2n}})$
of Section~\ref{section Whittaker functions}, sending
  $\sw_j\mapsto T_{-\varpi_j}, \sD_i/\sD_{i+1}\mapsto e^{-\alpha_i}, \sD_n^2\mapsto e^{-\alpha_n}$.
Then, the hamiltonian $\D_1$ is the image of the following element
$\HH=\HH(\epsilon^\pm, \n^\pm, c^\pm)$ of $\CalC_n$:
\begin{multline}\label{H for C-type}
  \HH(\epsilon^\pm, \n^\pm, c^\pm)=
  \sum_{i=1}^n (\sw_i^{-2}+\sw_i^2)+\\
  b_n v^{\sum_{k=1}^n (k-n-1) \m_{nk}} \cdot \prod_{k=1}^n \sw_k^{-\m_{nk}}\cdot \sD_n^2+\\
  \sum_{i=1}^{n-1} b_iv^{\sum_{k=1}^n (k-n-1) \m_{ik}}\cdot
  \left(\prod_{k=1}^n \sw_k^{-\m_{ik}-\delta_{k,i}-\delta_{k,i+1}}+
        \prod_{k=1}^n \sw_k^{-\m_{ik}+\delta_{k,i}+\delta_{k,i+1}}\right)\cdot\frac{\sD_i}{\sD_{i+1}}+\\
  \sum_{1\leq i<j<n}^{\epsilon^\pm_{i,i+1}=\ldots=\epsilon^\pm_{j-1,j}=\pm 1}
  b_i\cdots b_j v^{j-i+\sum_{i\leq a<b\leq j}(\n^+_{ab}-\n^-_{ab})+\sum_{k=1}^n\sum_{s=i}^j (k-n-1) \m_{sk}}\times\\
  \prod_{k=1}^n \sw_k^{-\sum_{s=i}^j \m_{sk}-\delta_{k,i}-\delta_{k,j+1}}\cdot \frac{\sD_i}{\sD_{j+1}}+\\
  \sum_{1\leq i<j<n}^{\epsilon^\pm_{i,i+1}=\ldots=\epsilon^\pm_{j-1,j}=\mp 1}
  b_i\cdots b_j v^{j-i+\sum_{i\leq a<b\leq j}(\n^+_{ba}-\n^-_{ba})+\sum_{k=1}^n\sum_{s=i}^j (k-n-1) \m_{sk}}\times\\
  \prod_{k=1}^n \sw_k^{-\sum_{s=i}^j \m_{sk}+\delta_{k,i}+\delta_{k,j+1}}\cdot\frac{\sD_i}{\sD_{j+1}}+\\
  \sum_{1\leq i<n}^{\epsilon^\pm_{i,i+1}=\ldots=\epsilon^\pm_{n-1,n}=\pm 1}
  b_i\cdots b_n v^{n+1-i+\sum_{i\leq a<b\leq n}(\n^+_{ab}-\n^-_{ab})(1+\delta_{b,n})+\sum_{k=1}^n\sum_{s=i}^n (k-n-1) \m_{sk}}\times\\
  \prod_{k=1}^n \sw_k^{-\sum_{s=i}^n \m_{sk}-\delta_{k,i}+\delta_{k,n}}\cdot\sD_i\sD_n+\\
  \sum_{1\leq i<n}^{\epsilon^\pm_{i,i+1}=\ldots=\epsilon^\pm_{n-1,n}=\mp 1}
  b_i\cdots b_n v^{n+1-i+\sum_{i\leq a<b\leq n} (\n^+_{ba}-\n^-_{ba})+\sum_{k=1}^n\sum_{s=i}^n (k-n-1) \m_{sk}}\times\\
  \prod_{k=1}^n \sw_k^{-\sum_{s=i}^n \m_{sk}+\delta_{k,i}-\delta_{k,n}}\cdot\sD_i\sD_n.
\end{multline}

The following is the key property of $\HH(\epsilon^\pm,\n^\pm,c^\pm)$ in type $C$:

\begin{Prop}\label{bound for C}
$\HH(\epsilon^\pm,\n^\pm,c^\pm)$ depends only on
$\vec{\epsilon}=(\epsilon_{n-1},\ldots,\epsilon_1)\in \{-1,0,1\}^{n-1}$
with $\epsilon_i:=\frac{\epsilon^+_{i,i+1}-\epsilon^-_{i,i+1}}{2}$,
up to algebra automorphisms of $\CalC_n$.
\end{Prop}

The proof of this result is completely analogous to that of Proposition~\ref{bound for A}
given in Appendix~\ref{Proof of bound for A}, see also Remark~\ref{Remark on other types},
we leave details to the interested reader. Proposition~\ref{bound for C} implies that given
two pairs of Sevostyanov triples $(\epsilon^\pm,\n^\pm,c^\pm)$ and
$(\tilde{\epsilon}^\pm,\tilde{\n}^\pm,\tilde{c}^\pm)$ with
  $\epsilon^+_{i,i+1}-\epsilon^-_{i,i+1}=\tilde{\epsilon}^+_{i,i+1}-\tilde{\epsilon}^-_{i,i+1}$,
there exists an algebra automorphism of $\CalD_v(H^\ad_{\mathfrak{sp}_{2n}})$ which maps
the first hamiltonian $\D_1(\epsilon^\pm,\n^\pm,c^\pm)$ to
$\D_1(\tilde{\epsilon}^\pm,\tilde{\n}^\pm,\tilde{c}^\pm)$. As we will see in
Appendix~\ref{Proof of Main Theorem A}, the same automorphism maps the modified quantum Toda system
$\CT(\epsilon^\pm,\n^\pm,c^\pm)$ to $\CT(\tilde{\epsilon}^\pm,\tilde{\n}^\pm,\tilde{c}^\pm)$.


\subsection{Explicit formulas and classification in type $D_n$}\label{explicit hamiltonians D}
\

Recall explicit formulas for the action of $U_v(\mathfrak{so}_{2n})$ on its
first fundamental representation $V_1$. The space $V_1$ has a basis
$\{w_1,\ldots,w_{2n}\}$, in which the action is given via
\begin{align*}
  &  E_i(w_{2j-1})=\delta_{j,i+1}w_{2j-3},\ E_i(w_{2j})=\delta_{j,i}w_{2j+2},\\
  &  F_i(w_{2j-1})=\delta_{j,i}w_{2j+1},\ F_i(w_{2j})=\delta_{j,i+1}w_{2j-2},\\
  &  L_p(w_{2j-1})=v^{\delta_{j\leq p}}w_{2j-1}, \ L_p(w_{2j})=v^{-\delta_{j\leq p}}w_{2j},\\
  &  L_{n-1}(w_{2j-1})=v^{\frac{1}{2}-\delta_{j,n}}w_{2j-1}, \ L_{n-1}(w_{2j})=v^{-\frac{1}{2}+\delta_{j,n}}w_{2j},\\
  &  K_i(w_{2j-1})=v^{\delta_{j,i}-\delta_{j,i+1}}w_{2j-1},\ K_i(w_{2j})=v^{-\delta_{j,i}+\delta_{j,i+1}}w_{2j},\\
  &  E_n(w_{2j})=\delta_{j,n-1}w_{2n-1}+\delta_{j,n}w_{2n-3},\ F_n(w_{2j-1})=\delta_{j,n-1}w_{2n}+\delta_{j,n}w_{2n-2},\\
  &  E_n(w_{2j-1})=0,\ F_n(w_{2j})=0,\
     L_n(w_{2j-1})=v^{\frac{1}{2}} w_{2j-1},\ L_n(w_{2j})=v^{-\frac{1}{2}}w_{2j},\\
  &  K_n(w_{2j-1})=v^{\delta_{j,n}+\delta_{j,n-1}}w_{2j-1},\ K_n(w_{2j})=v^{-\delta_{j,n}-\delta_{j,n-1}}w_{2j}
\end{align*}
for any $1\leq p\leq n-2, 1\leq i<n, 1\leq j\leq n$. Let $\varpi_j$ be the weight of
$w_{2j-1}\ (1\leq j\leq n)$, so that the weight of $w_{2j}$ equals $-\varpi_j$,
while $(\varpi_i,\varpi_j)=\delta_{i,j}$. Recall that the simple roots are given by
$\alpha_i=\varpi_i-\varpi_{i+1}\ (1\leq i\leq n-1)$ and $\alpha_n=\varpi_{n-1}+\varpi_n$,
while $\rho=\sum_{i=1}^n (n-i)\varpi_i$ and $\sd_1=\ldots=\sd_{n}=1$.

To compute $\D_1$ explicitly, we use the same strategy as in type $A$.
Similarly to the types $A$ and $C$ treated above, we note that the operators
$\{E_i^r,F_i^r\}^{r>1}_{1\leq i\leq n}$ act trivially on $V_1$, hence,
applying formula~(\ref{C'}), we can replace $R_{\alpha_i}$ by
$\bar{R}_{\alpha_i}:=1+(v-v^{-1})E_i\otimes F_i$.
Let us now compute all the non-zero terms contributing to $C'_{V_1}$:

$\bullet$
Picking $1$ out of each $\bar{R}^\op_{\alpha_i}, \bar{R}_{\alpha_i}$, we recover
$\sum_{i=1}^n \left(v^{2(n-i)}\cdot K_{2\varpi_i}+v^{-2(n-i)}\cdot K_{-2\varpi_i}\right)$.

$\bullet$
Picking non-trivial terms only at $\bar{R}_{\alpha_j}^\op, \bar{R}_{\alpha_i}$,
the result does not depend on $\Or^\pm$ (hence, the orderings $\prec_\pm$) and
the total contribution of the non-zero terms equals
\begin{align*}
  & \sum_{i=1}^{n-1} (v-v^{-1})^2\left(v^{2(n-i)}F_iK_{\varpi_{i+1}}E_iK_{\varpi_i}+
    v^{-2(n-i-1)}F_iK_{-\varpi_i}E_iK_{-\varpi_{i+1}}\right)+\\
  &  (v-v^{-1})^2\left(F_nK_{-\varpi_{n-1}}E_nK_{\varpi_n}+F_nK_{-\varpi_n}E_nK_{\varpi_{n-1}}\right).
\end{align*}

$\bullet$
In contrast to the types $A,C$ considered above, there is one more summand
independent of the orientations. It arises by picking non-trivial terms only at
$\bar{R}^\op_{\alpha_{n-1}},\bar{R}^\op_{\alpha_n}$ and
$\bar{R}_{\alpha_{n-1}}, \bar{R}_{\alpha_n}$
(note that $E_{n-1}E_{n}=E_{n}E_{n-1}, F_{n-1}F_{n}=F_{n}F_{n-1}$,
due to the $v$-Serre relations) and equals
\begin{equation*}
  (v-v^{-1})^4v^2F_nF_{n-1}K_{-\varpi_{n-1}}E_nE_{n-1}K_{\varpi_{n-1}}.
\end{equation*}

$\bullet$
The contribution of the remaining terms to $C'_{V_1}$ depends on $\Or^\pm$.
Tracing back explicit formulas for the action of $U_v(\mathfrak{so}_{2n})$ on $V_1$,
we see that the total sum of such terms equals
\begin{align*}
  &  \sum_{1\leq i<j<n}^{\epsilon^\pm_{i,i+1}=\ldots=\epsilon^\pm_{j-1,j}=\pm 1}
   (v-v^{-1})^{2(j-i+1)}v^{2(n-i)}\cdot F_i\cdots F_{j}K_{\varpi_{j+1}}E_{j}\cdots E_i K_{\varpi_i}+\\
  &  \sum_{1\leq i<n-1}^{\epsilon^\pm_{i,i+1}=\ldots=\epsilon^\pm_{n-3,n-2}=\epsilon^\pm_{n-2,n}=\pm 1}
    (v-v^{-1})^{2(n-i)}v^{2(n-i)}\cdot F_i\cdots F_{n-2}F_nK_{-\varpi_n}E_{n}E_{n-2}\cdots E_i K_{\varpi_i}+\\
  &  \sum_{1\leq i<n-1}^{\epsilon^\pm_{i,i+1}=\ldots=\epsilon^\pm_{n-2,n-1}=\epsilon^\pm_{n-2,n}=\pm 1}
    (v-v^{-1})^{2(n-i+1)}v^{2(n-i)}\cdot F_i\cdots F_{n-1}F_nK_{-\varpi_{n-1}}E_{n}E_{n-1}\cdots E_i K_{\varpi_i}+\\
  &  \sum_{1\leq i<j<n}^{\epsilon^\pm_{i,i+1}=\ldots=\epsilon^\pm_{j-1,j}=\mp 1}
    (v-v^{-1})^{2(j-i+1)}v^{-2(n-j-1)}\cdot F_j\cdots F_{i}K_{-\varpi_{i}}E_{i}\cdots E_j K_{-\varpi_{j+1}}+\\
  &  \sum_{1\leq i<n-1}^{\epsilon^\pm_{i,i+1}=\ldots=\epsilon^\pm_{n-3,n-2}=\epsilon^\pm_{n-2,n}=\mp 1}
    (v-v^{-1})^{2(n-i)}\cdot F_nF_{n-2}\cdots F_{i}K_{-\varpi_i}E_{i}\cdots E_{n-2}E_n K_{\varpi_n}+\\
  &  \sum_{1\leq i<n-1}^{\epsilon^\pm_{i,i+1}=\ldots=\epsilon^\pm_{n-2,n-1}=\epsilon^\pm_{n-2,n}=\mp 1}
    (v-v^{-1})^{2(n-i+1)}v^2\cdot F_nF_{n-1}\cdots F_i K_{-\varpi_i}E_{i}\cdots E_{n-1}E_n K_{\varpi_{n-1}}.
\end{align*}

We have listed all the non-zero terms contributing to $C'_{V_1}$. To obtain
$\tilde{\D}_1=\bar{\D}_{V_1}$, we should rewrite the above formulas via $e_i,f_i$ and
  $L_p=\begin{cases}
     K_{\varpi_1+\ldots+\varpi_p}, & \mathrm{if}\ 1\leq p\leq n-2\\
     K_{\frac{1}{2}(\varpi_1+\ldots+\varpi_{n-1}-\varpi_n)}, & \mathrm{if}\ p=n-1\\
     K_{\frac{1}{2}(\varpi_1+\ldots+\varpi_{n-1}+\varpi_n)}, & \mathrm{if}\ p=n
   \end{cases},$
moving all the Cartan terms to the middle, and then apply the characters
$\chi^\pm$ with $\chi^+(e_i)=c^+_i, \chi^-(f_i)=c^-_i$. Conjugating further
by $e^\rho$, we obtain the explicit formula for the first hamiltonian $\D_1$
of the type $D_n$ modified quantum difference Toda system. To write it down,
define constants
  $\m_{ik}:=\begin{cases}
     \sum_{p=k}^{n-2}(\n^-_{ip}-\n^+_{ip})+\frac{1}{2}(\n^-_{i,n-1}-\n^+_{i,n-1})+\frac{1}{2}(\n^-_{in}-\n^+_{in}), & \mathrm{if}\ k<n\\
     -\frac{1}{2}(\n^-_{i,n-1}-\n^+_{i,n-1})+\frac{1}{2}(\n^-_{in}-\n^+_{in}), & \mathrm{if}\ k=n
   \end{cases}$
and $b_i:=(v-v^{-1})^2v^{\n^+_{ii}-\n^-_{ii}}c^+_ic^-_i$. Then, we have
\begin{multline}\label{FinalD for D-type}
  \D_1=
  \sum_{i=1}^n (T_{2\varpi_i}+T_{-2\varpi_i})+
  b_n v^{\sum_{k=1}^n (k-n) \m_{nk}}\cdot e^{-\alpha_n}T_{\sum_{k=1}^n (\m_{nk}-\delta_{k,n-1}+\delta_{k,n})\varpi_k}+\\
  b_n v^{-2+\sum_{k=1}^n (k-n) \m_{nk}}\cdot e^{-\alpha_n}T_{\sum_{k=1}^n (\m_{nk}+\delta_{k,n-1}-\delta_{k,n})\varpi_k}+\\
  b_{n-1}b_nv^{(\n^+_{n-1,n}-\n^-_{n-1,n})+\sum_{k=1}^n (k-n) (\m_{n-1,k}+\m_{nk})}\cdot e^{-\alpha_{n-1}-\alpha_n}T_{\sum_{k=1}^n (\m_{n-1,k}+\m_{n,k})\varpi_k}+\\
  \sum_{i=1}^{n-1} b_iv^{\sum_{k=1}^n (k-n) \m_{ik}}\cdot
  e^{-\alpha_i}\left(T_{\sum_{k=1}^n (\m_{ik}+\delta_{k,i}+\delta_{k,i+1})\varpi_k}+T_{\sum_{k=1}^n (\m_{ik}-\delta_{k,i}-\delta_{k,i+1})\varpi_k}\right)+\\
  \sum_{1\leq i<j<n}^{\epsilon^\pm_{i,i+1}=\ldots=\epsilon^\pm_{j-1,j}=\pm 1}
  b_i\cdots b_j v^{j-i+\sum_{i\leq a<b\leq j} (\n^+_{ab}-\n^-_{ab})+\sum_{k=1}^n\sum_{s=i}^j (k-n) \m_{sk}}\times\\
  e^{-\alpha_i-\ldots-\alpha_j}T_{\sum_{k=1}^n(\sum_{s=i}^j \m_{sk}+\delta_{k,i}+\delta_{k,j+1})\varpi_k}+\\
  \sum_{1\leq i<n-1}^{\epsilon^\pm_{i,i+1}=\ldots=\epsilon^\pm_{n-2,n}=\pm 1}
  b_i\cdots b_{n-2}b_n v^{n-i-1+\sum_{i\leq a<b\leq n}^{a,b\ne n-1} (\n^+_{ab}-\n^-_{ab})+\sum_{k=1}^n\sum_{i\leq s\leq n}^{s\ne n-1} (k-n) \m_{sk}}\times\\
  e^{-\alpha_i-\ldots-\alpha_{n-2}-\alpha_n}T_{\sum_{k=1}^n(\sum_{i\leq s\leq n}^{s\ne n-1} \m_{sk}+\delta_{k,i}-\delta_{k,n})\varpi_k}+\\
  \sum_{1\leq i<n-1}^{\epsilon^\pm_{i,i+1}=\ldots=\epsilon^\pm_{n-2,n-1}=\epsilon^\pm_{n-2,n}=\pm 1}
  b_i\cdots b_n v^{n-i+\sum_{i\leq a<b\leq n} (\n^+_{ab}-\n^-_{ab})+\sum_{k=1}^n \sum_{s=i}^n (k-n) \m_{sk}}\times\\
  e^{-\alpha_i-\ldots-\alpha_n}T_{\sum_{k=1}^n(\sum_{s=i}^{n} \m_{sk}+\delta_{k,i}-\delta_{k,n-1})\varpi_k}+\\
  \sum_{1\leq i<j<n}^{\epsilon^\pm_{i,i+1}=\ldots=\epsilon^\pm_{j-1,j}=\mp 1}
  b_i\cdots b_j v^{j-i+\sum_{i\leq a<b\leq j}(\n^+_{ba}-\n^-_{ba})+\sum_{k=1}^n \sum_{s=i}^j (k-n) \m_{sk}}\times\\
  e^{-\alpha_i-\ldots-\alpha_j}T_{\sum_{k=1}^n(\sum_{s=i}^j \m_{sk}-\delta_{k,i}-\delta_{k,j+1})\varpi_k}+\\
  \sum_{1\leq i<n-1}^{\epsilon^\pm_{i,i+1}=\ldots=\epsilon^\pm_{n-2,n}=\mp 1}
  b_i\cdots b_{n-2}b_n v^{n-i-1+\sum_{i\leq a<b\leq n}^{a,b\ne n-1} (\n^+_{ba}-\n^-_{ba})+\sum_{k=1}^n \sum_{i\leq s\leq n}^{s\ne n-1} (k-n) \m_{sk}}\times\\
  e^{-\alpha_i-\ldots-\alpha_{n-2}-\alpha_n}T_{\sum_{k=1}^n(\sum_{i\leq s\leq n}^{s\ne n-1} \m_{sk}-\delta_{k,i}+\delta_{k,n})\varpi_k}+\\
  \sum_{1\leq i<n-1}^{\epsilon^\pm_{i,i+1}=\ldots=\epsilon^\pm_{n-2,n-1}=\epsilon^\pm_{n-2,n}=\mp 1}
  b_i\cdots b_n v^{n-i+\sum_{i\leq a<b\leq n} (\n^+_{ba}-\n^-_{ba})+\sum_{k=1}^n \sum_{s=i}^n (k-n) \m_{sk}}\times\\
  e^{-\alpha_i-\ldots-\alpha_n}T_{\sum_{k=1}^n(\sum_{s=i}^n \m_{sk}-\delta_{k,i}+\delta_{k,n-1})\varpi_k}.
\end{multline}

\begin{Rem}\label{Etingof's formula for D}
If $\epsilon^+=\epsilon^-$, then the last six sums are vacuous. If we further set
$\n^+=\n^-$ and $c^\pm_i=\pm 1$ for all $i$, then we obtain the formula for the first
hamiltonian of the type $D_n$ quantum difference Toda lattice as defined in~\cite{E}
(we write down this formula as we could not find it in the literature, even though it
can be derived completely analogously to~\cite[(5.7)]{E}, cf.~\cite[the end of Section 3]{FFJMM}):
\begin{multline}\label{standard qToda for D}
  \D_1=
  \sum_{i=1}^n (T_{2\varpi_i}+T_{-2\varpi_i})-
  (v-v^{-1})^2 \sum_{i=1}^{n-1} e^{-\alpha_i} \left(T_{\varpi_i+\varpi_{i+1}}+T_{-\varpi_i-\varpi_{i+1}}\right)-\\
  (v-v^{-1})^2 e^{-\alpha_n}\left(T_{-\varpi_{n-1}+\varpi_n}+v^{-2} T_{\varpi_{n-1}-\varpi_n}\right)+
  (v-v^{-1})^4 e^{-\alpha_{n-1}-\alpha_n}.
\end{multline}
\end{Rem}

Recall the algebra $\CalC_n$ from Section~\ref{explicit hamiltonians C}. Consider the
anti-isomorphism from $\CalC_n$ to the algebra $\CalD_v(H^\ad_{\mathfrak{so}_{2n}})$
of Section~\ref{section Whittaker functions}, sending
  $\sw_j\mapsto T_{-\varpi_j},
   \sD_i/\sD_{i+1}\mapsto e^{-\alpha_i},
   \sD_n^2\mapsto e^{\alpha_{n-1}-\alpha_n}$.
Let $\HH=\HH(\epsilon^\pm, \n^\pm, c^\pm)$ be the element of $\CalC_n$ which corresponds
to $\D_1$ under this anti-isomorphism (to save space, we omit the explicit long formula for $\HH$).
The following is the key property of $\HH(\epsilon^\pm,\n^\pm,c^\pm)$ in type $D$:

\begin{Prop}\label{bound for D}
$\HH(\epsilon^\pm,\n^\pm,c^\pm)$ depends only on
$\vec{\epsilon}=(\epsilon_{n-1},\ldots,\epsilon_1)\in \{-1,0,1\}^{n-1}$ with
  $\epsilon_i:=\frac{\epsilon^+_{i,i+1}-\epsilon^-_{i,i+1}}{2}\ (1\leq i\leq n-2),
   \epsilon_{n-1}:=\frac{\epsilon^+_{n-2,n}-\epsilon^-_{n-2,n}}{2}$,
up to algebra automorphisms of $\CalC_n$.
\end{Prop}

The proof of this result is completely analogous to that of Proposition~\ref{bound for A}
given in Appendix~\ref{Proof of bound for A}, see also Remark~\ref{Remark on other types},
we leave details to the interested reader. Proposition~\ref{bound for D} implies that
given two pairs of Sevostyanov triples $(\epsilon^\pm,\n^\pm,c^\pm)$ and
$(\tilde{\epsilon}^\pm,\tilde{\n}^\pm,\tilde{c}^\pm)$ with
  $\epsilon^+_{i,i+1}-\epsilon^-_{i,i+1}=
   \tilde{\epsilon}^+_{i,i+1}-\tilde{\epsilon}^-_{i,i+1}\ (1\leq i\leq n-2)$
and
  $\epsilon^+_{n-2,n}-\epsilon^-_{n-2,n}=\tilde{\epsilon}^+_{n-2,n}-\tilde{\epsilon}^-_{n-2,n}$,
there exists an algebra automorphism of $\CalD_v(H^\ad_{\mathfrak{so}_{2n}})$ which maps
the first hamiltonian $\D_1(\epsilon^\pm,\n^\pm,c^\pm)$ to
$\D_1(\tilde{\epsilon}^\pm,\tilde{\n}^\pm,\tilde{c}^\pm)$. As we will see in
Appendix~\ref{Proof of Main Theorem A}, the same automorphism maps the modified quantum Toda system
$\CT(\epsilon^\pm,\n^\pm,c^\pm)$ to $\CT(\tilde{\epsilon}^\pm,\tilde{\n}^\pm,\tilde{c}^\pm)$.


\subsection{Explicit formulas and classification in type $B_n$}\label{explicit hamiltonians B}
\

Recall explicit formulas for the action of $U_v(\mathfrak{so}_{2n+1})$ on
its first fundamental representation $V_1$. The space $V_1$ has a basis
$\{w_0,\ldots,w_{2n}\}$, in which the action is given via
\begin{align*}
  & E_i(w_{2j-1})=\delta_{j,i+1}w_{2j-3},\ E_i(w_{2j})=\delta_{j,i}w_{2j+2},\ E_i(w_0)=0,\\
  & F_i(w_{2j-1})=\delta_{j,i}w_{2j+1},\ F_i(w_{2j})=\delta_{j,i+1}w_{2j-2},\ F_i(w_0)=0,\\
  & L_i(w_{2j-1})=v^{2\delta_{j\leq i}}w_{2j-1},\ L_i(w_{2j})=v^{-2\delta_{j\leq i}}w_{2j},\ L_i(w_0)=w_0,\\
  & K_i(w_{2j-1})=v^{2\delta_{j,i}-2\delta_{j,i+1}}w_{2j-1},\ K_i(w_{2j})=v^{-2\delta_{j,i}+2\delta_{j,i+1}}w_{2j},\ K_i(w_0)=w_0,\\
  & E_n(w_{2j-1})=0,\ E_n(w_{2j})=\delta_{j,n}w_{0},\ E_n(w_0)=w_{2n-1},\\
  & F_n(w_{2j-1})=\delta_{j,n}w_{0},\ F_n(w_{2j})=0,\ F_n(w_0)=w_{2n},\\
  & L_n(w_{2j-1})=vw_{2j-1},\ L_n(w_{2j})=v^{-1}w_{2j},\ L_n(w_0)=w_0,\\
  &  K_n(w_{2j-1})=v^{\delta_{j,n}}w_{2j-1},\ K_n(w_{2j})=v^{-\delta_{j,n}}w_{2j},\ K_n(w_0)=w_0
\end{align*}
for any $1\leq i<n, 1\leq j\leq n$.

Let $\varpi_j$ be the weight of $w_{2j-1}\ (1\leq j\leq n)$, so that
the weight of $w_{2j}$ equals $-\varpi_j$, while $w_0$ has the zero weight.
We note that now $(\varpi_i,\varpi_j)=2\delta_{i,j}$. Recall that the simple
roots are given by $\alpha_i=\varpi_i-\varpi_{i+1}\ (1\leq i\leq n-1)$ and
$\alpha_n=\varpi_n$, while $\rho=\sum_{i=1}^n (n+\frac{1}{2}-i)\varpi_i$ and
$\sd_1=\ldots=\sd_{n-1}=2, \sd_n=1$.

To compute $\D_1$ explicitly, we use the same strategy as in type $A$. In contrast to
the types $A,C,D$ treated above, $E_n^2$ and $F_n^2$ act non-trivially on $V_1$, while
  $\{E_i^r,F_i^r\}^{r>1}_{1\leq i<n}\cup\{E_n^s,F_n^s\}_{s>2}$
still act by zero on $V_1$.
Therefore, applying formula~(\ref{C'}), we can replace $R_{\alpha_i}$ by
  $\bar{R}_{\alpha_i}:=1+(v^2-v^{-2})E_i\otimes F_i$ for $1\leq i<n$
and $R_{\alpha_n}$ by
  $\bar{R}_{\alpha_n}:=1+(v-v^{-1})E_n\otimes F_n+cE_n^2\otimes F_n^2$
for $c:=(1-v^{-1})(v-v^{-1})$. Let us now compute all the non-zero terms
contributing to $C'_{V_1}$:

$\bullet$
Picking $1$ out of each $\bar{R}^\op_i, \bar{R}_i$, we recover
  $1+\sum_{i=1}^n \left(v^{4n+2-4i}\cdot K_{2\varpi_i}+v^{-4n-2+4i}\cdot K_{-2\varpi_i}\right)$.

$\bullet$
Picking non-trivial terms only at $\bar{R}_{\alpha_j}^\op, \bar{R}_{\alpha_i}$,
the result does not depend on $\Or^\pm$ (hence, the orderings $\prec_\pm$) and
the total contribution of the non-zero terms equals
\begin{align*}
  &  \sum_{i=1}^{n-1} (v^2-v^{-2})^2\left(v^{4n+2-4i}F_iK_{\varpi_{i+1}}E_iK_{\varpi_i}+
     v^{-4n+2+4i}F_iK_{-\varpi_i}E_iK_{-\varpi_{i+1}}\right)+\\
  &  (v-v^{-1})^2(F_nK_{-\varpi_n}E_n+v^2F_nE_nK_{\varpi_n})+
     c^2v^2F_n^2K_{-\varpi_n}E_n^2K_{\varpi_n}.
\end{align*}

$\bullet$
The contribution of the remaining terms to $C'_{V_1}$ depends on $\Or^\pm$.
Tracing back explicit formulas for the action of $U_v(\mathfrak{so}_{2n+1})$
on $V_1$, we see that the total sum of such terms equals
\begin{align*}
  & \sum_{1\leq i<j<n}^{\epsilon^\pm_{i,i+1}=\ldots=\epsilon^\pm_{j-1,j}=\pm 1}
    (v^2-v^{-2})^{2(j-i+1)}v^{4n+2-4i}\cdot F_i\cdots F_{j}K_{\varpi_{j+1}}E_{j}\cdots E_i K_{\varpi_i}+\\
  & \sum_{1\leq i<n}^{\epsilon^\pm_{i,i+1}=\ldots=\epsilon^\pm_{n-1,n}=\pm 1}
    (v^2-v^{-2})^{2(n-i)}(v-v^{-1})^2v^{4n+2-4i}\cdot F_i\cdots F_{n}E_{n}\cdots E_i K_{\varpi_i}+\\
  & \sum_{1\leq i<n}^{\epsilon^\pm_{i,i+1}=\ldots=\epsilon^\pm_{n-1,n}=\pm 1}
    c^2(v^2-v^{-2})^{2(n-i)}v^{4n+2-4i}\cdot F_i\cdots F_{n-1}F_n^2K_{-\varpi_n}E_{n}^2E_{n-1}\cdots E_i K_{\varpi_i}+\\
  & \sum_{1\leq i<j<n}^{\epsilon^\pm_{i,i+1}=\ldots=\epsilon^\pm_{j-1,j}=\mp 1}
    (v^2-v^{-2})^{2(j-i+1)}v^{-4n+2+4j}\cdot F_j\cdots F_{i}K_{-\varpi_{i}}E_{i}\cdots E_j K_{-\varpi_{j+1}}+\\
  & \sum_{1\leq i<n}^{\epsilon^\pm_{i,i+1}=\ldots=\epsilon^\pm_{n-1,n}=\mp 1}
    (v^2-v^{-2})^{2(n-i)}(v-v^{-1})^2\cdot F_n\cdots F_{i}K_{-\varpi_i}E_{i}\cdots E_n+\\
  & \sum_{1\leq i<n}^{\epsilon^\pm_{i,i+1}=\ldots=\epsilon^\pm_{n-1,n}=\mp 1}
    c^2(v^2-v^{-2})^{2(n-i)} v^2\cdot F_n^2F_{n-1}\cdots F_i K_{-\varpi_i}E_{i}\cdots E_{n-1}E_n^2 K_{\varpi_{n}}.
\end{align*}

We have listed all the non-zero terms contributing to $C'_{V_1}$. To obtain
$\tilde{\D}_1=\bar{\D}_{V_1}$, we should rewrite the above formulas via $e_i,f_i$ and
  $L_p=\begin{cases}
     K_{\varpi_1+\ldots+\varpi_p}, & \mathrm{if}\ 1\leq p<n\\
     K_{\frac{1}{2}(\varpi_1+\ldots+\varpi_n)}, & \mathrm{if}\ p=n
   \end{cases},$
moving all the Cartan terms to the middle, and then apply the characters $\chi^\pm$
with $\chi^+(e_i)=c^+_i, \chi^-(f_i)=c^-_i$. Conjugating further by $e^\rho$, we obtain
the explicit formula for the first hamiltonian $\D_1$ of the type $B_n$ modified
quantum difference Toda system. To write it down, define constants $b_i,\m_{ik}$ via
  $\m_{ik}:=\sum_{p=k}^n(\n^-_{ip}-\n^+_{ip})(1-\frac{1}{2}\delta_{p,n})$
and
  $b_i:=(v_i-v_i^{-1})^2v_i^{\n^+_{ii}-\n^-_{ii}}c^+_ic^-_i$.
Then, we have
\begin{multline}\label{FinalD for B-type}
  \D_1=1+
  \sum_{i=1}^n (T_{2\varpi_i}+T_{-2\varpi_i})+\\
  b_n v^{\sum_{k=1}^n (2k-2n-1)\m_{nk}} \cdot e^{-\alpha_n}
  \left(vT_{\sum_{k=1}^n (\m_{nk}-\delta_{k,n})\varpi_k}+v^{-1}T_{\sum_{k=1}^n (\m_{nk}+\delta_{k,n})\varpi_k}\right)+\\
  (1+v)^{-2}b_n^2 v^{-2+(\n_{nn}^+-\n_{nn}^-) + \sum_{k=1}^n (4k-4n-2)\m_{nk}} \cdot e^{-2\alpha_n}T_{\sum_{k=1}^n 2\m_{n,k}\varpi_k}+\\
  \sum_{i=1}^{n-1} b_i v^{\sum_{k=1}^n (2k-2n-1)\m_{ik}}\cdot e^{-\alpha_i}
  \left(T_{\sum_{k=1}^n (\m_{ik}+\delta_{k,i}+\delta_{k,i+1})\varpi_k}+T_{\sum_{k=1}^n (\m_{ik}-\delta_{k,i}-\delta_{k,i+1})\varpi_k}\right)+\\
  \sum_{1\leq i<j<n}^{\epsilon^\pm_{i,i+1}=\ldots=\epsilon^\pm_{j-1,j}=\pm 1}
  b_i\cdots b_j v^{2j-2i+2\sum_{i\leq a<b\leq j} (\n^+_{ab}-\n^-_{ab})+\sum_{k=1}^n \sum_{s=i}^j (2k-2n-1)\m_{sk}}\times\\
  e^{-\alpha_i-\ldots-\alpha_j}T_{\sum_{k=1}^n(\sum_{s=i}^j \m_{sk}+\delta_{k,i}+\delta_{k,j+1})\varpi_k}+\\
  \sum_{1\leq i<n}^{\epsilon^\pm_{i,i+1}=\ldots=\epsilon^\pm_{n-1,n}=\pm 1}
  b_i\cdots b_n v^{2n-2i-1+\sum_{i\leq a<b\leq n}(\n^+_{ab}-\n^-_{ab})(2-\delta_{b,n}) +\sum_{k=1}^n \sum_{s=i}^n (2k-2n-1)\m_{sk}}\times\\
  e^{-\alpha_i-\ldots-\alpha_n}T_{\sum_{k=1}^n(\sum_{s=i}^n \m_{sk}+\delta_{k,i})\varpi_k}+\\
  \sum_{1\leq i<n}^{\epsilon^\pm_{i,i+1}=\ldots=\epsilon^\pm_{n-1,n}=\pm 1}
  (1+v)^{-2}b_i\cdots b_{n-1}b_n^2 v^{2n-2i+(\n^+_{nn}-\n^-_{nn})+2\sum_{i\leq a<b\leq n} (\n^+_{ab}-\n^-_{ab})}\times\\
  v^{\sum_{k=1}^n \sum_{s=i}^j (2k-2n-1)(1+\delta_{s,n})\m_{sk}}\cdot
  e^{-\alpha_i-\ldots-\alpha_{n-1}-2\alpha_n}T_{\sum_{k=1}^n(\sum_{s=i}^{n} \m_{sk}+\m_{nk}+\delta_{k,i}-\delta_{k,n})\varpi_k}+\\
  \sum_{1\leq i<j<n}^{\epsilon^\pm_{i,i+1}=\ldots=\epsilon^\pm_{j-1,j}=\mp 1}
  b_i\cdots b_j v^{2j-2i+2\sum_{i\leq a<b\leq j}(\n^+_{ba}-\n^-_{ba})+\sum_{k=1}^n \sum_{s=i}^j (2k-2n-1)\m_{sk}}\times\\
  e^{-\alpha_i-\ldots-\alpha_j}T_{\sum_{k=1}^n(\sum_{s=i}^j \m_{sk}-\delta_{k,i}-\delta_{k,j+1})\varpi_k}+\\
  \sum_{1\leq i<n}^{\epsilon^\pm_{i,i+1}=\ldots=\epsilon^\pm_{n-1,n}=\mp 1}
  b_i\cdots b_n v^{2n-2i+1+2\sum_{i\leq a<b\leq n} (\n^+_{ba}-\n^-_{ba})+\sum_{k=1}^n \sum_{s=i}^n (2k-2n-1)\m_{sk}}\times\\
  e^{-\alpha_i-\ldots-\alpha_n}T_{\sum_{k=1}^n(\sum_{s=i}^n \m_{sk}-\delta_{k,i})\varpi_k}+\\
  \sum_{1\leq i<n}^{\epsilon^\pm_{i,i+1}=\ldots=\epsilon^\pm_{n-1,n}=\mp 1}
  (1+v)^{-2}b_i\cdots b_{n-1}b_n^2
  v^{2n-2i+(\n^+_{nn}-\n^-_{nn})+\sum_{i\leq a<b\leq n} (\n^+_{ba}-\n^-_{ba})(2+2\delta_{b,n})}\times\\
  v^{\sum_{k=1}^n \sum_{s=i}^n (2k-2n-1)(1+\delta_{s,n})\m_{sk}}\cdot
  e^{-\alpha_i-\ldots-\alpha_{n-1}-2\alpha_n}T_{\sum_{k=1}^n(\sum_{s=i}^n \m_{sk}+\m_{nk}-\delta_{k,i}+\delta_{k,n})\varpi_k}.
\end{multline}

\begin{Rem}\label{Etingof's formula for B}
If $\epsilon^+=\epsilon^-$, then the last six sums are vacuous. If we further set
$\n^+=\n^-$ and $c^\pm_i=\pm 1$ for all $i$, then we obtain the formula for the first
hamiltonian of the type $B_n$ quantum difference Toda lattice as defined in~\cite{E}
(we write down this formula as we could not find it in the literature, even though
it can be derived completely analogously to~\cite[(5.7)]{E}, cf.~\cite[the end of Section 3]{FFJMM}):
\begin{multline}\label{standard qToda for B}
  \D_1=
  1+\sum_{i=1}^n (T_{2\varpi_i}+T_{-2\varpi_i})-
  (v^2-v^{-2})^2 \sum_{i=1}^{n-1} e^{-\alpha_i} \left(T_{\varpi_i+\varpi_{i+1}}+T_{-\varpi_i-\varpi_{i+1}}\right)-\\
  (v-v^{-1})^2 e^{-\alpha_n} \left(vT_{-\varpi_n}+v^{-1}T_{\varpi_n}\right)+
  v^{-2}(1-v^{-1})^2(v-v^{-1})^2 e^{-2\alpha_n}.
\end{multline}
\end{Rem}

Recall the algebra $\CA_n$ from Section~\ref{explicit hamiltonians A}. Consider the
anti-isomorphism from $\CA_n$ to the algebra $\CalD_v(H^\ad_{\mathfrak{so}_{2n+1}})$
of Section~\ref{section Whittaker functions}, sending
  $\sw_j\mapsto T_{-\varpi_j}, \sD_j\mapsto e^{-\sum_{k=j}^n \alpha_k}$.
Let $\HH=\HH(\epsilon^\pm, \n^\pm, c^\pm)$ be the element of $\CA_n$ which corresponds
to $\D_1$ under this anti-isomorphism (to save space, we omit the explicit long formula for $\HH$).
The following is the key property of $\HH(\epsilon^\pm,\n^\pm,c^\pm)$ in type $B$:

\begin{Prop}\label{bound for B}
$\HH(\epsilon^\pm,\n^\pm,c^\pm)$ depends only on
$\vec{\epsilon}=(\epsilon_{n-1},\ldots,\epsilon_1)\in \{-1,0,1\}^{n-1}$
with $\epsilon_i:=\frac{\epsilon^+_{i,i+1}-\epsilon^-_{i,i+1}}{2}$,
up to algebra automorphisms of $\CA_n$.
\end{Prop}

The proof of this result is analogous to that of Proposition~\ref{bound for A} given
in Appendix~\ref{Proof of bound for A}, see also Remark~\ref{Remark on other types},
we leave details to the interested reader. Proposition~\ref{bound for B} implies that
given two pairs of Sevostyanov triples $(\epsilon^\pm,\n^\pm,c^\pm)$ and
$(\tilde{\epsilon}^\pm,\tilde{\n}^\pm,\tilde{c}^\pm)$ with
  $\epsilon^+_{i,i+1}-\epsilon^-_{i,i+1}=\tilde{\epsilon}^+_{i,i+1}-\tilde{\epsilon}^-_{i,i+1}$,
there exists an algebra automorphism of $\CalD_v(H^\ad_{\mathfrak{so}_{2n+1}})$ which maps
the first hamiltonian $\D_1(\epsilon^\pm,\n^\pm,c^\pm)$ to
$\D_1(\tilde{\epsilon}^\pm,\tilde{\n}^\pm,\tilde{c}^\pm)$. As we will see in
Appendix~\ref{Proof of Main Theorem A}, the same automorphism maps the modified quantum Toda system
$\CT(\epsilon^\pm,\n^\pm,c^\pm)$ to $\CT(\tilde{\epsilon}^\pm,\tilde{\n}^\pm,\tilde{c}^\pm)$.


\subsection{Lax matrix realization in type $A$}\label{section A type}
\

In this section, we identify the type $A_{n-1}$ modified quantum difference Toda systems
with those discovered in~\cite[11(ii, iii)]{FT} via the Lax matrix formalism.

Recall the algebra $\CA_n$ of Section~\ref{explicit hamiltonians A}.
Consider the following three \emph{(local) Lax matrices}:
\begin{equation}\label{relativistic Lax}
  L^{v,-1}_i(z):=\left(%
  \begin{array}{cc}
    \sw_i^{-1}-\sw_iz^{-1} & \sw_i \sD_i^{-1} \\
    -\sw_i\sD_iz^{-1} & \sw_i \\
  \end{array}%
  \right)\in \mathrm{Mat}(2, z^{-1}\CA_n[z]),
\end{equation}
\begin{equation}\label{hatL}
  L^{v,0}_i(z):=
  \left(%
  \begin{array}{cc}
    \sw_i^{-1}z^{1/2}-\sw_i z^{-1/2} & \sD^{-1}_iz^{1/2} \\
    -\sD_iz^{-1/2} & 0 \\
  \end{array}%
  \right)\in{\on{Mat}}(2,z^{-1/2}\CA_n[z]),
\end{equation}
\begin{equation}\label{new Lax matrix}
  L^{v,1}_i(z):=\left(%
  \begin{array}{cc}
    \sw_i^{-1}z-\sw_i & \sw_i^{-1}\sD_i^{-1}z \\
    -\sw_i^{-1}\sD_i & -\sw_i^{-1} \\
  \end{array}%
  \right)\in \mathrm{Mat}(2, \CA_n[z]).
\end{equation}
For any $\vec{k}=(k_n,\ldots,k_1)\in \{-1,0,1\}^n$, define
the {\em mixed complete monodromy matrix}
\begin{equation}\label{complete monodromy matrix}
  \sT^v_{\vec{k}}(z):=L^{v,k_n}_n(z)\cdots L^{v,k_1}_{1}(z).
\end{equation}
We also recall the standard trigonometric $R$-matrix
\begin{equation*}
  R_\trig(z):=\left(%
  \begin{array}{cccc}
     1 & 0 & 0 & 0 \\
     0 & \frac{z-1}{vz-v^{-1}} & \frac{z(v-v^{-1})}{vz-v^{-1}} & 0 \\
     0 & \frac{v-v^{-1}}{vz-v^{-1}} & \frac{z-1}{vz-v^{-1}} & 0 \\
     0 & 0 & 0 & 1 \\
  \end{array}%
\right).
\end{equation*}
The following key property of the complete monodromy matrices is
established in~\cite[11(ii)]{FT}:

\begin{Prop}\label{RTT-relation}
For any $\vec{k}\in \{-1,0,1\}^n$, $\sT^v_{\vec{k}}(z)$ satisfies the trigonometric RTT-relation:
\begin{equation*}
  R_\trig(z/w)\left(\sT^v_{\vec{k}}(z)\otimes 1\right)\left(1\otimes \sT^v_{\vec{k}}(w)\right)=
  \left(1\otimes \sT^v_{\vec{k}}(w)\right)\left(\sT^v_{\vec{k}}(z)\otimes 1\right)R_\trig(z/w).
\end{equation*}
\end{Prop}

As an immediate corollary of this result, we obtained (see~\cite[11(iii)]{FT}):

\begin{Prop}[\cite{FT}]
Fix $\vec{k}=(k_n,\ldots,k_1)\in \{-1,0,1\}^n$.

\noindent
(a) The coefficients in powers of $z$ of the matrix element $\sT^v_{\vec{k}}(z)_{11}$
generate a commutative subalgebra of $\CA_n$. Moreover, they lie in the subalgebra
of $\CA_n$ generated by
  $\{\sw_j^{\pm 1}, (\sD_i/\sD_{i+1})^{\pm 1}\}_{1\leq i<n}^{1\leq j\leq n}$.

\noindent
(b) $\sT^v_{\vec{k}}(z)_{11}=
     (-1)^n\sw_1\cdots\sw_n\left(z^s-{\mathsf H}^{\vec{k}}_2z^{s+1}+\mathrm{higher\ powers\ of}\ z\right),$
where $s=\sum_{j=1}^n \frac{k_j-1}{2}$.
The hamiltonian ${\mathsf H}^{\vec{k}}_2$ equals
\begin{equation}\label{mixedH2}
  {\mathsf H}^{\vec{k}}_2=
  \sum_{j=1}^n \sw^{-2}_j+
  \sum_{i=1}^{n-1} \sw_i^{-k_i-1}\sw_{i+1}^{-k_{i+1}-1}\cdot \frac{\sD_i}{\sD_{i+1}}+
  \sum_{1\leq i<j-1\leq n-1}^{k_{i+1}=\ldots=k_{j-1}=1} \sw_i^{-k_i-1}\cdots\sw_j^{-k_j-1}\cdot \frac{\sD_i}{\sD_j}.
\end{equation}

\noindent
(c) Set $\vec{k}'=(0,k_{n-1},\ldots,k_2,0)$.
Then, ${\mathsf H}^{\vec{k}}_2$ is conjugate to ${\mathsf H}^{\vec{k}'}_2$.
\end{Prop}

Let $\CT^{\vec{k}}$ denote the commutative subalgebra of $\bar{\CA}_n$ generated
by the images of the coefficients in powers of $z$ of the matrix element $
\sT^v_{\vec{k}}(z)_{11}$, while $\bar{\HH}^{\vec{k}}_2\in \CT^{\vec{k}}$ denote
the image of $\HH^{\vec{k}}_2$. The main result of this section identifies
$\CT^{\vec{k}}$ with the pre-images $\widetilde{\CT}(\epsilon^\pm,\n^\pm,c^\pm)$
of $\CT(\epsilon^\pm,\n^\pm,c^\pm)$ in type $A_{n-1}$ under the anti-isomorphism
$\bar{\CA}_n\to \CalD_v(H^\ad_{\ssl_n})$ of Section~\ref{explicit hamiltonians A}.
This provides a Lax matrix realization of the type $A$ modified quantum difference Toda systems.

\begin{Thm}\label{Lax for type A}
Given a pair of type $A_{n-1}$ Sevostyanov triples $(\epsilon^\pm,\n^\pm,c^\pm)$
and $\vec{k}\in \{-1,0,1\}^n$ satisfying $k_{i+1}=\frac{\epsilon^+_{i,i+1}-\epsilon^-_{i,i+1}}{2}$
for any $1\leq i\leq n-2$, the following holds:

\noindent
(a) There is an algebra automorphism of $\bar{\CA}_n$ which maps
$\HH(\epsilon^\pm,\n^\pm,c^\pm)$ to $\bar{\HH}^{\vec{k}}_2$.

\noindent
(b) The automorphism of part (a) maps $\widetilde{\CT}(\epsilon^\pm,\n^\pm,c^\pm)$ to $\CT^{\vec{k}}$.
\end{Thm}

The proof of Theorem~\ref{Lax for type A} is presented in Appendix~\ref{Proof of Lax for A}
and closely follows our proofs of Proposition~\ref{bound for A} (see Appendix~\ref{Proof of bound for A})
and Theorem~\ref{Main Thm A} (see Appendix~\ref{Proof of Main Theorem A}).

\begin{Rem}\label{classical Kuznetsov-Tsyganov}
For $\vec{k}=\vec{0}$, we recover the Lax matrix realization of the
type $A$ quantum difference Toda lattice, due to~\cite{KuT}.
\end{Rem}

Actually, the above construction admits a standard one-parameter deformation
of commutative subalgebras of $\CA_n$ as provided by the following result:

\begin{Prop}\label{Closed Toda for A}
(a) For any $\varepsilon\in \BC(v^{1/\NN})$, the coefficients in powers of $z$
of the linear combination $\sT^v_{\vec{k}}(z)_{11}+\varepsilon \sT^v_{\vec{k}}(z)_{22}$
generate a commutative subalgebra of $\CA_n$. Moreover, they lie in the subalgebra
of $\CA_n$ generated by $\{\sw_j^{\pm 1}, (\sD_i/\sD_{i+1})^{\pm 1}\}_{1\leq i<n}^{1\leq j\leq n}$.

\noindent
(b) We have
  $\sT^v_{\vec{k}}(z)_{11}+\varepsilon \sT^v_{\vec{k}}(z)_{22}=
   (-1)^n\sw_1\cdots\sw_n\left(\hat{\HH}^{\vec{k}}_1 z^s-\hat{\HH}^{\vec{k}}_2 z^{s+1}+\mathrm{higher\ powers\ of}\ z\right)$,
where $s=\sum_{j=1}^n \frac{k_j-1}{2}$. Here,
  $\hat{\HH}^{\vec{k}}_1=1+\varepsilon\prod_{j=1}^n \delta_{k_j,1}\prod_{j=1}^n \sw_j^{-2}$,
while $\hat{\HH}^{\vec{k}}_2$ is given by
\begin{multline}\label{closedmixedH2}
  \hat{\HH}^{\vec{k}}_2=
  \sum_{j=1}^n \sw^{-2}_j+
  \sum_{i=1}^{n-1} \sw_i^{-k_i-1}\sw_{i+1}^{-k_{i+1}-1}\cdot \frac{\sD_i}{\sD_{i+1}}+
  \sum_{1\leq i<j-1\leq n-1}^{k_{i+1}=\ldots=k_{j-1}=1} \sw_i^{-k_i-1}\cdots\sw_j^{-k_j-1}\cdot \frac{\sD_i}{\sD_j}+\\
  \varepsilon \left(\sum_{1\leq i<j\leq n}^{\substack{k_1=\ldots=k_{i-1}=1\\k_{j+1}=\ldots=k_n=1}}
  \sw_1^{-k_1-1}\cdots \sw_i^{-k_i-1}\sw_j^{-k_j-1}\cdots \sw_n^{-k_n-1}\cdot \frac{\sD_j}{\sD_i}+
  \sum_{1\leq j\leq n}^{\substack{k_j=-1\\k_{i}=1 (i\ne j)}} \prod_{k\ne j} \sw_k^{-2}\right).
\end{multline}
\end{Prop}

\begin{proof}
(a) Follows from the equality
  $\left[\sT^v_{\vec{k}}(z)_{11}+\varepsilon \sT^v_{\vec{k}}(z)_{22},
   \sT^v_{\vec{k}}(w)_{11}+\varepsilon \sT^v_{\vec{k}}(w)_{22}\right]=0$,
which is implied by the RTT-relation of Proposition~\ref{RTT-relation}.
(b) Straightforward computation.
\end{proof}

\begin{Def}
A \textbf{type $A_{n-1}$ periodic modified quantum difference Toda system} is
the commutative subalgebra $\hat{\CT}^{\vec{k},\varepsilon}$ of $\bar{\CA}_n$.
\end{Def}

We note that $\hat{\CT}^{\vec{k},0}$ coincides with $\CT^{\vec{k}}$.

\begin{Rem}\label{affine Toda 1}
(a) In particular, $\hat{\HH}^{\vec{0}}_2$ is conjugate
(in the sense of~(\ref{Canonical Form},~\ref{conjugation A-type})) to the element of
$\bar{\CA}_n$ which corresponds under the anti-isomorphism $\bar{\CA}_n\to \CalD_v(H^\ad_{\ssl_n})$
of Section~\ref{explicit hamiltonians A} to the first hamiltonian of the type $A^{(1)}_{n-1}$
quantum difference affine Toda system of~\cite[(5.9)]{E}:
\begin{equation}\label{affine Toda A}
  \hat{\D}_1=\sum_{j=1}^n T_{2\varpi_j}-
  (v-v^{-1})^2\sum_{i=1}^{n-1} e^{-\varpi_i+\varpi_{i+1}}T_{\varpi_i+\varpi_{i+1}}-
  \kappa(v-v^{-1})^2e^{-\varpi_n+\varpi_1}T_{\varpi_n+\varpi_1}
\end{equation}
with $\kappa=(-1)^n(v-v^{-1})^{-2n}\varepsilon$.

The quantum difference affine Toda systems are defined similarly to the (finite type)
$q$-Toda systems of \emph{loc.cit.}, but starting from a quantum affine algebra
and its center at the critical level. The parameter $\kappa \in \BC(v^{1/\NN})$ is essential,
i.e., it can not be removed, cf.~\cite[Remark 1]{E}.

\noindent
(b) We expect that most of our results from this paper can be generalized to an affine setting.
In particular, the type $A_{n-1}$ periodic modified quantum difference Toda systems introduced
above should be conjugate to the type $A_{n-1}^{(1)}$ modified quantum difference affine Toda
systems, thus generalizing part (a) of the current Remark. To state this more precisely, let us
first specify what we mean by a Sevostyanov triple $(\hat{\epsilon},\hat{\n},\hat{c})$ for $\hat{\g}$
of an affine type, except $A^{(1)}_1$. Let $\alpha_0,\ldots,\alpha_n$ be simple positive roots of
$\hat{\g}$ (with $\alpha_0$ the distinguished one) and $\{a_i\}_{i=0}^n$ be the labels on the
Dynkin diagram of $\hat{\g}$ as in~\cite[Chapter IV, Table Aff]{Ka}.
Then, a Sevostyanov triple $(\hat{\epsilon},\hat{\n},\hat{c})$ is a collection of the following data:
 (1) $\hat{\epsilon}=(\epsilon_{ij})_{i,j=0}^n$ is the associated matrix of
     an orientation of $\Dyn(\hat{\g})$ as before,
 (2) $\hat{c}=(c_i)_{i=0}^n\in (\BC(v^{1/\NN})^\times)^{n+1}$ is a collection of non-zero constants,
 (3) $\hat{\n}=(\n_{ij})_{0\leq i\leq n}^{1\leq j\leq n}$ is an integer matrix satisfying
$\sd_j\n_{ij}-\sd_i\n_{ji}=\epsilon_{ij}b_{ij}$ for $1\leq i,j\leq n$ and
$\sd_j\n_{0j}+\sum_{p=1}^n\sd_p\n_{jp}\frac{a_p}{a_0}=\epsilon_{0j}b_{0j}$ for $1\leq j\leq n$.

In particular, given a pair of type $A^{(1)}_{n-1}\ (n>2)$ Sevostyanov triples
$(\hat{\epsilon}^\pm,\hat{\n}^\pm,\hat{c}^\pm)$, the corresponding difference operator
$\HH(\hat{\epsilon}^\pm,\hat{\n}^\pm,\hat{c}^\pm)\in \bar{\CA}_n$ should depend
(up to algebra automorphisms of $\bar{\CA}_n$) only on
$\vec{\epsilon}=(\epsilon_{n-1},\ldots,\epsilon_0)\in \{-1,0,1\}^n$, where
$\epsilon_i:=\frac{\epsilon^+_{i,i+1}-\epsilon^-_{i,i+1}}{2}$ with indices considered modulo $n$.
The resulting $3^n$ difference operators $\HH(\hat{\epsilon}^\pm,\hat{\n}^\pm,\hat{c}^\pm)$
should be conjugate to the images of $3^n$ hamiltonians $\hat{\HH}^{\vec{k}}_2$ in $\bar{\CA}_n$.
We have verified this result for $\epsilon^\pm_{0,1}=\epsilon^\pm_{0,n-1}=1$.

\noindent
(c) The type $A^{(1)}_1$ deserves a special treatment, since it is the only affine type
for which the analogue of Lemma~\ref{Sev twist} fails to hold (as $a_{01}=a_{10}=-2$).
Instead, such characters $\chi^\pm$ exist if and only if
$\n^\pm_{01}+\n^\pm_{11}\in \{-2,0,2\}$. Let $\HH(\hat{\n}^\pm,\hat{c}^\pm)$ be the element
of $\bar{\CA}_2$ which corresponds to the first hamiltonian of the type $A^{(1)}_1$ modified
quantum difference affine Toda system associated with the pair $(\hat{\n}^\pm,\hat{c}^\pm)$.
We expect that $\HH(\hat{\n}^\pm,\hat{c}^\pm)$ depends (up to algebra automorphisms of $\bar{\CA}_2$)
only on $\frac{\n^+_{01}+\n^+_{11}-\n^-_{01}-\n^-_{11}}{2}\in \{\pm 2,\pm 1,0\}$.
On the other hand, it is easy to see that the equivalence class of the difference
operator $\hat{\HH}^{(k_2,k_1)}_2$ depends only on $k_2-k_1\in \{\pm 2,\pm 1,0\}$.
We expect that the resulting five difference operators in $\bar{\CA}_2$ are conjugate
to the aforementioned five difference operators $\HH(\hat{\n}^\pm,\hat{c}^\pm)$.
\end{Rem}


\subsection{Lax matrix realization in type $C$}\label{section C type}
\

Motivated by the construction of the previous section, we provide a Lax matrix realization
of the type $C$ modified quantum difference Toda systems.

In addition to $L^{v,k}_i(z)\ (k=\pm 1,0)$ of~(\ref{relativistic Lax}--\ref{new Lax matrix}),
consider three more \emph{(local) Lax matrices}:
\begin{equation}\label{Lax conj -1}
  \bar{L}^{v,-1}_i(z):=\left(%
  \begin{array}{cc}
    \sw_i-\sw_i^{-1}z^{-1} & \sw_i^{-1}\sD_i \\
    -\sw_i^{-1}\sD_i^{-1}z^{-1} & \sw_i^{-1} \\
  \end{array}%
  \right)\in {\on{Mat}}(2,z^{-1}\CA_n[z]),
\end{equation}
\begin{equation}\label{Lax conj 0}
  \bar{L}^{v,0}_i(z):=
  \left(%
  \begin{array}{cc}
    \sw_i z^{1/2}-\sw_i^{-1} z^{-1/2} & \sD_iz^{1/2} \\
    -\sD_i^{-1}z^{-1/2} & 0 \\
  \end{array}%
  \right)\in {\on{Mat}}(2,z^{-1/2}\CA_n[z]),
\end{equation}
\begin{equation}\label{Lax conj 1}
  \bar{L}^{v,1}_i(z):=\left(%
  \begin{array}{cc}
    \sw_iz-\sw_i^{-1} & \sw_i\sD_iz \\
    -\sw_i\sD_i^{-1} & -\sw_i \\
  \end{array}%
  \right)\in {\on{Mat}}(2,\CA_n[z]).
\end{equation}
For any $\vec{k}=(k_n,\ldots,k_1)\in \{-1,0,1\}^n$, define the
{\em double mixed complete monodromy matrix}
\begin{equation}\label{double complete monodromy matrix}
  \BT^v_{\vec{k}}(z):=
  \bar{L}^{v,-k_1}_1(z)\cdots \bar{L}^{v,-k_n}_{n}(z)L^{v,k_n}_n(z)\cdots L^{v,k_1}_{1}(z).
\end{equation}

Let us summarize the key properties of the double mixed complete monodromy matrices:

\begin{Prop}
Fix $\vec{k}=(k_n,\ldots,k_1)\in \{-1,0,1\}^n$.

\noindent
(a) $\BT^v_{\vec{k}}(z)$ satisfies the trigonometric RTT-relation:
\begin{equation*}
  R_\trig(z/w)\left(\BT^v_{\vec{k}}(z)\otimes 1\right)\left(1\otimes \BT^v_{\vec{k}}(w)\right)=
  \left(1\otimes \BT^v_{\vec{k}}(w)\right)\left(\BT^v_{\vec{k}}(z)\otimes 1\right)R_\trig(z/w).
\end{equation*}

\noindent
(b) The coefficients in powers of $z$ of the matrix element $\BT^v_{\vec{k}}(z)_{11}$
generate a commutative subalgebra of $\CA_n$. Moreover, they belong to the subalgebra
$\CalC_n$ of $\CA_n$, introduced in Section~\ref{explicit hamiltonians C}.

\noindent
(c) We have
  $\BT^v_{\vec{k}}(z)_{11}=z^{-n}-\BH^{\vec{k}}_2z^{-n+1}+\mathrm{higher\ powers\ of}\ z$.
The hamiltonian $\BH^{\vec{k}}_2$ equals
\begin{multline}\label{doublemixedH2}
  \BH^{\vec{k}}_2=
  \sum_{i=1}^n (\sw_i^2+\sw_i^{-2})+
  \sum_{i=1}^{n-1} \left(\sw_i^{-k_i-1}\sw_{i+1}^{-k_{i+1}-1}+\sw_i^{-k_i+1}\sw_{i+1}^{-k_{i+1}+1}\right)
  \cdot \frac{\sD_i}{\sD_{i+1}}+v^{-k_n}\sw_n^{-2k_n}\cdot \sD_n^2+\\
  \sum_{1\leq i<j< n}^{k_{i+1}=\ldots=k_j=1} \sw_i^{-k_i-1}\cdots \sw_{j+1}^{-k_{j+1}-1}\cdot \frac{\sD_i}{\sD_{j+1}}+
  \sum_{1\leq i<n}^{k_{i+1}=\ldots=k_n=1} v^{-1}\sw_i^{-k_i-1}\cdots \sw_n^{-k_n-1}\cdot \sD_i\sD_n+\\
  \sum_{1\leq i<j< n}^{k_{i+1}=\ldots=k_j=-1} \sw_i^{-k_i+1}\cdots \sw_{j+1}^{-k_{j+1}+1}\cdot \frac{\sD_i}{\sD_{j+1}}+
  \sum_{1\leq i<n}^{k_{i+1}=\ldots=k_n=-1} v \sw_i^{-k_i+1}\cdots \sw_n^{-k_n+1}\cdot \sD_i\sD_n.
\end{multline}
\end{Prop}

\begin{proof}
(a) Note that $\bar{L}^{v,k}_i(z)$ is obtained from $L^{v,k}_i(z)$ by applying
the automorphism of $\CA_n$ which maps
  $\sw_i^{\pm 1}\mapsto \sw_i^{\mp 1}, \sD_i^{\pm 1}\mapsto \sD_i^{\mp 1}$.
Hence, each of them satisfies the trigonometric RTT-relation.
Thus, an arbitrary product of $L^{v,k}_i(z)$ and $\bar{L}^{v,k}_i(z)$
also satisfies the trigonometric RTT-relation.
(b) This is an immediate corollary of (a). (c) Straightforward computation.
\end{proof}

Let $\CT^{\vec{k}}$ denote the commutative subalgebra of $\CalC_n$ generated by
the coefficients in powers of $z$ of the matrix element $\BT^v_{\vec{k}}(z)_{11}$.
The main result of this section identifies $\CT^{\vec{k}}$ with the pre-images
$\widetilde{\CT}(\epsilon^\pm,\n^\pm,c^\pm)$ of $\CT(\epsilon^\pm,\n^\pm,c^\pm)$
in type $C_n$ under the anti-isomorphism $\CalC_n\to \CalD_v(H^\ad_{\mathfrak{sp}_{2n}})$
of Section~\ref{explicit hamiltonians C}. This provides a Lax matrix realization of
the type $C$ modified quantum difference Toda systems.

\begin{Thm}\label{Lax for type C}
Given a pair of type $C_n$ Sevostyanov triples $(\epsilon^\pm,\n^\pm,c^\pm)$ and
$\vec{k}\in \{-1,0,1\}^n$ satisfying
$k_{i+1}=\frac{\epsilon^+_{i,i+1}-\epsilon^-_{i,i+1}}{2}$ for any $1\leq i\leq n-1$,
the following holds:

\noindent
(a) There is an algebra automorphism of $\CalC_n$ which maps
$\HH(\epsilon^\pm,\n^\pm,c^\pm)$ to $\BH^{\vec{k}}_2$.

\noindent
(b) The automorphism of part (a) maps $\widetilde{\CT}(\epsilon^\pm,\n^\pm,c^\pm)$ to $\CT^{\vec{k}}$.
\end{Thm}

The proof of Theorem~\ref{Lax for type C} is completely analogous to that of
Theorem~\ref{Lax for type A} given in Appendix~\ref{Proof of Lax for A},
see also Remark~\ref{Remark on other types}, we leave details to the interested reader.

\begin{Rem}\label{Folding}
Recall that any non simply laced simple Lie algebra $\g'$ admits a
\emph{folding realization} $\g'\simeq \g^\sigma$, where $\g$ is a simply laced
Lie algebra endowed with an outer automorphism $\sigma$ of a finite order
(arising as an automorphism of the corresponding Dynkin diagram). This observation
allows to relate the classical Toda system of $\g'$ to the Toda system of $\g$,
see~\cite{OT}. The above construction of the modified quantum difference Toda systems
in types $C_n$ and $A_{2n-1}$ via Lax matrices exhibits the former as a folding
of the latter, once we require that the orientations $\epsilon^\pm$ of
$\Dyn(\ssl_{2n})=A_{2n-1}$ satisfy $\epsilon^\pm_{i,i+1}=-\epsilon^\pm_{2n-i-1,2n-i}$
for all $i$. However, such a naive approach fails to work for the pair $(B_n,D_{n+1})$.
It would be interesting to understand the explicit relation. Let us warn the interested
reader that the folding for quantum groups is more elaborate than in the classical setup,
see~\cite{BG}.
\end{Rem}

Analogously to the type $A$ case, the above construction admits a standard one-parameter
deformation of commutative subalgebras of $\CalC_n$ as provided by the following result:

\begin{Prop}\label{Closed Toda for C}
(a) For any $\varepsilon\in \BC(v^{1/\NN})$, the coefficients in powers of $z$ of
the linear combination $\BT^v_{\vec{k}}(z)_{11}+\varepsilon \BT^v_{\vec{k}}(z)_{22}$
generate a commutative subalgebra of $\CA_n$. Moreover, they belong to the subalgebra
$\CalC_n$ of $\CA_n$.

\noindent
(b) We have
  $\BT^v_{\vec{k}}(z)_{11}+\varepsilon \BT^v_{\vec{k}}(z)_{22}=
   z^{-n}-\hat{\BH}^{\vec{k}}_2 z^{-n+1}+\mathrm{higher\ powers\ of}\ z$.
The hamiltonian $\hat{\BH}^{\vec{k}}_2$ is given by the following formula:
\begin{multline}\label{closedmixedH2 type C}
  \hat{\BH}^{\vec{k}}_2=
  \sum_{i=1}^n (\sw_i^2+\sw_i^{-2})+
  \sum_{i=1}^{n-1} \left(\sw_i^{-k_i-1}\sw_{i+1}^{-k_{i+1}-1}+\sw_i^{-k_i+1}\sw_{i+1}^{-k_{i+1}+1}\right)
  \cdot \frac{\sD_i}{\sD_{i+1}}+v^{-k_n}\sw_n^{-2k_n}\cdot \sD_n^2+\\
  \sum_{1\leq i<j< n}^{k_{i+1}=\ldots=k_j=1} \sw_i^{-k_i-1}\cdots \sw_{j+1}^{-k_{j+1}-1}\cdot \frac{\sD_i}{\sD_{j+1}}+
  \sum_{1\leq i<n}^{k_{i+1}=\ldots=k_n=1} v^{-1}\sw_i^{-k_i-1}\cdots \sw_n^{-k_n-1}\cdot \sD_i\sD_n+\\
  \sum_{1\leq i<j< n}^{k_{i+1}=\ldots=k_j=-1} \sw_i^{-k_i+1}\cdots \sw_{j+1}^{-k_{j+1}+1}\cdot \frac{\sD_i}{\sD_{j+1}}+
  \sum_{1\leq i<n}^{k_{i+1}=\ldots=k_n=-1} v \sw_i^{-k_i+1}\cdots \sw_n^{-k_n+1}\cdot \sD_i\sD_n+\\
  \varepsilon\left(\delta_{n,1}(1-\delta_{k_1,0})\sw_1^{-2k_1}+v^{k_1}\sw_1^{-2k_1}\cdot \sD_1^{-2}+
  \prod_{i=1}^n \delta_{k_i,1}\cdot \prod_{i=1}^n \sw_i^{-2}\cdot \frac{\sD_n}{\sD_1}+
  \prod_{i=1}^n \delta_{k_i,-1}\cdot \prod_{i=1}^n \sw_i^{2}\cdot \frac{\sD_n}{\sD_1}+\right.\\
  \left.\sum_{1\leq i<n}^{k_1=\ldots=k_i=1} v \sw_1^{-k_1-1}\cdots \sw_{i+1}^{-k_{i+1}-1}\cdot \frac{1}{\sD_1\sD_{i+1}}+
  \sum_{1\leq i<n}^{k_1=\ldots=k_i=-1} v^{-1} \sw_1^{-k_1+1}\cdots \sw_{i+1}^{-k_{i+1}+1}\cdot \frac{1}{\sD_1\sD_{i+1}}\right).
\end{multline}
\end{Prop}

\begin{Def}
A \textbf{type $C_n$ periodic modified quantum difference Toda system} is
the commutative subalgebra $\hat{\CT}^{\vec{k},\varepsilon}$ of $\CalC_n$.
\end{Def}

We note that $\hat{\CT}^{\vec{k},0}$ coincides with $\CT^{\vec{k}}$.

\begin{Rem}\label{affine Toda 2}
(a)  In particular, $\hat{\BH}^{\vec{0}}_2$ is conjugate to the element of $\CalC_n$
which corresponds under the anti-isomorphism $\CalC_n\to \CalD_v(H^\ad_{\mathfrak{sp}_{2n}})$
of Section~\ref{explicit hamiltonians C} to the first hamiltonian of the type $C^{(1)}_{n}$
quantum difference affine Toda system of~\cite{E} (cf.~Remark~\ref{affine Toda 1}(a)):
\begin{multline}\label{affine Toda C}
  \hat{\D}_1=\sum_{i=1}^n (T_{2\varpi_i}+T_{-2\varpi_i})-
  (v-v^{-1})^2\sum_{i=1}^{n-1}  e^{-\varpi_i+\varpi_{i+1}}\left(T_{\varpi_i+\varpi_{i+1}}+T_{-\varpi_i-\varpi_{i+1}}\right)-\\
  (v^2-v^{-2})^2 e^{-2\varpi_n}-\kappa v^{-2n-2}(v^2-v^{-2})^2e^{2\varpi_1}
\end{multline}
with $\kappa=v^{2n+2}(v-v^{-1})^{-4(n-1)}(v^2-v^{-2})^{-4}\varepsilon$.

Here $\kappa\in \BC(v^{1/\NN})$ is an essential parameter.
For $\kappa=0$, we recover $\D_1$ of~(\ref{standard qToda for C}).

\noindent
(b) Following our discussion and notations of Remark~\ref{affine Toda 1}(b),
we have also verified that the element of $\CalC_n$ corresponding under
the anti-isomorphism $\CalC_n\to \CalD_v(H^\ad_{\mathfrak{sp}_{2n}})$ to
$\HH(\hat{\epsilon}^\pm,\hat{\n}^\pm,\hat{c}^\pm)$ with $\epsilon^\pm_{01}=1$,
is conjugate to $\hat{\BH}^{\vec{k}}_2$ with $k_1=0$ and
$k_{i+1}=\frac{\epsilon^+_{i,i+1}-\epsilon^-_{i,i+1}}{2}$ for any $1\leq i\leq n-1$.

\noindent
(c) For completeness of our list~(\ref{affine Toda A},~\ref{affine Toda C}),
let us present explicit formulas for the first hamiltonian $\hat{\D}_1$ of the
quantum difference affine Toda systems (defined in~\cite{E}) for the remaining
classical series $D^{(1)}_n$ and $B^{(1)}_n$ (as we could not find such formulas
in the literature):

$\bullet$ In type $D^{(1)}_n$, we have
\begin{multline}\label{affine Toda D}
  \hat{\D}_1=\sum_{i=1}^n (T_{2\varpi_i}+T_{-2\varpi_i})-
  (v-v^{-1})^2 \sum_{i=1}^{n-1} e^{-\varpi_i+\varpi_{i+1}} \left(T_{\varpi_i+\varpi_{i+1}}+T_{-\varpi_i-\varpi_{i+1}}\right)-\\
  (v-v^{-1})^2 e^{-\varpi_{n-1}-\varpi_n}\left(T_{-\varpi_{n-1}+\varpi_n}+v^{-2} T_{\varpi_{n-1}-\varpi_n}\right)+
  (v-v^{-1})^4 e^{-2\varpi_{n-1}}-\\
  \kappa v^{-2n+2}(v-v^{-1})^2 e^{\varpi_1+\varpi_2}(T_{-\varpi_1+\varpi_2}+T_{\varpi_1-\varpi_2})+
  \kappa v^{-2n+2}(v-v^{-1})^4 e^{2\varpi_2}.
\end{multline}

$\bullet$ In type $B^{(1)}_n$, we have
\begin{multline}\label{affine Toda B}
  \hat{\D}_1=1+\sum_{i=1}^n (T_{2\varpi_i}+T_{-2\varpi_i})-
  (v^2-v^{-2})^2 \sum_{i=1}^{n-1} e^{-\varpi_i+\varpi_{i+1}} \left(T_{\varpi_i+\varpi_{i+1}}+T_{-\varpi_i-\varpi_{i+1}}\right)-\\
  (v-v^{-1})^2 e^{-\varpi_n} \left(vT_{-\varpi_n}+v^{-1}T_{\varpi_n}\right)+v^{-2}(1-v^{-1})^2(v-v^{-1})^2 e^{-2\varpi_n}-\\
  \kappa v^{-4n+2}(v^2-v^{-2})^2e^{\varpi_1+\varpi_2}(T_{-\varpi_1+\varpi_2}+T_{\varpi_1-\varpi_2})+
  \kappa v^{-4n+2}(v^2-v^{-2})^4e^{2\varpi_2}.
\end{multline}
For $\kappa=0$, these formulas recover $\D_1$ of~(\ref{standard qToda for D})
and~(\ref{standard qToda for B}), respectively.
\end{Rem}


\subsection{Explicit formulas and classification in type $G_2$}\label{explicit hamiltonians G2}
\

Recall explicit formulas for the action of $U_v(\g_2)$ on its first
fundamental representation $V_1$. The space $V_1$ has a basis $\{w_i\}_{i=0}^6$,
in which the action is given by the following formulas:
\begin{align*}
  & E_1\colon w_0\mapsto (v+v^{-1})w_2,\ w_1\mapsto 0,\ w_2\mapsto 0,\ w_3\mapsto 0,\
    w_4\mapsto w_3,\ w_5\mapsto (v+v^{-1})w_0,\ w_6\mapsto w_1,\\
  & E_2\colon w_0\mapsto 0,\ w_1\mapsto 0,\ w_2\mapsto w_6,\ w_3\mapsto w_5,\
    w_4\mapsto 0,\ w_5\mapsto 0,\ w_6\mapsto 0,\\
  & F_1\colon w_0\mapsto w_5,\ w_1\mapsto w_6,\ w_2\mapsto w_0,\ w_3\mapsto w_4,\
    w_4\mapsto 0,\ w_5\mapsto 0,\ w_6\mapsto 0,\\
  & F_2\colon w_0\mapsto 0,\ w_1\mapsto 0,\ w_2\mapsto 0,\ w_3\mapsto 0,\
    w_4\mapsto 0,\ w_5\mapsto w_3,\ w_6\mapsto w_2,\\
  & L_1\colon w_0\mapsto w_0, w_1\mapsto v^2w_1, w_2\mapsto vw_2, w_3\mapsto v^{-1}w_3,
    w_4\mapsto v^{-2}w_4, w_5\mapsto v^{-1}w_5, w_6\mapsto vw_6,\\
  & L_2\colon w_0\mapsto w_0,\ w_1\mapsto v^3w_1,\ w_2\mapsto w_2,\ w_3\mapsto v^{-3}w_3,\
    w_4\mapsto v^{-3}w_4,\ w_5\mapsto w_5,\ w_6\mapsto v^{3}w_6,\\
  & K_1\colon w_0\mapsto w_0, w_1\mapsto vw_1, w_2\mapsto v^2w_2, w_3\mapsto vw_3,
    w_4\mapsto v^{-1}w_4, w_5\mapsto v^{-2}w_5, w_6\mapsto v^{-1}w_6,\\
  & K_2\colon w_0\mapsto w_0,\ w_1\mapsto w_1,\ w_2\mapsto v^{-3}w_2,\ w_3\mapsto v^{-3}w_3,\
    w_4\mapsto w_4,\ w_5\mapsto v^{3}w_5,\ w_6\mapsto v^{3}w_6.
\end{align*}

Let $\varpi_1$ and $\varpi_2$ be the weights of $w_2$ and $w_6$, respectively,
so that $(\varpi_1,\varpi_1)=(\varpi_2,\varpi_2)=2, (\varpi_1,\varpi_2)=-1$.
Then, the weights of $w_0, w_1, w_3, w_4, w_5$ are equal to
$0, \varpi_1+\varpi_2, -\varpi_2, -\varpi_1-\varpi_2, -\varpi_1$, respectively.
We also note that the simple roots are given by
  $\alpha_1=\varpi_1, \alpha_2=-\varpi_1+\varpi_2$.
Finally, we have $\rho=5\alpha_1+3\alpha_2=2\varpi_1+3\varpi_2$ and $\sd_1=1, \sd_2=3$.

To compute $\D_1$, we use the same strategy as for the classical types.
Analogously to type $B$, the operators $E_1^2$ and $F_1^2$ act non-trivially on $V_1$,
while $\{E_i^r,F_i^r\}^{r>1+\delta_{i,1}}_{i=1,2}$ still act by zero on $V_1$.
Therefore, applying formula~(\ref{C'}), we can replace $R_{\alpha_2}$
by $\bar{R}_{\alpha_2}:=1+(v^3-v^{-3})E_2\otimes F_2$ and $R_{\alpha_1}$ by
$\bar{R}_{\alpha_1}:=1+(v-v^{-1})E_1\otimes F_1+cE_1^2\otimes F_1^2$
for $c:=(1-v^{-1})(v-v^{-1})$. Let us now compute all the non-zero terms
contributing to $C'_{V_1}$:

$\bullet$
Picking $1$ out of each $\bar{R}^\op_i, \bar{R}_i$, we recover
\begin{equation*}
  1+v^2\cdot K_{2\varpi_1}+v^{-2}\cdot K_{-2\varpi_1}+v^8\cdot K_{2\varpi_2}+v^{-8}\cdot K_{-2\varpi_2}+
  v^{10}\cdot K_{2\varpi_1+2\varpi_2}+v^{-10}\cdot K_{-2\varpi_1-2\varpi_2}.
\end{equation*}

$\bullet$
Picking non-trivial terms only at $\bar{R}_{\alpha_j}^\op, \bar{R}_{\alpha_i}$,
the result does not depend on $\Or^\pm$ (hence, the orderings $\prec_\pm$) and
the total contribution of the non-zero terms equals
\begin{align*}
  &  (v-v^{-1})^2 \left(v^2(v+v^{-1})F_1E_1K_{\varpi_1}+(v+v^{-1})F_1K_{-\varpi_1}E_1+\right.\\
  &  \left.v^{10}F_1K_{\varpi_2}E_1K_{\varpi_1+\varpi_2}+
     v^{-8}F_1K_{-\varpi_1-\varpi_2}E_1K_{-\varpi_2}\right)+\\
  &  (v^3-v^{-3})^2 (v^8F_2K_{\varpi_1}E_2K_{\varpi_2}+v^{-2}F_2 K_{-\varpi_2}E_2K_{-\varpi_1})+
     c^2v^2(v+v^{-1})^2F_1^2K_{-\varpi_1}E_1^2K_{\varpi_1}.
\end{align*}

$\bullet$
The contribution of the remaining terms to $C'_{V_1}$ depends on $\Or^\pm$.
Tracing back explicit formulas for the action of $U_v(\g_2)$ on $V_1$,
let us evaluate the total contribution of such terms for each of the four
possible pairs $(\Or^+,\Or^-)$.

\medskip
\emph{Case 1:} If $\Or^+=\Or^-$, then there are no other terms.

\medskip
\emph{Case 2:} If $\Or^+\colon \alpha_1\leftarrow \alpha_2$ and
$\Or^-\colon \alpha_1\to \alpha_2$, then the total contribution equals
\begin{align*}
  &  (v-v^{-1})^2(v^3-v^{-3})^2(v^8(v+v^{-1})F_2F_1E_1E_2K_{\varpi_2}+v^{-2}F_2F_1K_{-\varpi_1-\varpi_2}E_1E_2K_{-\varpi_1})+\\
  &  c^2(v^3-v^{-3})^2v^8(v+v^{-1})^2\cdot F_2F_1^2K_{-\varpi_1}E_1^2E_2K_{\varpi_2}.
\end{align*}

\emph{Case 3:} If $\Or^+\colon \alpha_1\to \alpha_2$ and
$\Or^-\colon \alpha_1\leftarrow \alpha_2$, then the total contribution equals
\begin{align*}
  &  (v-v^{-1})^2(v^3-v^{-3})^2(v^{10}F_1F_2K_{\varpi_1}E_2E_1K_{\varpi_1+\varpi_2}+(v+v^{-1})F_1F_2K_{-\varpi_2}E_2E_1)+\\
  &  c^2(v^3-v^{-3})^2v^2(v+v^{-1})^2\cdot F_1^2F_2K_{-\varpi_2}E_2E_1^2K_{\varpi_1}.
\end{align*}

Thus, we have listed all the non-zero terms contributing to $C'_{V_1}$.
To obtain $\tilde{\D}_1=\bar{\D}_{V_1}$, we should rewrite the above formulas via $e_i, f_i$ and
  $L_p=\begin{cases}
     K_{\varpi_1+\varpi_2}, & \ \mathrm{if}\ p=1\\
     K_{\varpi_1+2\varpi_2}, & \ \mathrm{if}\ p=2
   \end{cases},$
moving all the Cartan terms to the middle, and then apply the characters
$\chi^\pm$ with $\chi^+(e_i)=c^+_i, \chi^-(f_i)=c^-_i$. Conjugating further
by $e^\rho$, we obtain the explicit formula for the first hamiltonian $\D_1$
of the type $G_2$ modified quantum difference Toda system. To write it down,
define constants
  $\m_{i1}:=(\n^-_{i1}-\n^+_{i1})+(\n^-_{i2}-\n^+_{i2}),\
   \m_{i2}:=(\n^-_{i1}-\n^+_{i1})+2(\n^-_{i2}-\n^+_{i2})$
and
  $b_i:=(v_i-v_i^{-1})^2v_i^{\n^+_{ii}-\n^-_{ii}}c^+_ic^-_i$.
Then, we have
\begin{multline}\label{FinalD for G2-type}
  \D_1=1+(T_{2\varpi_1}+T_{-2\varpi_1}+T_{2\varpi_2}+T_{-2\varpi_2}+T_{2\varpi_1+2\varpi_2}+T_{-2\varpi_1-2\varpi_2})+\\
  b_1v^{-\m_{11}-4\m_{12}}\cdot e^{-\alpha_1}
  \left(v^{-1}(v+v^{-1})T_{(\m_{11}+1)\varpi_1+\m_{12}\varpi_2}+v(v+v^{-1})T_{(\m_{11}-1)\varpi_1+\m_{12}\varpi_2}\right.+\\
  \left.T_{(\m_{11}+1)\varpi_1+(\m_{12}+2)\varpi_2}+T_{(\m_{11}-1)\varpi_1+(\m_{12}-2)\varpi_2}\right)+\\
  b_2v^{-\m_{21}-4\m_{22}}\cdot e^{-\alpha_2}\left(T_{(\m_{21}+1)\varpi_1+(\m_{22}+1)\varpi_2}+T_{(\m_{21}-1)\varpi_1+(\m_{22}-1)\varpi_2}\right)+\\
  \frac{b_1^2(v+v^{-1})^2}{(1+v)^2}v^{-2+(\n^+_{11}-\n^-_{11})-(2\m_{11}+8\m_{12})}\cdot e^{-2\alpha_1}T_{2\m_{11}\varpi_1+2\m_{12}\varpi_2}+\\
  \delta_{\epsilon^+_{12},-1}\delta_{\epsilon^-_{12},1}\cdot
  \left\{b_1b_2v^{-4+3(\n^+_{12}-\n^-_{12})-(\m_{11}+\m_{21}+4\m_{12}+4\m_{22})}\times\right. \\
  \left. e^{-\alpha_1-\alpha_2}\left((v+v^{-1})T_{(\m_{11}+\m_{21})\varpi_1+(\m_{12}+\m_{22}+1)\varpi_2}+vT_{(\m_{11}+\m_{21}-2)\varpi_1+(\m_{12}+\m_{22}-1)\varpi_2}\right)+\right. \\
  \left. \frac{b_1^2b_2(v+v^{-1})^2}{(1+v)^2} v^{-8+(\n^+_{11}-\n^-_{11})+6(\n^+_{12}-\n^-_{12})-(2\m_{11}+\m_{21}+8\m_{12}+4\m_{22})}\times\right. \\
  \left. e^{-2\alpha_1-\alpha_2} T_{(2\m_{11}+\m_{21}-1)\varpi_1+(2\m_{12}+\m_{22}+1)\varpi_2}\right\}+\\
  \delta_{\epsilon^+_{12},1}\delta_{\epsilon^-_{12},-1}\cdot
  \left\{b_1b_2v^{3+3(\n^+_{12}-\n^-_{12})-(\m_{11}+\m_{21}+4\m_{12}+4\m_{22})}\times\right. \\
  \left. e^{-\alpha_1-\alpha_2}\left(v(v+v^{-1})T_{(\m_{11}+\m_{21})\varpi_1+(\m_{12}+\m_{22}-1)\varpi_2}+T_{(\m_{11}+\m_{21}+2)\varpi_1+(\m_{12}+\m_{22}+1)\varpi_2}\right)+\right. \\
  \left. \frac{b_1^2b_2(v+v^{-1})^2}{(1+v)^2} v^{4+(\n^+_{11}-\n^-_{11})+6(\n^+_{12}-\n^-_{12})-(2\m_{11}+\m_{21}+8\m_{12}+4\m_{22})}\times\right. \\
  \left. e^{-2\alpha_1-\alpha_2} T_{(2\m_{11}+\m_{21}+1)\varpi_1+(2\m_{12}+\m_{22}-1)\varpi_2}\right\}.
\end{multline}

\begin{Rem}\label{Etingof's formula for G2}
If $\epsilon^+=\epsilon^-$, then the terms with $\delta$'s are vacuous.
If we further set $\n^+=\n^-$ and $c^\pm_i=\pm 1$ for all $i$, then we obtain
the formula for the first hamiltonian of the type $G_2$ quantum difference
Toda lattice as defined in~\cite{E}:
\begin{multline}\label{standard qToda for G2}
  \D_1=1+(T_{2\varpi_1}+T_{-2\varpi_1}+T_{2\varpi_2}+T_{-2\varpi_2}+T_{2\varpi_1+2\varpi_2}+T_{-2\varpi_1-2\varpi_2})-\\
  (v-v^{-1})^2\cdot e^{-\alpha_1}(v^{-1}(v+v^{-1})T_{\varpi_1}+v(v+v^{-1})T_{-\varpi_1}+T_{\varpi_1+2\varpi_2}+T_{-\varpi_1-2\varpi_2})-\\
  (v^3-v^{-3})^2\cdot e^{-\alpha_2}(T_{\varpi_1+\varpi_2}+T_{-\varpi_1-\varpi_2})+
  v^{-2}(1-v^{-1})^2(v^2-v^{-2})^2\cdot e^{-2\alpha_1}.
\end{multline}
\end{Rem}

Let $\CG_2$ be the associative $\BC(v)$-algebra generated by
$\{\sw_i^{\pm 1}, \sD_i^{\pm 1}\}_{i=1}^2$ subject to
\begin{align*}
  & [\sw_1,\sw_2]=[\sD_1,\sD_2]=0,\ \sw_i^{\pm 1}\sw_i^{\mp 1}=\sD_i^{\pm1}\sD_i^{\mp 1}=1,\\
  & \sD_1\sw_1=v^2\sw_1\sD_1,\ \sD_1\sw_2=v^{-1}\sw_2\sD_1,\ \sD_2\sw_1=v^{-3}\sw_1\sD_2,\ \sD_2\sw_2=v^3\sw_2\sD_2.
\end{align*}
Consider the anti-isomorphism from $\CG_2$ to the algebra $\CalD_v(H^\ad_{\g_2})$
of Section~\ref{section Whittaker functions}, sending
  $\sw_i\mapsto T_{-\varpi_i}, \sD_i\mapsto e^{-\alpha_i}$.
Let $\HH=\HH(\epsilon^\pm, \n^\pm, c^\pm)$ be the element of $\CG_2$ which
corresponds to $\D_1$ under this anti-isomorphism. The following is the key
property of $\HH(\epsilon^\pm,\n^\pm,c^\pm)$ in type $G_2$:

\begin{Prop}\label{bound for G2}
$\HH(\epsilon^\pm,\n^\pm,c^\pm)$ depends only on
$\epsilon:=\frac{\epsilon^+_{12}-\epsilon^-_{12}}{2}\in \{-1,0,1\}$,
up to algebra automorphisms of $\CG_2$.
\end{Prop}

The proof of Proposition~\ref{bound for G2} is completely analogous to that
of Proposition~\ref{bound for A} given in Appendix~\ref{Proof of bound for A},
we leave details to the interested reader. Proposition~\ref{bound for G2} implies
that given two pairs of Sevostyanov triples $(\epsilon^\pm,\n^\pm,c^\pm)$ and
$(\tilde{\epsilon}^\pm,\tilde{\n}^\pm,\tilde{c}^\pm)$ with
  $\epsilon^+_{12}-\epsilon^-_{12}=\tilde{\epsilon}^+_{12}-\tilde{\epsilon}^-_{12}$,
there exists an algebra automorphism of $\CalD_v(H^\ad_{\g_2})$ which maps the first hamiltonian
$\D_1(\epsilon^\pm,\n^\pm,c^\pm)$ to $\D_1(\tilde{\epsilon}^\pm,\tilde{\n}^\pm,\tilde{c}^\pm)$.
As we will see in Appendix~\ref{Proof of Main Theorem A}, the same automorphism maps the modified
quantum Toda system $\CT(\epsilon^\pm,\n^\pm,c^\pm)$ to $\CT(\tilde{\epsilon}^\pm,\tilde{\n}^\pm,\tilde{c}^\pm)$.

\begin{Rem}\label{Affine Toda for G2}
For completeness of our list~(\ref{affine Toda A},~\ref{affine Toda C},~\ref{affine Toda D},~\ref{affine Toda B}),
let us present the explicit formula for the first hamiltonian $\hat{\D}_1$ of the type $G^{(1)}_2$
quantum difference affine Toda system:
\begin{multline}\label{affine Toda G2}
  \hat{\D}_1=1+(T_{2\varpi_1}+T_{-2\varpi_1}+T_{2\varpi_2}+T_{-2\varpi_2}+T_{2\varpi_1+2\varpi_2}+T_{-2\varpi_1-2\varpi_2})-\\
  (v-v^{-1})^2 e^{-\varpi_1}(v^{-1}(v+v^{-1})T_{\varpi_1}+v(v+v^{-1})T_{-\varpi_1}+T_{\varpi_1+2\varpi_2}+T_{-\varpi_1-2\varpi_2})-\\
  (v^3-v^{-3})^2 e^{\varpi_1-\varpi_2}(T_{\varpi_1+\varpi_2}+T_{-\varpi_1-\varpi_2})+
  v^{-2}(1-v^{-1})^2(v^2-v^{-2})^2\cdot e^{-2\varpi_1}-\\
  \kappa v^{-12}(v^3-v^{-3})^2 e^{\varpi_1+2\varpi_2}(T_{\varpi_1}+T_{-\varpi_1})+
  \kappa v^{-12}(v-v^{-1})^2(v^3-v^{-3})^2 e^{2\varpi_2}.
\end{multline}
For $\kappa=0$, this recovers $\D_1$ of~(\ref{standard qToda for G2}).
\end{Rem}


\section{Whittaker vectors and their pairing}\label{section properties}

In this section, we study a pairing of two general Whittaker vectors (associated to
a pair of Sevostyanov triples $(\epsilon^\pm,\n^\pm,c^\pm)$) in universal Verma modules,
following~\cite{FFJMM}. We obtain a fermionic formula for the corresponding terms
$\tilde{J}_\beta$. We show that their generating series is a natural solution of
the modified quantum difference Toda system $\CT(\epsilon^\pm,\n^\pm,c^\pm)$ of
Section~\ref{section basic}. This provides a natural generalization of~\cite[Section 3]{FFJMM},
where $\epsilon^+=\epsilon^-$ and $\n^+=\n^-$.


\subsection{Whittaker vectors}\label{Whit vector defn}
\

Following the notations of~\cite{FFJMM}, consider $U_v(\g)$ and $U_{v^{-1}}(\g)$, whose
generators will be denoted by $E_i,F_i,L_i$ and $\bar{E}_i,\bar{F}_i,\bar{L}_i$,
respectively. In contrast to~\cite{FFJMM}, we will work with universal Verma modules.
Let $\{u_i\}_{i=1}^n$ be indeterminates and consider an extension
$\sk:=\BC(v^{1/\NN},u_1,\ldots,u_n)$ of $\BC(v^{1/\NN})$. Let $U_v(\g)^\leq$ be the subalgebra
of $U_v(\g)$ generated by $\{L_i^{\pm 1},F_i\}_{i=1}^n$ and consider its action on $\sk$ with
$F_i$ acting trivially and $L_i$ acting via multiplication by $u_i$. We define the universal
Verma module $\CV$ over $U_v(\g)$ as $\CV:=U_v(\g)\otimes_{U_v(\g)^\leq} \sk$.
It is generated by $\textbf{1}\in \sk$ such that $E_i(\textbf{1})=0$ and
$L_i(\textbf{1})=u_i\cdot \textbf{1}$ for $1\leq i\leq n$. We define the formal symbol
$\lambda:=\sum_{i=1}^n \frac{\log(u_i)}{\sd_i\log(v)}\alpha_i$, which will appear only
as an index or in the context of the homomorphism $v^{(\lambda,\cdot)}\colon P\to \sk$ defined
by $P\ni m_i\omega_i\mapsto v^{(\lambda,\sum_{i=1}^n m_i\omega_i)}=\prod_{i=1}^n u_i^{m_i}$,
so that $K_\mu(\textbf{1})=v^{(\lambda,\mu)}\cdot \textbf{1}$ for $\mu\in P$. In particular,
$\CV$ is graded by $Q_+$: $\CV=\oplus_{\beta\in Q_+} \CV_\beta$ with
$\CV_\beta=\{w\in \CV| K_\mu(w)=v^{(\mu,\lambda-\beta)}w\ (\mu\in P)\}$.
Similarly, let $\bar{\CV}$ be the universal Verma module over $U_{v^{-1}}(\g)$ generated by
the highest weight vector $\bar{\textbf{1}}$ such that $\bar{E}_i(\bar{\textbf{1}})=0$ and
$\bar{L}_i(\bar{\textbf{1}})=u_i^{-1}\cdot \bar{\textbf{1}}$ for $1\leq i\leq n$.
It is also $Q_+$-graded: $\bar{\CV}=\oplus_{\beta\in Q_+} \bar{\CV}_\beta$ with
$\bar{\CV}_\beta=\{w\in \bar{\CV}| \bar{K}_\mu(w)=v^{-(\mu,\lambda-\beta)}w\ (\mu\in P)\}$.

\begin{Rem}
One can alternatively work with the standard Verma modules $\CV^\lambda$ and
$\bar{\CV}^\lambda, \lambda\in P$ (one should further require $\lambda$ to be
strictly antidominant for the existence of Whittaker vectors), so that
$u_i=v^{(\lambda,\omega_i)}\in \BC(v^{1/\NN})$. This viewpoint is used in~\cite{FFJMM}.
We prefer the current exposition as it is compatible with our discussion in
Section~\ref{section geometric}. Nevertheless, motivated by the above standard setup,
we will freely use the above notation $v^{(\lambda,\mu)}$ for $\mu\in P$.
\end{Rem}

There is a unique non-degenerate $\sk$-bilinear pairing
$(\cdot,\cdot)\colon \CV\times \bar{\CV}\to \sk$ such that
$(\textbf{1}, \bar{\textbf{1}})=1$ and $(xw,w')=(w,\sigma(x)w')$
for all $x\in U_v(\g), w\in \CV, w'\in \bar{\CV}$, where the algebra
anti-isomorphism $\sigma\colon U_v(\g)\to U_{v^{-1}}(\g)$ is determined by
$\sigma(E_i)=\bar{F}_i, \sigma(F_i)=\bar{E}_i, \sigma(L_i)=\bar{L}_i^{-1}$.

\begin{Rem}
One can alternatively work with a single universal Verma module $\CV$ over $U_v(\g)$
endowed with the Shapovalov  form $(\cdot,\cdot)\colon \CV\times \CV\to \sk$, see our
discussion in Remark~\ref{Shapovalov vs dual}.
\end{Rem}

For the key definition of this section, consider the completions
$\CV^\wedge, \bar{\CV}^\wedge$ of $\CV, \bar{\CV}$, defined via
\begin{equation*}
  \CV^\wedge:=\prod_{\beta\in Q_+} \CV_\beta,\
  \bar{\CV}^\wedge:=\prod_{\beta\in Q_+} \bar{\CV}_\beta.
\end{equation*}
Given a pair of Sevostyanov triples $(\epsilon^\pm,\n^\pm,c^\pm)$
(now $c^\pm_i\in \sk^\times$), define $\nu^\pm_i\in P$ via
$\nu^\pm_i:=\sum_{k=1}^n \n_{ik}^\pm\omega_k$, so that
$(\nu^\pm_i,\alpha_j)-(\nu^\pm_j,\alpha_i)=\epsilon^\pm_{i,j}b_{i,j}$.
We have associated \textbf{Whittaker vectors}
\begin{equation*}
  \theta=\theta(\epsilon^+,\n^+,c^+)=\sum_{\beta\in Q_+} \theta_\beta\in \CV^\wedge
    \ \ (\mathrm{with}\ \theta_\beta\in \CV_\beta)
\end{equation*}
and
\begin{equation*}
  \bar{\theta}=\bar{\theta}(\epsilon^-,\n^-,c^-)=\sum_{\beta\in Q_+} \bar{\theta}_\beta\in \bar{\CV}^\wedge
    \ \ (\mathrm{with}\ \bar{\theta}_\beta\in \bar{\CV}_\beta)
\end{equation*}
which are uniquely determined by the following conditions:
\begin{equation}\label{Whittaker conditions}
  \theta_0=\textbf{1},\ E_iK_{\nu^+_i}(\theta)=c^+_i\cdot \theta\ \ \mathrm{and}\ \
  \bar{\theta}_0=\bar{\textbf{1}},\ \bar{E}_i\bar{K}_{\nu^-_i}(\bar{\theta})=c^-_i\cdot \bar{\theta}.
\end{equation}

\begin{Rem}\label{classical Whittaker}
This is a direct generalization of the classical notion of Whittaker vectors
for Lie algebras as defined by B.~Kostant in his milestone work on the subject~\cite{Ko}.
\end{Rem}


\subsection{Pairing of Whittaker vectors}
\

Set $(t;t)_r:=\prod_{k=1}^r (1-t^k)$ for $r\in \BZ_{>0}$, and $(t;t)_0:=1$.
Choose convex orderings $\prec_\pm$ on $\Delta_+$ such that
$\epsilon^\pm_{ij}=-1\Rightarrow \alpha_i\prec_\pm \alpha_j$,
cf.~Section~\ref{subsection technicalities}. Let
$\alpha_{i^\pm_1}\prec_\pm\cdots\prec_\pm\alpha_{i^\pm_n}$ be the simple roots
ordered with respect to $\prec_\pm$. For $1\leq i\ne j\leq n$ we write
$i\prec_\pm j$ if $\alpha_i\prec_\pm \alpha_j$. Define
\begin{equation}\label{J-factor}
  J_\beta=J_\beta(\epsilon^\pm,\n^\pm,c^\pm):=
  \left(\theta_\beta(\epsilon^+,\n^+,c^+),\bar{\theta}_\beta(\epsilon^-,\n^-,c^-)\right).
\end{equation}
Following~\cite[(3.11)]{FFJMM}\footnote{We note that $\tilde{J}_\beta$ is
denoted by $J^\lambda_\beta$ in~\cite{FFJMM}.} we also consider its slight modification
\begin{equation}\label{reduced J-factor}
  \tilde{J}_\beta=\tilde{J}_\beta(\epsilon^\pm,\n^\pm,c^\pm):=
  v^{-(\beta,\beta)/2+(\lambda,\beta)}
  \left(\theta_\beta(\epsilon^+,\n^+,c^+),\bar{\theta}_\beta(\epsilon^-,\n^-,c^-)\right).
\end{equation}
For $\beta\notin Q_+$, we set $J_\beta:=0$ and $\tilde{J}_\beta:=0$.
Our first result provides a recursive formula for $\tilde{J}_\beta$:

\begin{Thm}\label{Fermionic one}
We have
\begin{equation}\label{recursive}
  \tilde{J}_\beta=\sum_{0\leq \gamma\leq \beta} \frac{1}{(v^2)_{\beta-\gamma}}
  v^{(\gamma,\gamma)-2(\lambda+\rho,\gamma)}c^{\beta-\gamma}v^{\tau_\lambda(\beta-\gamma,\beta)}
  \cdot \tilde{J}_\gamma,
\end{equation}
where
  $(v^2)_{\alpha}:=\prod_{i=1}^n (v_i^2;v_i^2)_{m_i},
   c^{\alpha}:=\prod_{i=1}^n (-c^+_ic^-_i(v_i-v_i^{-1})^2)^{m_i},
   \tau_\lambda(\alpha,\beta):=\sum_{i=1}^n m_i(\nu^-_i-\nu^+_i,\lambda-\beta)+
   \sum_{j\prec_+ i} m_im_j(\nu^-_i-\nu^+_i,\alpha_j)+
   \sum_{i=1}^n \frac{m_i(m_i-1)}{2}(\nu^-_i-\nu^+_i,\alpha_i)$
for $\alpha=\sum_{i=1}^n m_i\alpha_i\in Q_+$.
\end{Thm}

\begin{proof}
The proof is completely analogous to that of~\cite[Theorem 3.1]{FFJMM} and is
based on an evaluation of $(C(\theta_\beta),\bar{\theta}_\beta)$ in two different ways,
where $C$ is the \emph{Drinfeld Casimir element}.
\end{proof}

Solving the recursive relation~(\ref{recursive}), one obtains an
explicit \emph{fermionic formula} for $\tilde{J}_\beta$:

\begin{Thm}\label{Fermionic two}
We have
\begin{equation}\label{explicit fermionic}
  \tilde{J}_\beta=c^\beta\cdot
  \sum_{\unl{\beta}=\{\beta^{(t)}\}_{t=0}^\infty\in Q_+^\BN}^{\sum_{t=0}^\infty \beta^{(t)}=\beta}
  \frac{v^{2B(\unl{\beta})+R(\unl{\beta})}}{\prod_{t=0}^\infty (v^2)_{\beta^{(t)}}},
\end{equation}
where we set
  $R(\unl{\beta}):=\sum_{t=0}^\infty \tau_\lambda\left(\beta^{(t)},\beta-\sum_{s=0}^{t-1}\beta^{(s)}\right)$
and
\begin{equation*}
  B(\unl{\beta}):=\frac{1}{2}\sum_{t,t'=0}^\infty \min(t,t')\left(\beta^{(t)},\beta^{(t')}\right)-
  \sum_{t=0}^\infty t \left(\lambda+\rho,\beta^{(t)}\right).
\end{equation*}
\end{Thm}

\begin{proof}
We will give a direct proof as the general machinery of fermionic formulas developed
in~\cite[Section 2]{FFJMM} does not apply to our setup. The formula is obvious for $\beta=0$.
From now on, fix $\beta>0$ (that is, $\beta\in Q_+\backslash\{0\}$). Let us rewrite
the equality~(\ref{recursive}) as
\begin{equation*}
  \left(1-v^{(\beta,\beta)-2(\lambda+\rho,\beta)}\right)\tilde{J}_\beta=
  \sum_{0<\beta_1\leq \beta}\frac{1}{(v^2)_{\beta_1}}v^{(\beta-\beta_1,\beta-\beta_1)-2(\lambda+\rho,\beta-\beta_1)}
  c^{\beta_1}v^{\tau_\lambda(\beta_1,\beta)}\tilde{J}_{\beta-\beta_1}.
\end{equation*}
We apply the same formula for $\tilde{J}_{\beta-\beta_1}$ if $\beta_1<\beta$.
Proceeding in the same way, we finally obtain
\begin{equation}\label{explicit J}
  \tilde{J}_\beta=\sum_{d\geq 1}\sum_{\beta_1,\ldots,\beta_d>0}^{\beta_1+\ldots+\beta_d=\beta}
  \frac{\prod_{e=1}^d c^{\beta_e}v^{\tau_\lambda(\beta_e,\beta_e+\ldots+\beta_{d})}
        v^{(\beta_{e+1}+\ldots+\beta_d,\beta_{e+1}+\ldots+\beta_d)-2(\lambda+\rho,\beta_{e+1}+\ldots+\beta_d)}}
  {\prod_{e=1}^d (v^2)_{\beta_e}
  (1-v^{(\beta_e+\ldots+\beta_d, \beta_e+\ldots+\beta_d)-2(\lambda+\rho,\beta_e+\ldots+\beta_d)})}.
\end{equation}

On the other hand, the summation in the right-hand side of~(\ref{explicit fermionic}) is over all
  $\unl{\beta}=\{\beta^{(t)}\}_{t=0}^\infty\in Q_+^\BN$ with $\sum_{t=0}^\infty\beta^{(t)}=\beta$.
Such sequences are in bijection with tuples
  $\{d,\{\beta_e\}_{e=1}^d,\{t_e\}_{e=1}^d|d\geq 1, \beta_e>0, t_e\in \BN, \sum_{e=1}^d\beta_e=\beta\}$
via $\beta^{(t_1+\ldots+t_e+e-1)}=\beta_e\ (1\leq e\leq d)$ and $\beta^{(t)}=0$ otherwise.

Hence, the right-hand side of~(\ref{explicit fermionic}) equals
\begin{multline}\label{I-terms}
  \sum_{d\geq 1}\sum_{\beta_1,\ldots,\beta_d>0}^{\beta_1+\ldots+\beta_d=\beta}
  \left\{\frac{c^\beta v^{\sum_{e=1}^d \tau_\lambda(\beta_e,\beta-\beta_1-\ldots-\beta_{e-1})}}{\prod_{e=1}^d (v^2)_{\beta_e}}
  \times\right.\\
  \left.\sum_{t_1,\ldots,t_d\geq 0} v^{\sum_{e=1}^d (t_1+\ldots+t_e+e-1)(\beta_e,\beta_e)+
  2\sum_{e<e'}(t_1+\ldots+t_e+e-1)(\beta_e,\beta_{e'})-2(\lambda+\rho,\sum_{e=1}^d (t_1+\ldots+t_e+e-1)\beta_e)}\right\}.
\end{multline}
It is straightforward to verify that the right-hand side of~(\ref{explicit J})
coincides with~(\ref{I-terms}).
\end{proof}


\subsection{$J$-functions: eigenfunctions of modified quantum difference Toda systems}\label{J-function}
\

Recall the elements $\omega_i^\vee=\omega_i/\sd_i\in P\otimes_{\BZ}\BQ$
satisfying $(\alpha_j,\omega_i^\vee)=\delta_{i,j}$. Consider a vector space
$N_\lambda$ that consists of all formal sums
  $\left\{\sum_{\beta\in Q} a_\beta \unl{y}^{\beta-\lambda}|a_\beta\in \sk\right\}$
for which there exists $\beta_0\in Q$ such that $a_\beta=0$ unless $\beta-\beta_0\in Q_+$,
where $\unl{y}^{\beta-\lambda}$ is used to denote $\prod_{i=1}^n y_i^{(\beta-\lambda,\omega_i^\vee)}$.
The vector space $N_{\lambda+\rho}$ is defined analogously with $\lambda$ being replaced
by $\lambda+\rho$. Consider the natural action of the algebra $\CalD_v(H^\ad)$ of
Section~\ref{section Whittaker functions} on the vector space $N_\lambda$, determined by
\begin{equation}\label{action on N}
  T_\mu(\unl{y}^{\beta-\lambda})=v^{-(\mu,\beta-\lambda)}\unl{y}^{\beta-\lambda},\
  e^{\alpha}(\unl{y}^{\beta-\lambda})=\unl{y}^{\beta-\alpha-\lambda}\
  \mathrm{for}\ \alpha,\beta\in Q,\mu\in P.
\end{equation}
The action of $\CalD_v(H^\ad)$ on $N_{\lambda+\rho}$ is defined analogously.

Consider the following generating functions of the terms $J_\beta$ defined in~(\ref{J-factor}):
\begin{equation}\label{J function}
\begin{split}
  & \tilde{J}=\tilde{J}(\{y_i\}_{i=1}^n):=
    \sum_{\beta\in Q_+} J_\beta \prod_{i=1}^n y_i^{(\beta-\lambda,\omega_i^\vee)}\in N_\lambda,\\
  & J=J(\{y_i\}_{i=1}^n):=
    \sum_{\beta\in Q_+} J_\beta \prod_{i=1}^n y_i^{(\beta-\lambda-\rho,\omega_i^\vee)}\in N_{\lambda+\rho}.
\end{split}
\end{equation}
Recall the difference operators $\tilde{\D}_V,\D_V\in \CalD_v(H^\ad)$ of
Section~\ref{section Whittaker functions}, associated with the pair of Sevostyanov
triples $(\epsilon^\pm, \n^\pm, c^\pm)$ and a finite-dimensional
$U_v(\g)$-representation $V$. The following is the key result of this section:

\begin{Thm}\label{J-eigenfunction}
(a) We have $\tilde{\D}_V(\tilde{J})=\tr_V\left(v^{2(\lambda+\rho)}\right)\cdot \tilde{J}$.

\noindent
(b) We have $\D_V(J)=\tr_V\left(v^{2(\lambda+\rho)}\right)\cdot J$.
\end{Thm}

\begin{proof}
First, we note that part (a) implies part (b), due to
  $\D_V=e^{\rho} \tilde{\D}_V e^{-\rho}$ and $J=e^{\rho}(\tilde{J})$.
The proof of part (a) is based on an evaluation of
$\left(C_V(\theta_\beta),\bar{\theta}_\beta\right)$ in two different ways,
where $C_V$ is the central element of~(\ref{quantum center}). On the one hand,
$C_V$ acts on $\CV$ as a multiplication by $\tr_V(v^{2(\lambda+\rho)})$
(since $C_V$ is central, $\CV$ is generated by $\textbf{1}$, and
$C_V(\textbf{1})=\tr_V(v^{2(\lambda+\rho)})\cdot \textbf{1}$), so that
  $\left(C_V(\theta_\beta),\bar{\theta}_\beta\right)=\tr_V(v^{2(\lambda+\rho)})\cdot J_\beta$.
On the other hand, we can use the explicit formula for $C_V$.

Let $\{w_k\}_{k=1}^N$ be a weight basis of $V$, and $\mu_k\in P$ be
the weight of $w_k$, cf. Section~\ref{section Whittaker functions}.
Then, we have
\begin{multline}\label{explicit computation C}
  C_V(\theta_\beta)=\sum_{\vec{m}=(m_1,\ldots,m_n)\in \BN^n}^{1\leq k\leq N} c(\vec{m})\cdot
  v^{(\lambda-\beta+2\rho,\mu_k)+(\lambda-\beta+\alpha(\vec{m}),\mu_k-\alpha(\vec{m}))}\times\\
  \left\langle w_k\middle| E_{\alpha_{i^-_1}}^{m_{i^-_1}}\cdots E_{\alpha_{i^-_n}}^{m_{i^-_n}}\cdot
  F_{\alpha_{i^+_1}}^{m_{i^+_1}}\cdots F_{\alpha_{i^+_n}}^{m_{i^+_n}}\middle| w_k\right\rangle\cdot
  F_{\alpha_{i^-_1}}^{m_{i^-_1}}\cdots F_{\alpha_{i^-_n}}^{m_{i^-_n}}\cdot
  E_{\alpha_{i^+_1}}^{m_{i^+_1}}\cdots E_{\alpha_{i^+_n}}^{m_{i^+_n}}(\theta_\beta),
\end{multline}
where $\langle w_k| x | w_k\rangle$ is the matrix coefficient of $x\in U_v(\g)$,
$\alpha(\vec{m}):=\sum_{i=1}^n m_i\alpha_i\in Q_+$, and $c(\vec{m})\in \BC(v)$ are
certain coefficients for which we currently do not need explicit formulas.

Using the defining property of $(\cdot,\cdot)$, we get
\begin{equation*}
  \left(F_{\alpha_{i^-_1}}^{m_{i^-_1}}\cdots F_{\alpha_{i^-_n}}^{m_{i^-_n}}\cdot
  E_{\alpha_{i^+_1}}^{m_{i^+_1}}\cdots E_{\alpha_{i^+_n}}^{m_{i^+_n}}(\theta_\beta),\bar{\theta}_\beta\right)=
  \left(E_{\alpha_{i^+_1}}^{m_{i^+_1}}\cdots E_{\alpha_{i^+_n}}^{m_{i^+_n}}(\theta_\beta),
  \bar{E}_{\alpha_{i^-_n}}^{m_{i^-_n}}\cdots \bar{E}_{\alpha_{i^-_1}}^{m_{i^-_1}}(\bar{\theta}_\beta)\right).
\end{equation*}
To evaluate the pairing in the right-hand side, note that the defining
conditions~(\ref{Whittaker conditions}) of the Whittaker vectors imply
$E_iK_{\nu^+_i}(\theta_{\gamma})=c^+_i\theta_{\gamma-\alpha_i}$ and
$\bar{E}_i\bar{K}_{\nu^-_i}(\bar{\theta}_{\gamma})=c^-_i\bar{\theta}_{\gamma-\alpha_i}$
for $\gamma\in Q$, hence,
$E_i(\theta_{\gamma})=c^+_iv^{-(\nu^+_i,\lambda-\gamma)}\theta_{\gamma-\alpha_i}$ and
$\bar{E}_i(\bar{\theta}_{\gamma})=c^-_iv^{(\nu^-_i,\lambda-\gamma)}\bar{\theta}_{\gamma-\alpha_i}$.
Applying this iteratively, we find
\begin{align*}
  & E_{\alpha_{i^+_1}}^{m_{i^+_1}}\cdots E_{\alpha_{i^+_n}}^{m_{i^+_n}}(\theta_\beta)=
  \prod_{i=1}^n (c^+_i)^{m_i} v^{\tau^+_\lambda(\vec{m},\beta)}\cdot \theta_{\beta-\alpha(\vec{m})},\\
  & \bar{E}_{\alpha_{i^-_n}}^{m_{i^-_n}}\cdots \bar{E}_{\alpha_{i^-_1}}^{m_{i^-_1}}(\bar{\theta}_\beta)=
  \prod_{i=1}^n (c^-_i)^{m_i} v^{\tau^-_\lambda(\vec{m},\beta)}\cdot \bar{\theta}_{\beta-\alpha(\vec{m})}
\end{align*}
with $\tau^\pm_\lambda(\vec{m},\beta)$ given by
  $\tau^+_\lambda(\vec{m},\beta):=-\sum_{k=1}^n \sum_{r=0}^{m_{i^+_k}-1}
   \left(\nu^+_{i^+_k},\lambda-\beta+r\alpha_{i^+_k}+\sum_{s=k+1}^n m_{i^+_{s}}\alpha_{i^+_{s}}\right)$
and
  $\tau^-_\lambda(\vec{m},\beta):=\sum_{k=1}^n \sum_{r=0}^{m_{i^-_k}-1}
   \left(\nu^-_{i^-_k},\lambda-\beta+r\alpha_{i^-_k}+\sum_{s=1}^{k-1} m_{i^-_{s}}\alpha_{i^-_{s}}\right)$.

Summarizing all these calculations, we obtain the following equality:
\begin{multline}\label{raz}
  \sum_{\vec{m}=(m_1,\ldots,m_n)\in \BN^n}^{1\leq k\leq N} c(\vec{m})\prod_{i=1}^n (c^+_ic^-_i)^{m_i}\cdot
  v^{(\lambda-\beta+2\rho,\mu_k)+(\lambda-\beta+\alpha(\vec{m}),\mu_k-\alpha(\vec{m}))+
     \tau^+_\lambda(\vec{m},\beta)+\tau^-_\lambda(\vec{m},\beta)}\times\\
  \left\langle w_k\middle| E_{\alpha_{i^-_1}}^{m_{i^-_1}}\cdots E_{\alpha_{i^-_n}}^{m_{i^-_n}}\cdot
  F_{\alpha_{i^+_1}}^{m_{i^+_1}}\cdots F_{\alpha_{i^+_n}}^{m_{i^+_n}}\middle| w_k\right\rangle\cdot
  J_{\beta-\alpha(\vec{m})}=\tr_V\left(v^{2(\lambda+\rho)}\right)\cdot J_\beta.
\end{multline}

Let us now compute $\tilde{\D}_V(\tilde{J})$. First, we need to rewrite
$F_{\alpha_{i^-_1}}^{m_{i^-_1}}\cdots F_{\alpha_{i^-_n}}^{m_{i^-_n}}$ and
$E_{\alpha_{i^+_1}}^{m_{i^+_1}}\cdots E_{\alpha_{i^+_n}}^{m_{i^+_n}}$ in
terms of the Sevostyanov generators $e_i,f_i$ and Cartan terms, moving
the latter to the right of $f_i$'s and to the left of $e_i$'s. We have:
\begin{align*}
  &  F_{\alpha_{i^-_1}}^{m_{i^-_1}}\cdots F_{\alpha_{i^-_n}}^{m_{i^-_n}}=
  v^{\tilde{\tau}^-_\lambda(\vec{m})}
  f_{\alpha_{i^-_1}}^{m_{i^-_1}}\cdots f_{\alpha_{i^-_n}}^{m_{i^-_n}}\cdot
  K_{\sum_{i=1}^n m_i\nu^-_i}, \\
  & E_{\alpha_{i^+_1}}^{m_{i^+_1}}\cdots E_{\alpha_{i^+_n}}^{m_{i^+_n}}=
  v^{\tilde{\tau}^+_\lambda(\vec{m})}
  K_{-\sum_{i=1}^n m_i\nu^+_i}\cdot
  e_{\alpha_{i^+_1}}^{m_{i^+_1}}\cdots e_{\alpha_{i^+_n}}^{m_{i^+_n}},
\end{align*}
where $\tilde{\tau}^\pm_\lambda(\vec{m})$ are given by
  $\tilde{\tau}^+_\lambda(\vec{m}):=\sum_{k=1}^n \sum_{r=1}^{m_{i^+_k}}
   \left(\nu^+_{i^+_k},r\alpha_{i^+_k}+\sum_{s=1}^{k-1} m_{i^+_s}\alpha_{i^+_{s}}\right)$
and
  $\tilde{\tau}^-_\lambda(\vec{m}):=-\sum_{k=1}^n \sum_{r=1}^{m_{i^-_k}}
   \left(\nu^-_{i^-_k},r\alpha_{i^-_k}+\sum_{s=k+1}^n m_{i^-_s}\alpha_{i^-_{s}}\right)$.
Tracing back the definition of $\tilde{\D}_V$, we find
\begin{multline}\label{dva}
  \tilde{\D}_V(\tilde{J})=\tilde{\D}_V\left(\sum_{\tilde{\beta}\in Q_+} J_{\tilde{\beta}}\unl{y}^{\tilde{\beta}-\lambda}\right)=
  \sum_{\tilde{\beta}\in Q_+} \sum_{\vec{m}=(m_1,\ldots,m_n)\in \BN^n}^{1\leq k\leq N}c(\vec{m})\prod_{i=1}^n (c^+_ic^-_i)^{m_i}\times\\
  v^{(2\rho,\mu_k)+\sum_{i=1}^n m_i(\nu^+_i-\nu^-_i,\tilde{\beta}-\lambda)-(2\mu_k-\alpha(\vec{m}),\tilde{\beta}-\lambda)-(\mu_k,\alpha(\vec{m}))+\tilde{\tau}^+_\lambda(\vec{m})+\tilde{\tau}^-_\lambda(\vec{m})}\times\\
  \left\langle w_k\middle| E_{\alpha_{i^-_1}}^{m_{i^-_1}}\cdots E_{\alpha_{i^-_n}}^{m_{i^-_n}}\cdot
  F_{\alpha_{i^+_1}}^{m_{i^+_1}}\cdots F_{\alpha_{i^+_n}}^{m_{i^+_n}}\middle| w_k\right\rangle\cdot
  J_{\tilde{\beta}}\unl{y}^{\alpha(\vec{m})+\tilde{\beta}-\lambda}.
\end{multline}
Due to the equalities
  $\tilde{\tau}^+_\lambda(\vec{m})+\tilde{\tau}^-_\lambda(\vec{m})=
   \tau^+_\lambda(\vec{m},\beta)+\tau^-_\lambda(\vec{m},\beta)+\sum_{i=1}^n m_i(\nu^+_i-\nu^-_i,\alpha(\vec{m})+\lambda-\beta)$
and
  $-(\mu_k,\alpha(\vec{m}))-(2\mu_k-\alpha(\vec{m}),\beta-\lambda-\alpha(\vec{m}))=
   (\lambda-\beta,\mu_k)+(\lambda-\beta+\alpha(\vec{m}),\mu_k-\alpha(\vec{m}))$,
the coefficient of $\unl{y}^{\beta-\lambda}$ in the right-hand side
of~(\ref{dva}) coincides with the left-hand side of~(\ref{raz}).

The equality $\tilde{\D}_V(\tilde{J})=\tr_V\left(v^{2(\lambda+\rho)}\right)\cdot \tilde{J}$ follows.
\end{proof}


\section{Geometric realization of the Whittaker vectors in type A}\label{section geometric}

In~\cite{BF}, A.~Braverman and M.~Finkelberg provided a geometric realization of the
universal Verma module over $U_v(\ssl_n)$, the Shapovalov form on it, and two particular
Whittaker vectors $\mathfrak{k},\mathfrak{w}$ of it via the Laumon based quasiflags'
moduli spaces. The vectors $\mathfrak{k}$ and $\mathfrak{w}$ correspond to particular
Sevostyanov triples $(\epsilon,\n,c)$ with the corresponding orientations of the $A_{n-1}$
Dynkin diagram being equioriented. In this section, we generalize their construction by
providing a geometric interpretation of all Whittaker vectors and their pairing.


\subsection{Laumon spaces}
\

First, we recall the setup of~\cite{FK}. Let $\bC$ be a smooth projective curve of genus zero.
We fix a coordinate $z$ on $\bC$, and consider the action of $\BC^\times$ on $\bC$ such that
$v(z)=v^{-2}z$. We have $\bC^{\BC^\times}=\{0,\infty\}$. We consider an $n$-dimensional vector
space $W$ with a basis $w_1,\ldots,w_n$. This defines a Cartan torus
  $T\subset G=\mathrm{SL}(n)\subset \mathrm{Aut}(W)$.
We also consider its $2^{n-1}$-fold cover, the bigger torus $\widetilde{T}$, acting on $W$
as follows: for $\widetilde{T}\ni\unl{t}=(t_1,\ldots,t_n)$ we have $\unl{t}(w_k)=t_k^2w_k$.

Given an $(n-1)$-tuple of nonnegative integers $\unl{d}=(d_1,\ldots,d_{n-1})\in \BN^{n-1}$,
we consider the \emph{Laumon's based quasiflags' space} $\fQ_{\unl{d}}$. It is the moduli
space of flags of locally free subsheaves
  $0\subset\CW_1\subset\cdots\subset\CW_{n-1}\subset\CW=W\otimes\CO_\bC$
such that $\rk(\CW_k)=k, \deg(\CW_k)=-d_k$, $\CW_k\subset\CW$ is a vector subbundle
in a neighbourhood of $\infty\in\bC$, and the fiber of $\CW_k$ at $\infty$ equals the span
$\langle w_1,\ldots,w_k\rangle\subset W$. It is a smooth connected
quasi-projective variety of dimension $\sum_{i=1}^{n-1} 2d_i$.

The group $\widetilde{T}\times\BC^\times$ acts naturally on $\fQ_{\unl{d}}$.
The set of fixed points of $\widetilde{T}\times\BC^\times$ on $\fQ_{\unl{d}}$ is finite
and is parametrized by collections $\wt{\unl{d}}$ of nonnegative integers
$(d_{ij})_{1\leq j\leq i\leq n-1}$ such that $d_i=\sum_{j=1}^i d_{ij}$ and
$d_{kj}\geq d_{ij}$ for $i\geq k\geq j$, see~\cite[2.11]{FK}.
Given a collection $\wt{\unl{d}}$ as above, we will denote by
$\wt{\unl{d}}+\delta_{ij}$ the collection $\wt{\unl{d}}{}'$, such that ${d{}}'_{ij}=d_{ij}+1$,
while ${d{}}'_{kl}=d_{kl}$ for $(k,l)\ne (i,j)$. By abuse of notation, we use $\wt{\unl{d}}$
to denote the corresponding $\widetilde{T}\times\BC^\times$-fixed point in $\fQ_{\unl{d}}$.

For $i\in\{1,\ldots,n-1\}$ and $\unl{d}=(d_1,\ldots,d_{n-1})$, we set
$\unl{d}+i:=(d_1,\ldots,d_i+1,\ldots,d_{n-1})$. We have a correspondence
  $\fE_{\unl{d},i}\subset\fQ_{\unl{d}}\times\fQ_{\unl{d}+i}$
formed by the pairs $(\CW_\bullet,\CW'_\bullet)$ such that
$\CW'_i\subset\CW_i$ and $\CW_j=\CW'_j$ for $j\ne i$.
It is a smooth quasi-projective variety of dimension $1+\sum_{i=1}^{n-1} 2d_i$.
We denote by $\bp$ (resp.\ $\bq$) the natural projection $\fE_{\unl{d},i}\to \fQ_{\unl{d}}$
(resp.\ $\fE_{\unl{d},i}\to \fQ_{\unl{d}+i}$). We also have a map
$\bs\colon\fE_{\unl{d},i}\to\bC$, given by
  $(\CW_\bullet,\CW'_\bullet)\mapsto\on{supp}(\CW_i/\CW'_i)$.
The correspondence $\fE_{\unl{d},i}$ comes equipped with a natural line bundle $\CL_i$
whose fiber at a point $(\CW_\bullet,\CW'_\bullet)$ equals $\Gamma(\bC,\CW_i/\CW'_i)$.

We denote by ${}'M$ the direct sum of equivariant $K$-groups:
  ${}'M:=\oplus_{\unl{d}}K^{\widetilde{T}\times\BC^\times}(\fQ_{\unl{d}})$.
It is a module over
  $K^{\wt{T}\times\BC^\times}(\on{pt})=\BC[\wt{T}\times\BC^\times]=
   \BC[t^{\pm 1}_1,\ldots,t^{\pm 1}_n,v^{\pm 1}:\ t_1\cdots t_n=1]$.
We define
  $M:={}'M\otimes_{K^{\wt{T}\times\BC^\times}(\on{pt})}
   \on{Frac}(K^{\wt{T}\times\BC^\times}(\on{pt}))\otimes_{\BC(v)} \BC(v^{1/\NN})$.
It is naturally graded:
  $M=\oplus_{\unl{d}}M_{\unl{d}}$.
According to the Thomason localization theorem, restriction to the
$\wt{T}\times\BC^\times$-fixed point set induces an isomorphism of localized $K$-groups
  $K^{\wt{T}\times\BC^\times}(\fQ_{\unl{d}})_\loc\iso
   K^{\wt{T}\times\BC^\times}(\fQ_{\unl{d}}^{\wt{T}\times\BC^\times})_\loc$.
The classes of the structure sheaves $[\wt{\unl{d}}]$ of the $\wt{T}\times\BC^\times$-fixed
points $\wt{\unl{d}}$ form a basis in
  $\oplus_{\unl{d}}K^{\wt{T}\times\BC^\times}(\fQ_{\unl{d}}^{\wt{T}\times\BC^\times})_\loc$.
The embedding of a point $\wt{\unl{d}}$ into $\fQ_{\unl{d}}$ is a proper morphism,
so the direct image in the equivariant $K$-theory is well-defined, and we will denote
by $[\wt{\unl{d}}]\in M_{\unl{d}}$ the direct image of the structure sheaf of the point
$\wt{\unl{d}}$. The set $\{[\wt{\unl{d}}]\}$ forms a basis of $M$.


\subsection{$U_v(\ssl_n)$-action via Laumon spaces}
\

Following Section~\ref{Whit vector defn}, consider the universal Verma module $\CV$
over $U_v(\ssl_n)$ with $u_i=v^{\frac{i(i-1)}{2}}t_1\cdots t_i$, i.e.,
  $L_i(\textbf{1})=v^{\frac{i(i-1)}{2}}t_1\cdots t_i\cdot \textbf{1}$.
We identify
  $\sk\simeq \on{Frac}(K^{\wt{T}\times\BC^\times}(\on{pt}))\otimes_{\BC(v)} \BC(v^{1/\NN})$.

\noindent
Define the following operators on $M$:

  $E_i:=t_{i+1}^{-1}v^{d_{i+1}-d_i+1-i}\bp_*\bq^*\colon M_{\unl{d}}\to M_{\unl{d}-i},$

  $F_i:=-t_i^{-1}v^{d_i-d_{i-1}+i}\bq_*(\CL_i\otimes \bp^*) \colon M_{\unl{d}}\to M_{\unl{d}+i},$

  $L_i:=t_1\cdots t_iv^{-d_i+\frac{i(i-1)}{2}}\colon M_{\unl{d}}\to M_{\unl{d}},$

  $K_i:=L_{i-1}^{-1}L_i^2L_{i+1}^{-1}=t_{i+1}^{-1}t_iv^{d_{i+1}-2d_i+d_{i-1}-1}\colon M_{\unl{d}}\to M_{\unl{d}}$.

\noindent
To each $\wt{\unl{d}}$, we also assign a collection of $\wt{T}\times\BC^\times$-weights $s_{ij}:=t_j^2v^{-2d_{ij}}$.

The following result is due to~\cite{BF} (though our formulas follow~\cite{T,FT}):

\begin{Thm}\label{Br-Fin}
(a) The operators $\{E_i, F_i, L^{\pm 1}_i\}_{i=1}^{n-1}$ give rise to the action of $U_v(\ssl_n)$ on $M$.

\noindent
(b) There is a unique $U_v(\ssl_n)$-module isomorphism $M\iso \CV$ taking $[\CO_{\fQ_{\unl{0}}}]$ to $\textbf{1}$.

\noindent
(c) The action of $L_i$ is diagonal in the basis $\{[\wt{\unl{d}}]\}$ and
  $L_i([\wt{\unl{d}}])=t_1\cdots t_iv^{-d_i+\frac{i(i-1)}{2}}\cdot [\wt{\unl{d}}]$.

\noindent
(d) The matrix coefficients of $F_i, E_i$ in the fixed point basis
$\{[\wt{\unl{d}}]\}$ of $M$ are as follows:
\begin{equation*}
  F_{i[\wt{\unl{d}},\wt{\unl{d}}{}']}=-(1-v^2)^{-1}t_i^{-1}v^{d_i-d_{i-1}+i} s_{ij}
  \prod_{j\ne k\leq i} (1-s_{ij}/s_{ik})^{-1}\prod_{k\leq i-1}(1-s_{ij}/s_{i-1,k})
\end{equation*}
if $\wt{\unl{d}'}=\wt{\unl{d}}+\delta_{ij}$ for certain $j\leq i$;
\begin{equation*}
  E_{i[\wt{\unl{d}},\wt{\unl{d}}{}']}=(1-v^2)^{-1}t_{i+1}^{-1}v^{d_{i+1}-d_i+1-i}
  \prod_{j\ne k\leq i} (1-s_{ik}/s_{ij})^{-1} \prod_{k\leq i+1}(1-s_{i+1,k}/s_{ij})
\end{equation*}
if $\wt{\unl{d}'}=\wt{\unl{d}}-\delta_{ij}$ for certain $j\leq i$.
All the other matrix coefficients of $F_i,E_i$ vanish.
\end{Thm}


\subsection{Geometric realization of the Whittaker vectors}\label{section geom whit}
\

Choose a Sevostyanov triple $(\epsilon,\n,c)$ and let
  $e_i:=E_i\prod_{p=1}^{n-1}L_p^{\n_{ip}}=E_iK_\nu,
   \nu:=\sum_{p=1}^{n-1}\n_{ip}\omega_p$,
be the corresponding Sevostyanov generators.
Choose $\unl{a}=(a_1,\ldots,a_{n-1})\in \{0,1\}^{n-1}$ so that
$a_i=\frac{1+\epsilon_{i-1,i}}{2}=\frac{1-\n_{i-1,i}+\n_{i,i-1}}{2}$
for $1<i\leq n-1$, while $a_1$ equals either $0$ or $1$.

Consider the line bundle $\CalD_i$ on $\fQ_{\unl{d}}$ whose fiber at the point
$(\CW_\bullet)$ equals $\det R\Gamma(\bC,\CW_i)$. We also define the line bundle
$\CalD^{\unl{a}}$ on $\fQ_{\unl{d}}$ via
$\CalD^{\unl{a}}:=\otimes_{i=1}^{n-1}\CalD_i^{-a_i}$.
Note that $\CalD_1$ is a pull-back of the first line bundle on the Drinfeld
compactification and therefore is trivial, which explains the irrelevance of
our choice of $a_1$. Finally, we introduce the constants
\begin{multline*}
  X(\unl{d}):=\prod_{i=1}^{n-1} ((1-v^2)c_i)^{d_i}
  \prod_{p=1}^{n-1}(t_1\cdots t_p)^{-2a_p-\sum_{i=1}^{n-1} d_i \n_{ip}}
  \prod_{p=1}^{n-1} t_p^{d_{p-1}-2a_pd_p}\times\\
  v^{\sum_{i=1}^{n-1} \left((\n_{ii}+i)d_i-2a_{i+1}d_id_{i+1}+\frac{d_i(d_i-1)}{2}(\n_{ii}+2a_i+1)\right)+
  \sum_{i<j} \n_{ji}d_id_j-\sum_{i,p=1}^{n-1} \frac{p(p-1)}{2} d_i\n_{ip}}.
\end{multline*}

The following is the key result of this section:

\begin{Thm}\label{geometric Whittaker}
Define $\theta_{\unl{d}}:=X(\unl{d})\cdot [\CalD^{\unl{a}}]\in M_{\unl{d}}$ and set
$\theta:=\sum_{\unl{d}}\theta_{\unl{d}}$. Then, $e_i(\theta)=c_i\cdot \theta$ for
any $1\leq i\leq n-1$.
\end{Thm}

\begin{Rem}\label{comment on Whit}
(a) Due to Theorem~\ref{Br-Fin}, this provides a geometric realization of all
Whittaker vectors (associated with Sevostyanov triples) of the universal Verma
module $\CV$ over $U_v(\ssl_n)$.

\noindent
(b) It is straightforward to verify that $\theta$ does not depend on the choice of
$a_1\in \{0,1\}$.
\end{Rem}

\begin{proof}
According to the Bott-Lefschetz formula, we have:

$\bullet$
  $\theta=\sum_{\wt{\unl{d}}} a_{\wt{\unl{d}}}\cdot X(\unl{d})\cdot
   (\CalD^{\unl{a}})_{\mid_{\wt{\unl{d}}}}\cdot[\wt{\unl{d}}],\ \mathrm{where}\
   a_{\wt{\unl{d}}}=\prod_{w\in T_{\wt{\unl{d}}}\fQ_{\unl{d}}}(1-w)^{-1};$

$\bullet$
  $\frac{a_{\wt{\unl{d}{}'}}}{a_{\wt{\unl{d}}}}(\bp_*\bq^*)_{[\wt{\unl{d}{}'},\wt{\unl{d}}]}=
   (\bq_*\bp^*)_{[\wt{\unl{d}},\wt{\unl{d}{}'}]}.$

\noindent
According to Theorem~\ref{Br-Fin}(c, d), we have:

$\bullet$
  $(\bq_*\bp^*)_{[\wt{\unl{d}},\wt{\unl{d}}+\delta_{ij}]}=
   \frac{1}{1-v^2}\prod_{j\ne k\leq i} (1-s_{ij}/s_{ik})^{-1}\prod_{k\leq i-1}(1-s_{ij}/s_{i-1,k});$

$\bullet$
  $(\prod_{p=1}^{n-1}L_p^{\n_{ip}})[\wt{\unl{d}}+\delta_{ij}]=
   v^{-\n_{ii}}\prod_{p=1}^{n-1}\left(t_1\cdots t_pv^{-d_p+\frac{p(p-1)}{2}}\right)^{\n_{ip}}\cdot [\wt{\unl{d}}+\delta_{ij}];$

\noindent
where $s_{ij}=t_j^2v^{-2d_{ij}}$ as before.
Finally, we also have:

$\bullet$
 $(\CalD^{\unl{a}})_{\mid_{\wt{\unl{d}}+\delta_{ij}}}/(\CalD^{\unl{a}})_{\mid_{\wt{\unl{d}}}}=s_{ij}^{a_i}$.

Therefore, it suffices to prove the following equality
for any $\wt{\unl{d}}$ and $1\leq i\leq n-1$:
\begin{equation*}\label{Whittaker condition}
  \frac{X(\unl{d}+i)}{X(\unl{d})} \frac{t_{i+1}^{-1}v^{d_{i+1}-d_i-i}}{1-v^2} v^{-\n_{ii}}
  \prod_{p=1}^{n-1}\left(t_1\cdots t_pv^{-d_p+\frac{p(p-1)}{2}}\right)^{\n_{ip}}\cdot \sum_{j\leq i}s_{ij}^{a_i}
  \frac{\prod_{k\leq i-1}(1-s_{ij}/s_{i-1,k})}{\prod_{k\leq i}^{k\ne j}(1-s_{ij}/s_{ik})}=c_i.
\end{equation*}

\begin{Lem}\label{residues}
For any $\wt{\unl{d}}$ and $1\leq i\leq n-1$, the following equality holds:
\begin{equation}\label{combinatorial equality}
  \sum_{j\leq i}s_{ij}^{a_i} \frac{\prod_{k\leq i-1}(1-s_{ij}/s_{i-1,k})}
  {\prod_{k\leq i}^{k\ne j}(1-s_{ij}/s_{ik})}=(t_i^2v^{2d_{i-1}-2d_i})^{a_i}.
\end{equation}
\end{Lem}

\begin{proof}
First, let us rewrite the left-hand side of~(\ref{combinatorial equality}) as
\begin{equation*}
  \frac{s_{i1}\cdots s_{ii}}{s_{i-1,1}\cdots s_{i-1,i-1}}\cdot
  \sum_{j\leq i} s_{ij}^{a_i-1}\frac{\prod_{k=1}^{i-1} (s_{i-1,k}-s_{ij})}{\prod_{k\leq i}^{k\ne j} (s_{ik}-s_{ij})}.
\end{equation*}

If $a_i=1$, then the above sum
  $\sum_{j\leq i} s_{ij}^{a_i-1}\frac{\prod_{k=1}^{i-1} (s_{i-1,k}-s_{ij})}{\prod_{k\leq i}^{k\ne j} (s_{ik}-s_{ij})}$
is a rational function in $\{s_{ij}\}_{j=1}^i$ of degree $0$ and without poles,
hence, a constant. To evaluate this constant, let $s_{ii}\to \infty$, in which case
the first $i-1$ summands tend to zero, while the last one tends to $1$. Hence,
this constant is $1$, and the left-hand side of ~(\ref{combinatorial equality}) equals
  $\frac{s_{i1}\cdots s_{ii}}{s_{i-1,1}\cdots s_{i-1,i-1}}=t_i^2v^{2d_{i-1}-2d_i}$.

If $a_i=0$, then
  $\sum_{j\leq i} s_{ij}^{-1}\frac{\prod_{k=1}^{i-1} (s_{i-1,k}-s_{ij})}{\prod_{k\leq i}^{k\ne j} (s_{ik}-s_{ij})}-
    \sum_{j\leq i} s_{ij}^{-1}\frac{\prod_{k=1}^{i-1} s_{i-1,k}}{\prod_{k\leq i}^{k\ne j} (s_{ik}-s_{ij})}=0$
as the left-hand side is a rational function in $\{s_{ij}\}_{j=1}^i$ of degree $-1$
and without poles. Thus, the left-hand side of~(\ref{combinatorial equality}) equals
  $\sum_{j\leq i}\frac{\prod_{k\leq i}^{k\ne j} s_{ik}}{\prod_{k\leq i}^{k\ne j} (s_{ik}-s_{ij})}$.
This is a rational function in $\{s_{ij}\}_{j=1}^i$ of degree $0$ and without poles,
hence, a constant. Specializing $s_{ii}\mapsto 0$, we see that this constant equals $1$
(as the first $i-1$ summands specialize to $0$, while the last one specializes to $1$).
\end{proof}

Due to Lemma~\ref{residues}, it remains to verify
\begin{equation*}
  \frac{X(\unl{d}+i)}{X(\unl{d})}=(1-v^2)c_i t_{i+1} v^{d_i-d_{i+1}+i} v^{\n_{ii}}
  \prod_{p=1}^{n-1}\left(t_1\cdots t_pv^{-d_p+\frac{p(p-1)}{2}}\right)^{-\n_{ip}}(t_i^2v^{2d_{i-1}-2d_i})^{-a_i},
\end{equation*}
which is  straightforward. This completes our proof of Theorem~\ref{geometric Whittaker}.
\end{proof}

\begin{Rem}\label{comparison to BF}
Note that if $\epsilon_{i,i+1}=-1$ (resp.\ $\epsilon_{i,i+1}=1$) for all $i$,
then $\theta$ is a linear combination of $[\CO_{\fQ_{\unl{d}}}]$
(resp.\ $[\CalD_{\unl{d}}^{-1}]$ with $\CalD_{\unl{d}}:=\otimes_{i=1}^{n-1}\CalD_i$).
These are exactly the two cases considered in~\cite{BF}.
\end{Rem}


\subsection{Geometric realization of the $J$-function}
\

Recall the Shapovalov form $(\cdot,\cdot)$ on the universal Verma module $\CV$,
which is a unique non-degenerate symmetric bilinear form on $\CV$ with values in
$\sk\simeq \on{Frac}(K^{\wt{T}\times\BC^\times}(\on{pt}))\otimes_{\BC(v)} \BC(v^{1/\NN})$
characterized by $(\textbf{1},\textbf{1})=1$ and $(xw,w')=(w,\wt{\sigma}(x)w')$ for all
$w,w'\in \CV, x\in U_v(\ssl_n)$, where $\wt{\sigma}$ is the antiautomorphism of $U_v(\ssl_n)$
determined by
  $\wt{\sigma}(E_i)=F_i, \wt{\sigma}(F_i)=E_i, \wt{\sigma}(L_i)=L_i$.

Identifying $\CV\cong M$ via Theorem~\ref{Br-Fin}(b), a geometric expression for the Shapovalov
form was obtained in~\cite[Proposition 2.29]{BF}\footnote{Note that our formula differs
from the one of~\cite{BF} as we use a slightly different action of $U_v(\ssl_n)$.}:

\begin{Prop}\label{geometric Shapovalov}
If $\unl{d}\ne \unl{d}'$, then $M_{\unl{d}}$ is orthogonal to $M_{\unl{d}'}$.
For $\CF,\CF'\in M_{\unl{d}}$, we have
\begin{equation}\label{pairing}
  (\CF,\CF')=(-1)^{\sum_{i=1}^{n-1}d_i}v^{\sum_{i=1}^{n-1} (d_id_{i+1}-d_i^2+(1-2i)d_i)}
  \prod_{i=1}^n t_i^{d_i-d_{i-1}}\cdot
  [R\Gamma(\fQ_{\unl{d}}, \CF\otimes \CF'\otimes \CalD_{\unl{d}})],
\end{equation}
where $\CalD_{\unl{d}}=\otimes_{i=1}^{n-1}\CalD_i$ as in Remark~\ref{comparison to BF}.
\end{Prop}

Given a pair of Sevostyanov triples $(\epsilon^\pm, \n^\pm, c^\pm)$,
choose the corresponding $\unl{a}^\pm\in \{0,1\}^{n-1}$ and
$X(\unl{d})^\pm\in \on{Frac}(K^{\wt{T}\times\BC^\times}(\on{pt}))$,
and define vectors
  $\theta^\pm_{\unl{d}}:=X(\unl{d})^\pm[\CalD^{\unl{a}^\pm}]\in M_{\unl{d}}$
as in Section~\ref{section geom whit}. Consider the following generating function:
\begin{equation}\label{geometric J}
  \fJ=\fJ(y_1,\ldots,y_{n-1}):=\prod_{i=1}^{n-1}y_i^{-\frac{\log(t_1\cdots t_i)}{\log(v)}-\frac{i(n-1)}{2}}\cdot
  \sum_{\unl{d}}\left(\theta^+_{\unl{d}},\theta^-_{\unl{d}}\right)y_1^{d_1}\cdots y_{n-1}^{d_{n-1}}.
\end{equation}
Due to~(\ref{pairing}), the coefficient $(\theta^+_{\unl{d}},\theta^-_{\unl{d}})$ equals
\begin{multline*}
  (-1)^{\sum_{i=1}^{n-1}d_i}v^{\sum_{i=1}^{n-1} (d_id_{i+1}-d_i^2+(1-2i)d_i)}
  X(\unl{d})^+X(\unl{d})^-\prod_{i=1}^n t_i^{d_i-d_{i-1}}\cdot
  [R\Gamma(\fQ_{\unl{d}}, \otimes_{i=1}^{n-1}\CalD_i^{1-a^+_i-a^-_i})].
\end{multline*}

The following result is an immediate consequence of Theorems~\ref{J-eigenfunction},~\ref{Br-Fin},
and Remark~\ref{comment on Whit}:

\begin{Thm}\label{geometric J-function}
$\fJ$ is an eigenfunction of the type $A_{n-1}$ modified quantum difference Toda system
$\CT(\pm\epsilon^\pm,\pm\n^\pm,c^\pm)$. In particular, for $\D_1$ computed explicitly
in~(\ref{FinalD for A-type}), we have
\begin{equation}\label{D and J}
  \D_1(\fJ)=\left(v^{n-1}\sum_{i=1}^n t_i^2\right)\cdot \fJ.
\end{equation}
\end{Thm}

\begin{Rem}\label{Shapovalov vs dual}
(a) There is an algebra isomorphism $\varsigma\colon U_v(\ssl_n)\iso U_{v^{-1}}(\ssl_n)$
determined by $E_i\mapsto \bar{E}_i, F_i\mapsto \bar{F}_i, L_i\mapsto \bar{L}^{-1}_i$.
Note that $\sigma=\varsigma\circ \wt{\sigma}$ (with $\sigma$ defined in Section~\ref{Whit vector defn})
and the action of $U_v(\ssl_n)$ on $\bar{\CV}$ (as a $\varsigma$-pull-back of $U_{v^{-1}}(\ssl_n)$-action)
is isomorphic to $\CV$. This implies that the Shapovalov form on $\CV$ is identified with
the $\sk$-bilinear form on $\CV\times \bar{\CV}$ of Section~\ref{Whit vector defn}.

\noindent
(b) Under the identification of part (a), the Whittaker vector of $\bar{\CV}$ associated
to a Sevostyanov triple $(\epsilon^-,\n^-,c^-)$ becomes the Whittaker vector of $\CV$ associated
to the Sevostyanov triple $(-\epsilon^-,-\n^-,c^-)$. This explains the appearance of the sign `--'
in front of $\epsilon^-,\n^-$ in Theorem~\ref{geometric J-function}.
\end{Rem}


\subsection{B.~Feigin's viewpoint via $U_v(L\ssl_n)$-action}
\

According to~\cite[Theorem 2.12]{T} (see also~\cite[Theorem 12.7]{FT}), the action of
$U_v(\ssl_n)$ on $M$ can be extended to an action of the quantum loop algebra $U_v(L\ssl_n)$
on $M$.\footnote{Actually, this action factors through the one of $U_v(\gl_n)$
(extending the $U_v(\ssl_n)$-action from Theorem~\ref{Br-Fin}) via
the evaluation homomorphism $\mathrm{ev}\colon U_v(L\ssl_n)\to U_v(\gl_n)$.}
In particular, loop generators $\{e_{i,r},f_{i,r}\}_{1\leq i\leq n-1}^{r\in \BZ}$ (see~\cite[2.10]{T})
act via
\begin{equation*}
  e_{i,r}=t_{i+1}^{-1}v^{d_{i+1}-d_i+1-i}\bp_*((v^i\CL_i)^{\otimes r}\otimes \bq^*)\colon
  M_{\unl{d}}\to M_{\unl{d}-i},
\end{equation*}
\begin{equation*}
  f_{i,r}=-t_i^{-1}v^{d_i-d_{i-1}+i}\bq_*(\CL_i\otimes (v^i\CL_i)^{\otimes r}\otimes \bp^*)\colon
  M_{\unl{d}}\to M_{\unl{d}+i}.
\end{equation*}
Note that $e_{i,0}=E_i$ and $f_{i,0}=F_i$. Following~\cite{BF}, define
$\fk\in M^\wedge:=\prod_{\unl{d}} M_{\unl{d}}\cong \CV^\wedge$ via
\begin{equation}\label{simplest Whittaker}
  \fk:=\sum_{\unl{d}} \fk_{\unl{d}}\
  \mathrm{with}\ \fk_{\unl{d}}:=[\CO_{\fQ_{\unl{d}}}]\in M_{\unl{d}}.
\end{equation}

\begin{Prop}\label{eigen-property}
(a) For any $1\leq i\leq n-1$, we have $e_{i,0}L_i^{-1}L_{i+1}(\fk)=\frac{v}{1-v^2}\fk$.

\noindent
(b) For any $1\leq i\leq n-1$, we have $e_{i,1}L_{i-1}^2L_i^{-3}L_{i+1}(\fk)=\frac{v^{5-i}}{1-v^2}\fk$.
\end{Prop}

\begin{proof}
Part (a) follows from Theorem~\ref{geometric Whittaker} (see also~\cite[Proposition 12.21]{FT}).
The proof of part (b) is completely analogous to our proof of Theorem~\ref{geometric Whittaker}
(see also~\cite[Remark 12.22(b)]{FT}).
\end{proof}

Let us now explain the relation between Proposition~\ref{eigen-property} regarding the
``eigen-property'' of the (geometrically) simplest Whittaker vector $\fk$ and the geometric
description of the general Whittaker vectors from Theorem~\ref{geometric Whittaker}.
In what follows, we will view the line bundle $\CalD_i^{\pm 1}$ as an endomorphism of $M$
given by the multiplication by $[\CalD_i^{\pm 1}]$.

\begin{Prop}\label{D conjugacy one}
We have the following equalities in $\End(M)$:

\noindent
(a) $\CalD_i e_{j,0} \CalD_i^{-1}=e_{j,0}$ for $j\ne i$,

\noindent
(b) $\CalD_i e_{i,0} \CalD_i^{-1}=v^{-i} e_{i,1}$.
\end{Prop}

\begin{proof}
According to~\cite[Corollary 6.5(a)]{FFFR}, the operator $\CalD_i$ is diagonal in the
fixed point basis $\{[\wt{\unl{d}}]\}$, and the eigenvalue at $[\wt{\unl{d}}]$ is equal to
  $\prod_{k=1}^i t_k^{2(1-d_{ik})}v^{d_{ik}(d_{ik}-1)}$.
Part (a) follows as $e_{j,0}\colon M_{\unl{d}}\to M_{\unl{d}-j}$.
Likewise, the only non-zero matrix coefficients of $\CalD_i e_{i,0} \CalD_i^{-1}$ are given by
  ${\CalD_i e_{i,0} \CalD_i^{-1}}_{[\wt{\unl{d}},\wt{\unl{d}}-\delta_{ij}]}=
   t_j^2v^{-2(d_{ij}-1)}\cdot e_{i,0 [\wt{\unl{d}},\wt{\unl{d}}-\delta_{ij}]}=
   v^2s_{ij}\cdot e_{i,0 [\wt{\unl{d}},\wt{\unl{d}}-\delta_{ij}]}$,
where $s_{ij}=t_j^2v^{-2d_{ij}}$ as before. According to~\cite[Proposition 2.15]{T},
the only nonzero matrix coefficients of $e_{i,1}$ in the fixed point basis are given by
  $e_{i,1 [\wt{\unl{d}},\wt{\unl{d}}-\delta_{ij}]}=
   v^{i+2}s_{ij}\cdot e_{i,0 [\wt{\unl{d}},\wt{\unl{d}}-\delta_{ij}]}$.
Part (b) follows.
\end{proof}

\begin{Cor}\label{Feigin's relation}
For any $\unl{a}=(a_1,\ldots,a_{n-1})\in \{0,1\}^{n-1}$,
the following holds in $\mathrm{End}(M)$:
\begin{equation}\label{conjugation one}
  (\CalD^{\unl{a}})^{-1} e_{i,0} \CalD^{\unl{a}}=
  \begin{cases}
    e_{i,0}, & \mbox{if } a_i=0,\\
    v^{-i}e_{i,1}, & \mbox{if } a_i=1.
  \end{cases}
\end{equation}
\end{Cor}

For $\unl{a}\in \{0,1\}^{n-1}$, define $\fk^{\unl{a}}\in M^\wedge$ via
\begin{equation}\label{next simplest Whittaker}
  \fk^{\unl{a}}:=\sum_{\unl{d}} \fk^{\unl{a}}_{\unl{d}}\
  \mathrm{with}\ \fk^{\unl{a}}_{\unl{d}}:=\prod_{i=1}^{n-1}(t_1\cdots t_i)^{-2a_i} \cdot [\CalD^{\unl{a}}]\in M_{\unl{d}}.
\end{equation}
Note that $\fk^{\unl{0}}=\fk$. The special case of Theorem~\ref{geometric Whittaker}
follows immediately from Proposition~\ref{eigen-property}:

\begin{Prop}\label{Feigin's relation continued}
$\fk^{\unl{a}}\in M^\wedge\cong \CV^\wedge$ is the Whittaker vector
corresponding to the Sevostyanov triple $(\epsilon,\n,c)$ with
  $\epsilon_{i,i+1}=2a_{i+1}-1,
   \n_{ij}=\delta_{j,i+1}-(1+2a_i)\delta_{j,i}+2a_i\delta_{j,i-1},
   c_i=\frac{v^{1+a_i(4-2i)}}{1-v^2}.$
\end{Prop}

\begin{proof}
Since $\fk^{\unl{a}}_{\unl{0}}=[\CO_{\fQ_{\unl{0}}}]$,
it remains to verify the following equality for any $1\leq i\leq n-1$:
\begin{equation}\label{conjugation two}
  e_{i,0}L_{i-1}^{2a_i}L_i^{-1-2a_i}L_{i+1}([\CalD^{\unl{a}}])=\frac{v^{1+a_i(4-2i)}}{1-v^2}\cdot [\CalD^{\unl{a}}].
\end{equation}
This follows by combining formula~(\ref{conjugation one}) with Proposition~\ref{eigen-property}.
\end{proof}


\begin{Rem}
We note that the above operator of multiplication by $[\CalD_k]$ can be interpreted
entirely algebraically as a product of the Drinfeld Casimir element of the subalgebra
$U_v(\ssl_k)\subset U_v(\ssl_n)$ and a certain Cartan element, due
to~\cite[Corollary 6.5(b)]{FFFR}. In~\emph{loc.cit.}, the authors choose to work with
the action of $U_v(\gl_n)$ instead of $U_v(\ssl_n)$, which results in shorter formulas.
\end{Rem}

We conclude this section by generalizing the construction of~\cite[Theorem 6.12]{DFKT}.
In~\emph{loc.cit.}, the authors established an edge-weight path model for the type $A_{n-1}$
Whittaker vector associated with a particular Sevostyanov triple $(\epsilon,\n,c)$ with
$\epsilon_{i,i+1}=1\ (1\leq i\leq n-2)$  and $\n_{ij}=(i-1)(\delta_{j,i+1}-2\delta_{j,i}+\delta_{j,i-1})$.
More generally, their construction can be applied to Whittaker vectors associated with
$(\epsilon,\n,c)$ satisfying $\epsilon_{1,2}=\epsilon_{2,3}=\ldots=\epsilon_{n-2,n-1}$
(corresponding to an equioriented $A_{n-1}$ Dynkin diagram). In particular, identifying
$\CV\cong M$, we obtain the following edge-weight path model for the Whittaker vector
$\fk\in M^\wedge$ of~(\ref{simplest Whittaker}):

\begin{Prop}\label{Path model 1}
The following equality holds:
\begin{equation}\label{path 1}
  \fk=\sum_{\beta\in Q_+} \left(\frac{v}{1-v^2}\right)^{|\beta|}\sum_{\bP\in \CP_\beta} y(\bP)\cdot |\bP\rangle,
\end{equation}
where we use the following notations:

$\bullet$
$|\beta|:=\sum_{i=1}^{n-1} m_i$ for $\beta=\sum_{i=1}^{n-1} m_i\alpha_i\in Q_+$,

$\bullet$
the set $\CP_\beta$ consists of all paths $\bP=(p_0,\ldots,p_N)$ such that
$p_0=0,\ p_N=\beta,$ and $p_k-p_{k-1}=\alpha_{i_k} (1\leq i_k\leq n-1)$
for all $1\leq k\leq N$,

$\bullet$
for $\bP=(p_0,\ldots,p_N)\in \CP_\beta$, the vector $|\bP\rangle\in M$ is defined as
$|\bP\rangle:=f_{i_N}f_{i_{N-1}}\cdots f_{i_1}([\CO_{\fQ_{\unl{0}}}])$ with $f_i:=L_iL_{i+1}^{-1}F_i$,

$\bullet$
for $\bP=(p_0,\ldots,p_N)\in \CP_\beta$ with $p_k-p_{k-1}=\alpha_{i_k}$,
the coefficient $y(\bP)$ edge-factorizes as
$y(\bP)=\prod_{k=1}^N \frac{1}{\fv^{(i_k)}(p_k)}$, where
$\fv^{(i)}(\gamma)=v^{\tau_i(\gamma)} \fv(\gamma)$
with $\tau_i(\gamma)=(\lambda+\rho-\gamma,\omega_{i-1}-\omega_{i+1})$ and
  $\fv(\gamma)=(v-v^{-1})^{-2}
   \sum_{i=0}^{n-1}\left(v^{2(\lambda+\rho,\omega_{i+1}-\omega_i)}-v^{2(\lambda+\rho-\gamma,\omega_{i+1}-\omega_i)}\right)$
for $1\leq i\leq n-1$ and $\gamma\in Q_+$. Here $\omega_i$ is the $i$-th fundamental
weight of $\ssl_n$ as before, and we set $\omega_0:=0,\ \omega_n:=0$.
\end{Prop}

Noteworthy, there seems to be no such straightforward edge-weight path model for a general
type $A$ Whittaker vector. Nevertheless, one can fix this by changing the above definition
of $|\bP\rangle$ with the help of the quantum loop algebra $U_v(L\ssl_n)$ in spirit of
Proposition~\ref{D conjugacy one} and Corollary~\ref{Feigin's relation}.
This is based on the following result:

\begin{Prop}\label{D conjugacy two}
For any $\unl{a}=(a_1,\ldots,a_{n-1})\in \{0,1\}^{n-1}$, the following holds in $\mathrm{End}(M)$:
\begin{equation}\label{conjugation three}
  \CalD^{\unl{a}} f_{i,0} (\CalD^{\unl{a}})^{-1}=
  \begin{cases}
    f_{i,0}, & \mbox{if } a_i=0, \\
    v^{-i}f_{i,1}, & \mbox{if } a_i=1,
  \end{cases}
\end{equation}
where $\CalD^{\unl{a}}$ denotes an endomorphism of $M$ given
by the multiplication by $[\CalD^{\unl{a}}]$.
\end{Prop}

\begin{proof}
The proof is completely analogous to that of Proposition~\ref{D conjugacy one}.
\end{proof}

Recall the element $\fk^{\unl{a}}\in M^\wedge$ of~(\ref{next simplest Whittaker}).
Since $[\CalD^{\unl{a}}]=\CalD^{\unl{a}}(\fk)$, we obtain the following edge-weight
path model for $\fk^{\unl{a}}$:

\begin{Prop}\label{Path model 2}
The following equality holds:
\begin{equation}\label{path 2}
  \fk^{\unl{a}}=\prod_{i=1}^{n-1}(t_1\cdots t_i)^{-2a_i} \cdot
  \sum_{\beta\in Q_+} \left(\frac{v}{1-v^2}\right)^{|\beta|}\sum_{\bP\in \CP_\beta} y(\bP)\cdot |\bP\rangle^{\unl{a}},
\end{equation}
where for a path $\bP=(p_0,\ldots,p_N)\in \CP_\beta$ with $p_k-p_{k-1}=\alpha_{i_k}$ we set
\begin{equation*}
 |\bP\rangle^{\unl{a}}:=f^{\unl{a}}_{i_N}f^{\unl{a}}_{i_{N-1}}\cdots f^{\unl{a}}_{i_1}([\CO_{\fQ_{\unl{0}}}])
 \ \mathrm{with}\
 f^{\unl{a}}_i:=
  \begin{cases}
    L_iL_{i+1}^{-1}f_{i,0}, & \mbox{if } a_i=0, \\
    v^{-i}L_iL_{i+1}^{-1}f_{i,1}, & \mbox{if } a_i=1.
  \end{cases}
\end{equation*}
\end{Prop}

\begin{proof}
Follows by combining Propositions~\ref{Path model 1} and~\ref{D conjugacy two}.
\end{proof}

According to Proposition~\ref{Feigin's relation continued}, $\fk^{\unl{a}}$ is the
Whittaker vector corresponding to the Sevostyanov triple $(\epsilon,\n,c)$ with
  $\epsilon_{i,i+1}=2a_{i+1}-1,
   \n_{ij}=\delta_{j,i+1}-(1+2a_i)\delta_{j,i}+2a_i\delta_{j,i-1},
   c_i=\frac{v^{1+a_i(4-2i)}}{1-v^2}$.
As $\unl{a}\in \{0,1\}^{n-1}$ varies, we get all possible orientations $\Or$ of
$\Dyn(\ssl_n)=A_{n-1}$ (here, $\Or$ is determined by $\epsilon$). Since it is clear
how the edge-weight path model gets modified once we change $\n,c$
(while $\epsilon$ is kept fixed), cf.~\cite[(3.8, 3.9)]{FFJMM}, Proposition~\ref{Path model 2}
provides an edge-weight path model for a general type $A$ Whittaker vector.


\appendix


\section{Proof of Proposition~\ref{bound for A}}\label{Proof of bound for A}

Given two pairs of type $A_{n-1}$ Sevostyanov triples $(\epsilon^\pm,\n^\pm,c^\pm)$
and $(\tilde{\epsilon}^\pm,\tilde{\n}^\pm,\tilde{c}^\pm)$ such that
  $\epsilon^+_{i,i+1}-\epsilon^-_{i,i+1}=\tilde{\epsilon}^+_{i,i+1}-\tilde{\epsilon}^-_{i,i+1}$
for $1\leq i\leq n-2$, we will prove that there exist constants
$\{r_{ij},r_i\}_{1\leq i\leq j\leq n}$ such that the function
\begin{equation}\label{Canonical Form}
  F=F(\sw_1,\ldots,\sw_n):=
  \exp\left(\sum_{1\leq i\leq j\leq n} r_{ij}\log(\sw_i)\log(\sw_j)+\sum_{1\leq i\leq n} r_i\log(\sw_i)\right)
\end{equation}
satisfies the equality
\begin{equation}\label{conjugation A-type}
  F^{-1}\HH(\epsilon^\pm,\n^\pm,c^\pm)F=\HH(\tilde{\epsilon}^\pm,\tilde{\n}^\pm,\tilde{c}^\pm).
\end{equation}
We will view this as an equality in $\CA_n$ (rather than $\bar{\CA}_n$),
treating $\HH$ of~(\ref{H for A-type}) as elements of $\CA_n$. This will
immediately imply the result of Proposition~\ref{bound for A}.
Set $\hbar:=\log(v)$.

$\bullet$
First, we note that the terms without $\sD_i$'s are the same (and equal to
$\sum_{j=1}^n \sw_j^{-2}$) both in $F^{-1}\HH(\epsilon^\pm,\n^\pm,c^\pm)F$
and $\HH(\tilde{\epsilon}^\pm,\tilde{\n}^\pm,\tilde{c}^\pm)$, independently
of our choice of constants  $\{r_{ij},r_i\}$.

$\bullet$
Second, we will match the terms with $\{\frac{\sD_i}{\sD_{i+1}}\}_{i=1}^{n-1}$
appearing in $F^{-1}\HH(\epsilon^\pm,\n^\pm,c^\pm)F$ and
$\HH(\tilde{\epsilon}^\pm,\tilde{\n}^\pm,\tilde{c}^\pm)$.
Their equality is equivalent to the following system of equations on $\{r_{ij}\}$:
\begin{equation}\label{system A.1}
  \frac{\m_{ij}-\tilde{\m}_{ij}}{\hbar}=
  \begin{cases}
    r_{ji}-r_{j,i+1}, &\ \mathrm{if}\ 1\leq j< i\\
    r_{ij}-r_{i+1,j}, &\ \mathrm{if}\ i+2\leq j\leq n\\
    2r_{ii}-r_{i,i+1}, &\ \mathrm{if}\ j=i\\
    r_{i,i+1}-2r_{i+1,i+1}, &\ \mathrm{if}\ j=i+1
  \end{cases}
\end{equation}
and the following system of equations on $\{r_i\}$:
\begin{equation}\label{system A.2}
  r_i-r_{i+1}=\hbar(r_{i,i+1}-r_{ii}-r_{i+1,i+1})+\hbar^{-1}\log(\tilde{b}_i/b_i)+
  \sum_{k=1}^n(n-k+1/2)(\m_{ik}-\tilde{\m}_{ik}),
\end{equation}
where the coefficients $\m_{ij},\tilde{\m}_{ij}, b_i, \tilde{b}_i$ are defined as
in Section~\ref{explicit hamiltonians A} via

  $\m_{ij}:=\sum_{p=j}^{n-1}(\n^-_{ip}-\n^+_{ip}),$

  $\tilde{\m}_{ij}:=\sum_{p=j}^{n-1}(\tilde{\n}^-_{ip}-\tilde{\n}^+_{ip}),$

  $b_i:=(v-v^{-1})^2v^{\n^+_{ii}-\n^-_{ii}}c^+_ic^-_i,$

  $\tilde{b}_i:=(v-v^{-1})^2v^{\tilde{\n}^+_{ii}-\tilde{\n}^-_{ii}}\tilde{c}^+_i\tilde{c}^-_i.$

It suffices to show that~(\ref{system A.1}) admits a solution, since~(\ref{system A.2})
obviously admits a solution in terms of $r_i$ (unique up to a common constant).
Pick any $r_{11}$. Using the last two cases of~(\ref{system A.1}), we determine
uniquely $\{r_{i,i+1},r_{i+1,i+1}\}_{i=1}^{n-1}$. Using the first case
of~(\ref{system A.1}), we determine uniquely $r_{ij}$ for $j>i+1$.
The resulting collection $\{r_{ij}\}_{1\leq i\leq j\leq n}$ satisfies
the first, third, and fourth cases of~(\ref{system A.1}). It remains
to verify that it also satisfies the second case of~(\ref{system A.1}).
We prove this by induction in $j-i\geq 2$.

(a) If $j=i+2$, then
  $r_{i,i+2}-r_{i+1,i+2}=(r_{i,i+1}-2r_{i+1,i+1})+(2r_{i+1,i+1}-r_{i+1,i+2})-(r_{i,i+1}-r_{i,i+2})=
   \hbar^{-1}(\m_{i,i+1}+\m_{i+1,i+1}-\m_{i+1,i}-\tilde{\m}_{i,i+1}-\tilde{\m}_{i+1,i+1}+\tilde{\m}_{i+1,i})$.
Hence, it remains to prove
  $\m_{i,i+1}+\m_{i+1,i+1}-\m_{i+1,i}-\m_{i,i+2}=
   \tilde{\m}_{i,i+1}+\tilde{\m}_{i+1,i+1}-\tilde{\m}_{i+1,i}-\tilde{\m}_{i,i+2}$.
Since $\m_{st}-\m_{s,t+1}=\n^-_{st}-\n^+_{st}$, this is reduced to
  $\n^-_{i,i+1}-\n^+_{i,i+1}-\n^-_{i+1,i}+\n^+_{i+1,i}=
   \tilde{\n}^-_{i,i+1}-\tilde{\n}^+_{i,i+1}-\tilde{\n}^-_{i+1,i}+\tilde{\n}^+_{i+1,i}$.
The latter equality follows from $\n^\pm_{st}-\n^\pm_{ts}=\epsilon^\pm_{st}b_{st}$
and our assumption on the triples.

(b) If $j>i+2$, then
  $r_{ij}-r_{i+1,j}=(r_{i,j-1}-r_{i+1,j-1})+(r_{i+1,j-1}-r_{i+1,j})-(r_{i,j-1}-r_{ij})=
   \hbar^{-1}(\m_{i,j-1}+\m_{j-1,i+1}-\m_{j-1,i}-\tilde{\m}_{i,j-1}-\tilde{\m}_{j-1,i+1}+\tilde{\m}_{j-1,i})$.
Hence, it remains to prove
  $(\m_{i,j-1}-\m_{ij})-(\m_{j-1,i}-\m_{j-1,i+1})=
   (\tilde{\m}_{i,j-1}-\tilde{\m}_{ij})-(\tilde{\m}_{j-1,i}-\tilde{\m}_{j-1,i+1})$.
Similarly to (a), this is reduced to the proof of
  $b_{i,j-1}(\epsilon^-_{i,j-1}-\epsilon^+_{i,j-1})=
   b_{i,j-1}(\tilde{\epsilon}^-_{i,j-1}-\tilde{\epsilon}^+_{i,j-1})$.
The latter follows immediately from the equality $b_{i,j-1}=0$.

Thus, we have determined a collection of constants $\{r_{ij},r_i\}_{1\leq i\leq j\leq n}$
satisfying~(\ref{system A.1},~\ref{system A.2}).\footnote{This collection is uniquely
determined by a choice of $r_{11},r_1$. However, we note that the image of $F$ defined
via~(\ref{Canonical Form}) in $\bar{\CA}_n$ is independent of this choice.}

$\bullet$
Finally, it remains to verify that for $F$ of~(\ref{Canonical Form}) with
the constants $r_{ij},r_i$ chosen as above, the terms with
$\frac{\sD_i}{\sD_j}\ (j>i+1)$ in $F^{-1}\HH(\epsilon^\pm,\n^\pm,c^\pm)F$
and $\HH(\tilde{\epsilon}^\pm,\tilde{\n}^\pm,\tilde{c}^\pm)$ do coincide.
First, we note that the conditions
  $\epsilon^\pm_{i,i+1}=\ldots=\epsilon^\pm_{j-2,j-1}=\pm 1$
and
  $\tilde{\epsilon}^\pm_{i,i+1}=\ldots=\tilde{\epsilon}^\pm_{j-2,j-1}=\pm 1$
are equivalent under our assumption on the triples.
Pick $j>i+1$ such that either of these equivalent conditions is satisfied.
Then, the compatibility of the terms with $\frac{\sD_i}{\sD_j}$ is equivalent
to the following equality:
\begin{multline}\label{system A.3}
  \frac{F(\sw_1,\ldots, v\sw_i,\ldots,v^{-1}\sw_j,\ldots,\sw_n)}{F(\sw_1,\ldots,\sw_n)}=\\
  \prod_{k=1}^n \sw_k^{\sum_{s=i}^{j-1}(\m_{sk}-\tilde{\m}_{sk})}\cdot
  \prod_{s=i}^{j-1}\frac{\tilde{b}_s}{b_s}\cdot
  v^{\sum_{i\leq a<b\leq j-1} (\n^-_{ab}-\n^+_{ab}-\tilde{\n}^-_{ab}+\tilde{\n}^+_{ab})+\sum_{k=1}^n \sum_{s=i}^{j-1} \frac{n+1-2k}{2}(\m_{sk}-\tilde{\m}_{sk})}.
\end{multline}

We prove this by induction in $j-i$. Note that the $j=i+1$ counterpart of~(\ref{system A.3})
is just the compatibility of the terms with $\frac{\sD_i}{\sD_{i+1}}$, established in
the previous step. Writing the left-hand side of~(\ref{system A.3}) as a product
\begin{equation}\label{induction factorization}
  \frac{F(\sw_1,\ldots, v\sw_i,\ldots,v^{-1}\cdot v\sw_{j-1},v^{-1}\sw_j,\ldots,\sw_n)}{F(\sw_1,\ldots,v\sw_{j-1},v^{-1}\sw_j,\ldots,\sw_n)}\cdot
  \frac{F(\sw_1,\ldots, v\sw_{j-1},v^{-1}\sw_j,\ldots,\sw_n)}{F(\sw_1,\ldots,\sw_n)}
\end{equation}
and applying the induction assumption to both fractions of~(\ref{induction factorization}),
it is straightforward to see that we obtain the right-hand side of~(\ref{system A.3}).

\medskip
Thus, the function $F$ defined via~(\ref{Canonical Form}) with the constants
$\{r_{ij},r_i\}_{1\leq i\leq j\leq n}$ determined in our second step satisfies
the equality~(\ref{conjugation A-type}).
This completes our proof of Proposition~\ref{bound for A}.
$\square$

\begin{Rem}\label{Remark on other types}
(a) The proofs of Propositions~\ref{bound for C},~\ref{bound for D},~\ref{bound for B}
are analogous to the above proof of Proposition~\ref{bound for A}. In each case, there
exists a unique collection of constants $\{r_{ij},r_i\}_{1\leq i\leq j\leq n}$ such that
the function $F$ defined via~(\ref{Canonical Form}) satisfies the corresponding
equality~(\ref{conjugation A-type}). The  way we choose such constants
closely follows the above second step in our proof of Propositions~\ref{bound for A}
and is determined by matching up the coefficients of

$\bullet$ $\{\sD_i/\sD_{i+1}\}_{i=1}^{n-1}$ and $\sD_n^2$ for the type $C_n$,

$\bullet$ $\{\sD_i/\sD_{i+1}\}_{i=1}^{n-1}$ and $\sD_{n-1}\sD_n$ for the type $D_n$,

$\bullet$ $\{\sD_i/\sD_{i+1}\}_{i=1}^{n-1}$ and $\sD_n$ for the type $B_n$.

\noindent
Finally, it remains to check that the function $F$ defined via~(\ref{Canonical Form})
with thus determined $\{r_{ij},r_i\}_{1\leq i\leq j\leq n}$ conjugates each of the
remaining terms appearing in $\HH(\epsilon^\pm,\n^\pm,c^\pm)$ into the one of
$\HH(\tilde{\epsilon}^\pm,\tilde{\n}^\pm,\tilde{c}^\pm)$. This is verified by
induction similarly to the above last step in our proof of Proposition~\ref{bound for A}.

\noindent
(b) The proof of Proposition~\ref{bound for G2} is also analogous, but the constants
$r_{ij},r_i$ are determined by matching the terms with $\sD_1$ and $\sD_2$.
\end{Rem}


\section{Proof of Theorem~\ref{Main Thm A}}\label{Proof of Main Theorem A}

Assume that $\g$ is either of the classical type or $G_2$.
Given two pairs of Sevostyanov triples $(\epsilon^\pm,\n^\pm,c^\pm)$ and
$(\tilde{\epsilon}^\pm,\tilde{\n}^\pm,\tilde{c}^\pm)$ with
$\vec{\epsilon}=\vec{\tilde{\epsilon}}$ (defined right before
Theorem~\ref{Main Thm A}), we need to show that there exists an automorphism
of $\CalD_v(H^\ad)$ which maps $\CT(\epsilon^\pm,\n^\pm,c^\pm)$ to
$\CT(\tilde{\epsilon}^\pm,\tilde{\n}^\pm,\tilde{c}^\pm)$.

According to our proof of Proposition~\ref{bound for A} and
Remark~\ref{Remark on other types} (which states that the same argument applies
to all classical types and $G_2$), there exists a ``formal function'' $F$ of the
shift operators $T_\mu$ such that conjugation by $F$ is a well-defined automorphism
of $\CalD_v(H^\ad)$ satisfying
  $F\D_1(\epsilon^\pm,\n^\pm,c^\pm)F^{-1}=\D_1(\tilde{\epsilon}^\pm,\tilde{\n}^\pm,\tilde{c}^\pm)$.
It remains to prove the following result:

\begin{Prop}\label{Etingof argument}
For any $1\leq i\leq n$, we have
  $F\D_i(\epsilon^\pm,\n^\pm,c^\pm)F^{-1}=\D_i(\tilde{\epsilon}^\pm,\tilde{\n}^\pm,\tilde{c}^\pm)$.
\end{Prop}

\begin{proof}
Recall that $\D_i\in \CalD^\leq_v(H^\ad)$, where $\CalD^\leq_v(H^\ad)$ is the subalgebra of
$\CalD_v(H^\ad)$ generated by $\{e^{-\alpha_i},T_\mu|1\leq i\leq n, \mu\in P\}$.
Let us extend the field $\BC(v^{1/\NN})$ to $\sk$ and recall the vector space $N_\lambda$
of Section~\ref{J-function}, which was equipped with a natural $\CalD_v(H^\ad)$-action.
In particular, the subspace $W_\lambda$ of $N_\lambda$ formed by the formal sums
  $\left\{\sum_{\beta\in Q_+} a_\beta \unl{y}^{\beta-\lambda}|a_\beta\in \sk\right\}$
is $\CalD^\leq_v(H^\ad)$-stable.
Moreover, $X(\unl{y}^{\beta-\lambda})$ contains only $\unl{y}^{\gamma-\lambda}$
with $\gamma\geq \beta$ for any $X\in \CalD^\leq_v(H^\ad)$, which we refer to
as the ``upper-triangular'' property of the $\CalD^\leq_v(H^\ad)$-action.
In particular, we have
\begin{equation}\label{diagonal values}
  \D_i(\epsilon^\pm,\n^\pm,c^\pm)(\unl{y}^{\beta-\lambda}),
  \D_i(\tilde{\epsilon}^\pm,\tilde{\n}^\pm,\tilde{c}^\pm)(\unl{y}^{\beta-\lambda})\in
  \left(\sum_{k=1}^{N_i}v^{2(\mu^{(i)}_k,\lambda-\beta)}\right)\cdot \unl{y}^{\beta-\lambda}
  \oplus \bigoplus_{\gamma>\beta}\sk \unl{y}^{\gamma-\lambda},
\end{equation}
where $N_i$ is the dimension  and $\{\mu^{(i)}_k\}_{k=1}^{N_i}$ are the weights
(counted with multiplicities) of the $i$-th fundamental $U_v(\g)$-representation $V_i$,
while $v^{(\nu,\lambda)}\ (\nu\in P)$ is defined as in Section~\ref{Whit vector defn}.

Therefore, the action of $\D_1(\tilde{\epsilon}^\pm,\tilde{\n}^\pm,\tilde{c}^\pm)$
on $W_\lambda$ is upper-triangular with pairwise distinct diagonal matrix coefficients,
hence, it is diagonalizable with a simple spectrum. Moreover, the eigenvalues are exactly
  $\left\{\sum_{k=1}^{N_1}v^{2(\mu^{(1)}_k,\lambda-\beta)}|\beta\in Q_+\right\}$.

\begin{Rem}\label{important rmk}
Due to Theorem~\ref{J-eigenfunction}, the corresponding eigenbasis consists of the $J$-functions
$J(\{y_i\}_{i=1}^n)$ associated with $\{\lambda-\rho-\beta|\beta\in Q_+\}$, cf.~(\ref{J function}).
\end{Rem}

Since
  $[\D_i(\tilde{\epsilon}^\pm,\tilde{\n}^\pm,\tilde{c}^\pm),\D_1(\tilde{\epsilon}^\pm,\tilde{\n}^\pm,\tilde{c}^\pm)]=0$,
the action of $\D_i(\tilde{\epsilon}^\pm,\tilde{\n}^\pm,\tilde{c}^\pm)$ on $W_\lambda$
is diagonal in a $\D_1(\tilde{\epsilon}^\pm,\tilde{\n}^\pm,\tilde{c}^\pm)$-eigenbasis
with the corresponding eigenvalues given by $\sum_{k=1}^{N_i}v^{2(\mu^{(i)}_k,\lambda-\beta)}$
(this also follows from Remark~\ref{important rmk} and Theorem~\ref{J-eigenfunction}).
On the other hand, the action of $F\D_i(\epsilon^\pm,\n^\pm,c^\pm)F^{-1}$ on $W_\lambda$
is also upper-triangular with the same diagonal matrix coefficients and commutes with
$\D_1(\tilde{\epsilon}^\pm,\tilde{\n}^\pm,\tilde{c}^\pm)$
(since $F\D_1(\epsilon^\pm,\n^\pm,c^\pm)F^{-1}=\D_1(\tilde{\epsilon}^\pm,\tilde{\n}^\pm,\tilde{c}^\pm)$
and $[\D_1(\epsilon^\pm,\n^\pm,c^\pm), \D_i(\epsilon^\pm,\n^\pm,c^\pm)]=0$).
 Thus, both $F\D_i(\epsilon^\pm,\n^\pm,c^\pm)F^{-1}$ and
$\D_i(\tilde{\epsilon}^\pm,\tilde{\n}^\pm,\tilde{c}^\pm)$ act diagonally in a
$\D_1(\tilde{\epsilon}^\pm,\tilde{\n}^\pm,\tilde{c}^\pm)$-eigenbasis and have
the same corresponding eigenvalues.

The equality
  $F\D_i(\epsilon^\pm,\n^\pm,c^\pm)F^{-1}=\D_i(\tilde{\epsilon}^\pm,\tilde{\n}^\pm,\tilde{c}^\pm)$
follows.
\end{proof}

Thus, conjugation by $F$ maps $\CT(\epsilon^\pm,\n^\pm,c^\pm)$ to
$\CT(\tilde{\epsilon}^\pm,\tilde{\n}^\pm,\tilde{c}^\pm)$. Theorem~\ref{Main Thm A} follows.
$\square$


\section{Proof of Theorem~\ref{Main Thm B}}\label{Proof of Main Theorem B}

Following the discussion in Appendix~\ref{Proof of Main Theorem A}, consider
a basis of $W_\lambda$ in which all $\D_i$ act simultaneously diagonally with
the corresponding eigenvalues given by $\sum_{k=1}^{N_i}v^{2(\mu^{(i)}_k,\lambda-\beta)}$.
The latter can be viewed as characters $\chi_i$ of the fundamental representations
evaluated at $v^{2(\lambda-\beta)}$.

Since the point $v^{2\lambda}\in H(\sk)$ is \emph{general} and the characters
$\{\chi_i\}_{i=1}^n$ are known to be algebraically independent, we immediately
obtain part (a) of Theorem~\ref{Main Thm B}.

Part (c) of Theorem~\ref{Main Thm B} follows from part (b) as
$\D_V(\epsilon^\pm,\n^\pm,c^\pm)\in \CalD^\leq_v(H^\ad)$ commutes with
$\D_1(\epsilon^\pm,\n^\pm,c^\pm)$ for any finite-dimensional
$U_v(\g)$-representation $V$, due to Lemma~\ref{commuting operators}(c).

It remains to prove part (b) of Theorem~\ref{Main Thm B}. The algebra $\CalD^\leq_v(H^\ad)$
is $\BZ$-graded via $\deg(T_\mu)=0$ and $\deg(e^{-\alpha_i})=-1$, so that the degree zero component
$\CalD^\leq_v(H^\ad)^0$ has a basis $\{T_\mu|\mu\in P\}$. Note that the degree zero component
$\D_i^{(0)}$ of $\D_i$ equals $\D_i^{(0)}=\sum_{k=1}^{N_i} T_{2\mu^{(i)}_k}$.
Let $\CalD^\leq_v(H^\ad)^0_{ev}$ be the subspace spanned by $\{T_{2\mu}|\mu\in P\}$.
We also consider a natural action of the Weyl group $W$ both on $\CalD^\leq_v(H^\ad)^0$
and  $\CalD^\leq_v(H^\ad)^0_{ev}$ via $w(T_\mu)=T_{w\mu}$ for $\mu\in P,w\in W$.
Given $\D\in \CalD^\leq_v(H^\ad)$ that commutes with $\D_1$, let $\D^{(0)}$ denote
its degree zero component.

\begin{Prop}\label{centralizer}
We have $\D^{(0)}\in (\CalD^\leq_v(H^\ad)^0_{ev})^W$.
\end{Prop}

The proof of Proposition~\ref{centralizer} is based on the rank $1$ case, for which we prove a slightly more
general result. In type $A_1$, the modified quantum difference Toda systems
are conjugate to the $q$-Toda of~\cite{E} with the first hamiltonian
  $\D_1=T_{2\varpi_1}+T_{-2\varpi_1}-(v-v^{-1})^2e^{-\alpha}T_0$,
see~(\ref{standard qToda for A}).

\begin{Lem}\label{centralizer sl2}
In type $A_1$, given $\D\in \CalD^\leq_v(H^\ad)$ that commutes with
$\D'=a_r\D_1^r+a_{r-1}\D_1^{r-1}+\ldots+a_0$ for some $a_0,\ldots,a_r\in \BQ(v^{1/\NN})$
with $a_r\ne 0,r>0$, $\D$ must be a polynomial in $\D_1$.
\end{Lem}

\begin{proof}
Let $\D=\D^{(0)}+e^{-\alpha}\D^{(-1)}+\ldots+e^{-s\alpha}\D^{(-s)}$ with
$\D^{(0)},\ldots,\D^{(-s)}\in \CalD^\leq_v(H^\ad)^0$ and $\D^{(-s)}\ne 0$.
We prove the claim by induction in $s$. Comparing the degree $-r-s$ terms
in $\D\D'=\D'\D$, we immediately get $\D^{(-s)}=c_sT_0$ for some constant $c_s$.
Replacing $\D$ by $\D-c_s(-(v-v^{-1})^{-2})^s\D_1^s$, we obtain another element
of $\CalD^\leq_v(H^\ad)$ which commutes with $\D'$ and has a smaller value of $s$,
hence, is a polynomial in $\D_1$ by the induction assumption.
Therefore, $\D$ is also a polynomial in $\D_1$.
\end{proof}

\begin{proof}[Proof of Proposition~\ref{centralizer}]
The result of Proposition~\ref{centralizer} follows immediately from Lemma~\ref{centralizer sl2}.
Indeed, it suffices to verify the following two claims for any $1\leq i\leq n$:

 (I) the operator $\D^{(0)}$ is invariant with respect to the simple reflection $s_i$,

 (II) every $\mu$ appearing in $\D^{(0)}$ satisfies $(\mu, \alpha_i)\in 2\sd_i\BZ$.

To prove this, consider a subspace $W'_\lambda$ of $W_\lambda$ that consists of
  $\left\{\sum_{\beta\in Q_+\backslash \BZ\alpha_i} a_\beta \unl{y}^{\beta-\lambda}|a_\beta\in \sk \right\}$.
It is stable under the action of $\CalD^\leq_v(H^\ad)$, hence, we obtain the action of
$\CalD^\leq_v(H^\ad)$ on the quotient $\bar{W}_\lambda:=W_\lambda/W'_\lambda$.
We also specialize $u_j\mapsto 1$ for $j\ne i$ (recall that $u_j$ were used in our definition
of $\lambda$). As a result, summands with $e^{-\alpha_j}\ (j\ne i)$ in $\D,\D_1$
act by zero on $\bar{W}_\lambda$, while $T_{\omega_j}\ (j\ne i)$ act by the identity operator.
Identifying further $\bar{W}_\lambda$ with the space $W^{(\ssl_2)}_{\lambda'}$ constructed
for $\ssl_2$ instead of $\g$ (hence, the superscript in our notations), $\D_1$ gives rise to the operator
$\D^{(\ssl_2)}_V$ with $V$ being the restriction of the first fundamental $U_v(\g)$-representation
$V_1$ to the subalgebra generated by $E_i,F_i,L^{\pm 1}_i$, which is isomorphic to $U_{v_i}(\ssl_2)$.
As $V$ is not a trivial $U_{v_i}(\ssl_2)$-module, $\D^{(\ssl_2)}_V$ is a non-constant polynomial
in the first hamiltonian $\D^{(\ssl_2)}_1$. Hence, Lemma~\ref{centralizer sl2} can be applied with
$\D'=\D^{(\ssl_2)}_V$ and $\D$ denoting the image of $\D$ acting on
$\bar{W}_\lambda\simeq W^{(\ssl_2)}_{\lambda'}$ by abuse of notation.
Therefore, both claims (I) and (II) follow.
\end{proof}

It is clear that $(\CalD^\leq_v(H^\ad)^0_{ev})^W$ is generated by
$\{\D_i^{(0)}\}_{i=1}^n$. Hence, due to Proposition~\ref{centralizer}, there exists
a polynomial $P$ in $n$ variables such that $\D':=\D-P(\D_1,\ldots,\D_n)$ is of strictly
negative degree. Thus, the action of $\D'$ on $W_\lambda$ is upper-triangular with
zeros on the diagonal. As $[\D',\D_1]=0$ and $\D_1$ acts on $W_\lambda$ with a simple
spectrum, we immediately get $\D'=0$.

This completes our proof of Theorem~\ref{Main Thm B}.
$\square$


\section{Proof of Theorem~\ref{Main Conj}}\label{Proof of Main Conjecture}

Given a rank $n$ simple Lie algebra $\g$, fix an arbitrary orientation of the edges
of $\Dyn(\g)$ as well as their labelling by numbers from $1$ up to $n-1$. For an edge
$1\leq e\leq n-1$, the vertices $t(e),h(e)$ will denote the \emph{tail} and the \emph{head}
of that edge, respectively. To every pair of Sevostyanov triples $(\epsilon^\pm,\n^\pm,c^\pm)$,
we associate an invariant $\vec{\epsilon}=(\epsilon_{n-1},\ldots,\epsilon_1)\in \{-1,0,1\}^{n-1}$ via
$\epsilon_e:=\frac{\epsilon^+_{t(e),h(e)}-\epsilon^-_{t(e),h(e)}}{2}\in \{-1,0,1\}$ for $1\leq e\leq n-1$.

To prove Theorem~\ref{Main Conj}, it suffices to verify that given two pairs of
Sevostyanov triples $(\epsilon^\pm,\n^\pm,c^\pm)$ and $(\tilde{\epsilon}^\pm,\tilde{\n}^\pm,\tilde{c}^\pm)$
satisfying $\vec{\epsilon}=\vec{\tilde{\epsilon}}$, there exists an automorphism of $\CalD_v(H^\ad)$
which maps $\CT(\epsilon^\pm,\n^\pm,c^\pm)$ to $\CT(\tilde{\epsilon}^\pm,\tilde{\n}^\pm,\tilde{c}^\pm)$.
Our proof is similar to that of Theorem~\ref{Main Thm A} presented in Appendix~\ref{Proof of Main Theorem A},
but is crucially based on the fermionic formula of Theorem~\ref{Fermionic two} for $\tilde{J}_\beta$ instead of
Propositions~\ref{bound for A},~\ref{bound for C},~\ref{bound for D},~\ref{bound for B},~\ref{bound for G2}.\footnote{We
owe this observation to A.~Braverman.}

Following Appendix~\ref{Proof of Main Theorem A}, consider the action of
$\CalD_v^\leq(H^\ad)$ on $W_\lambda$. Due to Remark~\ref{important rmk},
the action of pairwise commuting operators $\D_i(\epsilon^\pm,\n^\pm,c^\pm)$
(resp.\ $\D_i(\tilde{\epsilon}^\pm,\tilde{\n}^\pm,\tilde{c}^\pm)$) is
simultaneously diagonalizable in the basis of $J$-functions
$\{J^\Lambda(\epsilon^\pm,\n^\pm,c^\pm;\{y_i\})|\Lambda=\lambda-\rho-\beta,\beta\in Q_+\}$
(resp.\ $\{J^\Lambda(\tilde{\epsilon}^\pm,\tilde{\n}^\pm,\tilde{c}^\pm;\{y_i\})|
\Lambda=\lambda-\rho-\beta,\beta\in Q_+\}$).\footnote{As $\Lambda$ varies,
we will use the notations $J^\Lambda(\{y_i\}), J^\Lambda_\beta$ instead of
$J(\{y_i\}), J_\beta$ used in Section~\ref{section properties}.}
We note that both
$J^\Lambda(\epsilon^\pm,\n^\pm,c^\pm;\{y_i\})-\unl{y}^{\beta-\lambda}$ and
$J^\Lambda(\tilde{\epsilon}^\pm,\tilde{\n}^\pm,\tilde{c}^\pm;\{y_i\})-\unl{y}^{\beta-\lambda}$
contain only $\{\unl{y}^{\gamma-\lambda}\}_{\gamma>\beta}$.

The following is the key observation:

\begin{Prop}\label{Key point}
If $\vec{\epsilon}=\vec{\tilde{\epsilon}}$, there exists a difference operator $\fD$
which acts on $W_\lambda$ and maps $J^\Lambda(\epsilon^\pm,\n^\pm,c^\pm;\{y_i\})$ to
a non-zero multiple of $J^\Lambda(\tilde{\epsilon}^\pm,\tilde{\n}^\pm,\tilde{c}^\pm;\{y_i\})$
for any $\Lambda\in \lambda-\rho-Q_+$.
\end{Prop}

\begin{proof}
Define $\nu^\pm_i,\tilde{\nu}^\pm_i\in P$ via $\nu^\pm_i:=\sum_{k=1}^n \n^\pm_{ik}\omega_k$
and $\tilde{\nu}^\pm_i:=\sum_{k=1}^n \tilde{\n}^\pm_{ik}\omega_k$. Due to Theorem~\ref{Fermionic two},
the pairing $J^\Lambda_\beta(\epsilon^\pm,\n^\pm,c^\pm)$ depends only on $\{\nu^+_i-\nu^-_i\}_{i=1}^n$
and $\{c^+_ic^-_i\}_{i=1}^n$ for any fixed $\Lambda,\beta$. Hence, as $\vec{\epsilon}=\vec{\tilde{\epsilon}}$,
we may assume $\tilde{\epsilon}^\pm=\epsilon^\pm, \tilde{\n}^-=\n^-,\tilde{c}^-=c^-$, while
$\gamma_i:=\tilde{\nu}^+_i-\nu^+_i$ satisfy
\begin{equation}\label{homogenious solution}
  (\alpha_i,\gamma_j)=(\alpha_j,\gamma_i)\ \mathrm{for\ any}\ 1\leq i,j\leq n.
\end{equation}
In this setup, we have:

\begin{Lem}\label{relating Js}
There exist constants $\{s_i\}_{i=1}^n$ such that
\begin{equation*}
  J^\Lambda_\beta(\tilde{\epsilon}^\pm,\tilde{\n}^\pm,\tilde{c}^\pm)=
  v^{\sum_{i=1}^n s_i(\beta,\omega_i)+\frac{1}{2}\sum_{i=1}^n (\beta,\omega^\vee_i)(\beta-2\Lambda,\gamma_i)}
  \cdot J^\Lambda_\beta(\epsilon^\pm,\n^\pm,c^\pm)
\end{equation*}
for any $\Lambda\in \lambda-\rho-Q_+, \beta\in Q_+$.
\end{Lem}

This essentially follows from~\cite[(3.8, 3.9)]{FFJMM}, but let us provide a complete argument.

\begin{proof}
Since $\tilde{\epsilon}^\pm=\epsilon^\pm, \tilde{\n}^-=\n^-,\tilde{c}^-=c^-$ and
$J^\Lambda_\beta(\bullet,\bullet,\bullet)$ is defined via~(\ref{J-factor}),
it suffices to prove the following equality:
\begin{equation}\label{relation of Whit}
  \theta^\Lambda_\beta(\epsilon^+,\tilde{\n}^+,\tilde{c}^+)=
  \theta^\Lambda_\beta(\epsilon^+,\n^+,c^+)\cdot a_{\Lambda,\beta},
\end{equation}
where
\begin{equation}\label{a-coef}
   a_{\Lambda,\beta}:=
   v^{\frac{1}{2}\sum_{i=1}^n (\beta,\omega^\vee_i)(\beta-2\Lambda,\gamma_i)}
   \cdot \prod_{i=1}^n (\wt{c}^+_iv^{\frac{1}{2}(\gamma_i,\alpha_i)}/c^+_i)^{(\beta,\omega^\vee_i)}.
\end{equation}
Let $\wt{\theta}^\Lambda_\beta$ denote the right-hand side of~(\ref{relation of Whit}).
To prove~(\ref{relation of Whit}), it suffices to verify that
$\sum_{\beta\in Q_+} \wt{\theta}^\Lambda_\beta$ satisfies the defining
conditions~(\ref{Whittaker conditions}) of the Whittaker vector associated with
$(\epsilon^+,\tilde{\n}^+,\tilde{c}^+)$.

First, the equality $a_{\Lambda,0}=1$ implies
  $\wt{\theta}^\Lambda_0=\theta^\Lambda_0(\epsilon^+,\n^+,c^+)=\textbf{1}$.

Second, we note that the equality
  $E_iK_{\nu^+_i}(\theta^\Lambda_{\beta+\alpha_i}(\epsilon^+,\n^+,c^+))=
   c^+_i\cdot \theta^\Lambda_{\beta}(\epsilon^+,\n^+,c^+)$
implies
  $E_iK_{\nu^+_i+\gamma_i}(\wt{\theta}^\Lambda_{\beta+\alpha_i})=
   c^+_iv^{(\gamma_i,\Lambda-\beta-\alpha_i)}\frac{a_{\Lambda,\beta+\alpha_i}}{a_{\Lambda,\beta}}\cdot \wt{\theta}^\Lambda_{\beta}$.
Therefore, it remains to verify
\begin{equation}\label{c relation}
  c^+_iv^{(\gamma_i,\Lambda-\beta-\alpha_i)}\frac{a_{\Lambda,\beta+\alpha_i}}{a_{\Lambda,\beta}}=\wt{c}^+_i.
\end{equation}
Recalling the definition of  $a_{\Lambda,\beta}$ of~(\ref{a-coef}), we find
\begin{multline*}
  \frac{a_{\Lambda,\beta+\alpha_i}}{a_{\Lambda,\beta}}=
  \frac{\wt{c}^+_i}{c^+_i}v^{\frac{1}{2}(\gamma_i,\alpha_i)}\cdot
  v^{\frac{1}{2}\sum_{j=1}^n \{(\beta+\alpha_i,\omega^\vee_j)(\beta-2\Lambda+\alpha_i,\gamma_j)-(\beta,\omega^\vee_j)(\beta-2\Lambda,\gamma_j)\}}=\\
  \frac{\wt{c}^+_i}{c^+_i}v^{\frac{1}{2}(\gamma_i,\alpha_i)}\cdot
  v^{\frac{1}{2}(\beta-2\Lambda+\alpha_i,\gamma_i)+\frac{1}{2}\sum_j(\beta,\omega^\vee_j)(\alpha_i,\gamma_j)}=
  \frac{\wt{c}^+_i}{c^+_i}v^{(\beta-\Lambda+\alpha_i,\gamma_i)},
\end{multline*}
where we used~(\ref{homogenious solution}) to evaluate
  $\sum_j(\beta,\omega^\vee_j)(\alpha_i,\gamma_j)=\sum_j(\beta,\omega^\vee_j)(\alpha_j,\gamma_i)=(\beta,\gamma_i)$.

This implies~(\ref{c relation}), which completes our proof of Lemma~\ref{relating Js}.
\end{proof}

Set
\begin{equation*}
  \fD:=\exp\left(\sum_{i=1}^n\frac{\log(T_{\omega_i})\log(T_{\gamma_i})}{2\sd_i\log(v)}-
  \sum_{i=1}^n s_i\log(T_{\omega_i})\right).
\end{equation*}
This definition is motivated by the following result:

\begin{Lem}\label{explicit D}
$\fD(J^\Lambda(\epsilon^\pm,\n^\pm,c^\pm;\{y_i\}))$ is a non-zero multiple of
$J^\Lambda(\tilde{\epsilon}^\pm,\tilde{\n}^\pm,\tilde{c}^\pm;\{y_i\})$
for any $\Lambda\in \lambda-\rho-Q_+$.
\end{Lem}

\begin{proof}
Evoking formula~(\ref{action on N}), we get
\begin{equation*}
  \fD(\unl{y}^{\beta-\Lambda})=\unl{y}^{\beta-\Lambda}\cdot
  v^{\frac{1}{2}\sum_{i=1}^n (\omega^\vee_i,\beta-\Lambda)(\gamma_i,\beta-\Lambda)+\sum_i s_i(\omega_i,\beta-\Lambda)}.
\end{equation*}
Combining this with Lemma~\ref{relating Js}, the statement reduces to the $\beta$-independence of
\begin{multline*}
  \frac{1}{2}\sum_{i=1}^n (\omega^\vee_i,\beta-\Lambda)(\gamma_i,\beta-\Lambda)+\sum_i s_i(\omega_i,\beta-\Lambda)-
  \frac{1}{2}\sum_{i=1}^n \frac{(\beta,\omega^\vee_i)(\beta-2\Lambda,\gamma_i)}{2}-\sum_{i=1}^n s_i(\beta,\omega_i)=\\
  \sum_{i=1}^n \frac{(\omega^\vee_i,\beta)(\gamma_i,\Lambda)-(\omega^\vee_i,\Lambda)(\gamma_i,\beta-\Lambda)}{2}-\sum_{i} s_i(\omega_i,\Lambda).
\end{multline*}
The latter follows from
  $\sum_{i=1}^n (\omega^\vee_i,\beta)(\gamma_i,\Lambda)=\sum_{i=1}^n (\omega^\vee_i,\Lambda)(\gamma_i,\beta)$,
due to~(\ref{homogenious solution}).
\end{proof}

This completes our proof of Proposition~\ref{Key point}.
\end{proof}

Due to Proposition~\ref{Key point}, $\fD \D_i(\epsilon^\pm,\n^\pm,c^\pm) \fD^{-1}$
and $\D_i(\tilde{\epsilon}^\pm,\tilde{\n}^\pm,\tilde{c}^\pm)$ act diagonally in the basis
  $\{J^\Lambda(\tilde{\epsilon}^\pm,\tilde{\n}^\pm,\tilde{c}^\pm;\{y_i\})|\Lambda\in \lambda-\rho-Q_+\}$
of $W_\lambda$ with the same eigenvalues, hence, they coincide for every $1\leq i\leq n$.
Therefore, conjugation by $\fD$ is a well-defined automorphism of $\CalD_v(H^\ad)$ which maps
$\CT(\epsilon^\pm,\n^\pm,c^\pm)$ to $\CT(\tilde{\epsilon}^\pm,\tilde{\n}^\pm,\tilde{c}^\pm)$.

This completes our proof of Theorem~\ref{Main Conj}.
$\square$


\section{Proof of Theorem~\ref{Lax for type A}}\label{Proof of Lax for A}

The proof of Theorem~\ref{Lax for type A} is similar to the one of Proposition~\ref{bound for A}
given in Appendix~\ref{Proof of bound for A} and of Theorem~\ref{Main Thm A} given
in Appendix~\ref{Proof of Main Theorem A}, but we provide details as the formulas are different.

\noindent
\textbf{Proof of part (a).}

Given a pair of type $A_{n-1}$ Sevostyanov triples $(\epsilon^\pm,\n^\pm,c^\pm)$
and $\vec{k}=(k_n,\ldots,k_1)\in \{-1,0,1\}^n$ satisfying
$k_{i+1}=\frac{\epsilon^+_{i,i+1}-\epsilon^-_{i,i+1}}{2}$ for $1\leq i\leq n-2$,
we will prove that there exist constants $\{r_{ij},r_i\}_{1\leq i\leq j\leq n}$
such that the function $F$ defined in~(\ref{Canonical Form}) satisfies the equality
\begin{equation}\label{conjugation Lax}
  F^{-1}\HH(\epsilon^\pm,\n^\pm,c^\pm)F=\HH_2^{\vec{k}}.
\end{equation}
We will view this as an equality in $\CA_n$, treating $\HH(\epsilon^\pm,\n^\pm,c^\pm)$
as an element of $\CA_n$.

$\bullet$
First, we note that the terms without $\sD_i$'s are the same
(and equal to $\sum_{j=1}^n \sw_j^{-2}$) both in
$F^{-1}\HH(\epsilon^\pm,\n^\pm,c^\pm)F$ and $\HH_2^{\vec{k}}$,
independently of our choice of constants  $\{r_{ij},r_i\}$.

$\bullet$
Second, we will match the terms with $\{\frac{\sD_i}{\sD_{i+1}}\}_{i=1}^{n-1}$
appearing in $F^{-1}\HH(\epsilon^\pm,\n^\pm,c^\pm)F$ and $\HH_2^{\vec{k}}$.
Their equality is equivalent to the following system of equations on $\{r_{ij}\}$:
\begin{equation}\label{system A.4}
  \frac{\m_{ij}-\delta_{j,i}k_i-\delta_{j,i+1}k_{i+1}}{\hbar}=
  \begin{cases}
    r_{ji}-r_{j,i+1}, &\ \mathrm{if}\ 1\leq j< i\\
    r_{ij}-r_{i+1,j}, &\ \mathrm{if}\ i+2\leq j\leq n\\
    2r_{ii}-r_{i,i+1}, &\ \mathrm{if}\ j=i\\
    r_{i,i+1}-2r_{i+1,i+1}, &\ \mathrm{if}\ j=i+1
  \end{cases}
\end{equation}
and the following system of equations on $\{r_i\}$:
\begin{equation}\label{system A.5}
  r_i-r_{i+1}=\hbar(r_{i,i+1}-r_{ii}-r_{i+1,i+1})-\hbar^{-1}\log(b_i)+\sum_{k=1}^n(n-k+1/2)\m_{ik}.
\end{equation}
Pick any $r_{11},r_1$. It suffices to show that~(\ref{system A.4}) admits a solution
since~(\ref{system A.5}) obviously admits a unique solution with a given $r_1$ for
any choice of $r_{ij}$. For a fixed $r_{11}$, the constants $\{r_{ij}\}_{1\leq i\leq j\leq n}$
satisfying the first, third, and fourth cases of~(\ref{system A.4}) are determined uniquely.
It remains to verify that they also satisfy the second case of~(\ref{system A.4}).
We prove this by induction in $j-i\geq 2$.

(a) If $j=i+2$, then
  $r_{i,i+2}-r_{i+1,i+2}=(r_{i,i+1}-2r_{i+1,i+1})+(2r_{i+1,i+1}-r_{i+1,i+2})-(r_{i,i+1}-r_{i,i+2})=
   \hbar^{-1}(\m_{i,i+1}-k_{i+1}+\m_{i+1,i+1}-k_{i+1}-\m_{i+1,i})$.
Hence, it remains to prove
  $(\m_{i,i+1}-\m_{i,i+2})-(\m_{i+1,i}-\m_{i+1,i+1})=2k_{i+1}$.
The left-hand side is equal to
  $\n^-_{i,i+1}-\n^+_{i,i+1}-\n^-_{i+1,i}+\n^+_{i+1,i}=
   b_{i,i+1}(\epsilon^-_{i,i+1}-\epsilon^+_{i,i+1})=\epsilon^+_{i,i+1}-\epsilon^-_{i,i+1}=2k_{i+1}$,
due to the choice of $k_{i+1}$.

(b) If $j>i+2$, then
  $r_{ij}-r_{i+1,j}=(r_{i,j-1}-r_{i+1,j-1})+(r_{i+1,j-1}-r_{i+1,j})-(r_{i,j-1}-r_{ij})=
   \hbar^{-1}(\m_{i,j-1}+\m_{j-1,i+1}-\m_{j-1,i})$.
Hence, it remains to prove $(\m_{i,j-1}-\m_{ij})-(\m_{j-1,i}-\m_{j-1,i+1})=0$.
The left-hand side is equal to $b_{i,j-1}(\epsilon^-_{i,j-1}-\epsilon^+_{i,j-1})=0$
as $b_{i,j-1}=0$.

Thus, we have determined a collection of constants $\{r_{ij},r_i\}_{1\leq i\leq j\leq n}$
satisfying~(\ref{system A.4},~\ref{system A.5}).

$\bullet$
Finally, it remains to verify that for $F$ of~(\ref{Canonical Form}) with
the constants $r_{ij},r_i$ chosen as above the terms with
$\frac{\sD_i}{\sD_j}\ (j>i+1)$ in $F^{-1}\HH(\epsilon^\pm,\n^\pm,c^\pm)F$
and $\HH_2^{\vec{k}}$ do coincide. First, we note that the conditions
  $\epsilon^\pm_{i,i+1}=\ldots=\epsilon^\pm_{j-2,j-1}=\pm 1$
and $k_{i+1}=\ldots=k_{j-1}=1$ are equivalent under our assumption.
Pick $j>i+1$ such that either of these equivalent conditions is satisfied.
Then, the compatibility of the terms with $\frac{\sD_i}{\sD_j}$ is equivalent
to the following equality:
\begin{multline}\label{system A.6}
  \frac{F(\sw_1,\ldots, v\sw_i,\ldots,v^{-1}\sw_j,\ldots,\sw_n)}{F(\sw_1,\ldots,\sw_n)}=\\
  \prod_{k=1}^n \sw_k^{\sum_{s=i}^{j-1}\m_{sk}+\delta_{k,i}+\delta_{k,j}}\cdot \prod_{p=i}^j \sw_p^{-k_p-1}\cdot
  \prod_{s=i}^{j-1}b_s^{-1}\cdot
  v^{i+1-j+\sum_{i\leq a<b\leq j-1} (\n^-_{ab}-\n^+_{ab})+\sum_{k=1}^n \sum_{s=i}^{j-1} \frac{n+1-2k}{2}\m_{sk}}.
\end{multline}
This equality is proved by induction in $j-i$, factoring the left-hand side as
in~(\ref{induction factorization}) and noticing that the $j=i+1$ counterpart
of~(\ref{system A.6}) is just the compatibility of the terms with
$\frac{\sD_i}{\sD_{i+1}}$, established in the previous step.

\medskip
Thus, the function $F$ defined via~(\ref{Canonical Form}) with the constants
$\{r_{ij},r_i\}_{1\leq i\leq j\leq n}$ determined in our second step satisfies
the equality~(\ref{conjugation Lax}). This completes our proof of
Theorem~\ref{Lax for type A}(a).

\medskip
\noindent
\textbf{Proof of part (b).}

Let us write
  $$\sT^v_{\vec{k}}(z)_{11}=
   (-1)^n\sw_1\cdots\sw_nz^s\left(1-\HH^{\vec{k}}_2 z+\HH^{\vec{k}}_3 z^2-\ldots+(-1)^n\HH^{\vec{k}}_{n+1}z^n\right),$$
where $s=\sum_{j=1}^n \frac{k_j-1}{2}$. For $1\leq r\leq n$, let
$\bar{\HH}^{\vec{k}}_{r+1}\in \CT^{\vec{k}}$ be the image of
$\HH^{\vec{k}}_{r+1}$ in $\bar{\CA}_n$. Consider the summands
in $\bar{\HH}^{\vec{k}}_{r+1}$ without $\sD_i$'s and let $\bar{\HH}^{\vec{k};0}_{r+1}$ denote their sum.
Tracing back the definition of $\sT^v_{\vec{k}}(z)$, we get
$\bar{\HH}^{\vec{k};0}_{r+1}=\sigma_r(\{\sw_j^{-2}\})$: the $r$-th elementary
symmetric polynomial of $\{\sw_j^{-2}\}_{j=1}^n$.

Thus, the image of $\bar{\HH}^{\vec{k}}_{r+1}$ under the anti-isomorphism
$\bar{\CA}_n\to \CalD_v(H^\ad_{\ssl_n})$ of Section~\ref{explicit hamiltonians A}
is an element of $\CalD^\leq_v(H^\ad_{\ssl_n})$ whose action on $W_\lambda$
(see Appendix~\ref{Proof of Main Theorem A}) is upper-triangular with the same diagonal matrix coefficients
as in the action of $\D_r\in \CT(\epsilon^\pm,\n^\pm,c^\pm)$. Thus, the argument of
Proposition~\ref{Etingof argument} can be applied to show that the function $F$ of part~(a),
which conjugates $\HH(\epsilon^\pm,\n^\pm,c^\pm)$ into $\bar{\HH}_2^{\vec{k}}$ also conjugates the preimage
of $\D_r$ in $\bar{\CA}_n$ into $\bar{\HH}^{\vec{k}}_{r+1}$ for all $1\leq r\leq n$.
Therefore, conjugation with $F$ is an automorphism of $\bar{\CA}_n$  that maps
$\widetilde{\CT}(\epsilon^\pm,\n^\pm,c^\pm)$ to $\CT^{\vec{k}}$.

\medskip
This completes our proof of Theorem~\ref{Lax for type A}.
$\square$


\end{document}